\documentclass[11pt,letterpaper]{article}
\pdfoutput=1
\usepackage{amsmath,amsopn,amssymb,amsthm}
\usepackage{tikz}
\usepackage{tikz-cd}
\usepackage{geometry}
\usetikzlibrary{matrix,calc,arrows}

\usepackage{eucal}

\usepackage{graphicx}

\usepackage{color}
\definecolor{verydarkblue}{rgb}{0,0,0.4}
\usepackage{hyperref}
\hypersetup{
pdfauthor={David Dumas and Andrew Sanders},
pdftitle={Geometry of compact complex manifolds associated to generalized
  quasi-Fuchsian representations},
colorlinks=true,linkcolor=verydarkblue,
citecolor=verydarkblue,urlcolor=verydarkblue
}

\usepackage{thmtools}
\usepackage{thm-restate}

\usepackage{caption}
\newcommand\blfootnote[1]{%
   \begingroup
  \renewcommand\thefootnote{}\footnote{#1}%
  \addtocounter{footnote}{-1}%
  \endgroup
}

\newcommand{\noproof}{\hfill\qedsymbol}
\renewcommand{\bar}{\overline}

\newcommand{\into}{\hookrightarrow}
\newcommand{\tensor}{\otimes}

\newcommand{\Z}{\mathbb{Z}}
\newcommand{\R}{\mathbb{R}}
\newcommand{\C}{\mathbb{C}}

\newcommand{\card}[1]{|{#1}|}
\newcommand{\bigcard}[1]{\left|{#1}\right|}

\newcommand{\restrict}[2]{\left.{#1}\right|_{#2}}
\newcommand{\suchthat}{ \: : \: }

\renewcommand{\L}{\mathcal{L}}

\DeclareMathOperator{\Pic}{Pic}

\DeclareMathOperator{\re}{Re}
\DeclareMathOperator{\Sym}{Sym}
\DeclareMathOperator{\Isom}{Isom}

\DeclareMathOperator{\Span}{span}
\renewcommand{\Re}{\re}

\DeclareMathOperator{\hdim}{\mathrm{dim_H}}
\DeclareMathOperator{\umdim}{\mathrm{\overline{dim}_{M}}}
\newcommand{\sing}{\mathrm{sing}}
\newcommand{\smooth}{\mathrm{smooth}}
\newcommand{\cohdim}{\mathrm{cd}}
\renewcommand{\O}{\mathcal{O}}
\newcommand{\V}{\mathcal{V}}
\newcommand{\W}{\mathcal{W}}
\newcommand{\A}{\mathcal{A}}
\newcommand{\B}{\mathcal{B}}
\newcommand{\F}{\mathcal{F}}
\newcommand{\CP}{\mathbb{P}_\C}
\newcommand{\RP}{\mathbb{P}_\mathbb{R}}
\newcommand{\T}{\mathcal{T}}

\renewcommand{\H}{\mathbb{H}}

\newcommand{\half}{\frac{1}{2}}

\renewcommand{\Tilde}[1]{\widetilde{#1}}
\renewcommand{\tilde}{\Tilde}
\renewcommand{\sl}{\mathfrak{sl}}
\newcommand{\character}{\chi}
\newcommand{\weylvec}{\delta}

\renewcommand{\phi}{\varphi}
\renewcommand{\rho}{\varrho}
\newcommand{\eps}{\varepsilon}

\renewcommand{\leq}{\leqslant}
\renewcommand{\geq}{\geqslant}

\renewcommand{\setminus}{-}

\newcommand{\h}{\mathfrak{h}}
\newcommand{\s}{\mathfrak{s}}
\renewcommand{\b}{\mathfrak{b}}
\newcommand{\p}{\mathfrak{p}}
\newcommand{\g}{\mathfrak{g}}
\newcommand{\PO}{\mathrm{PO}}
\newcommand{\SO}{\mathrm{SO}}
\newcommand{\PSL}{\mathrm{PSL}}
\newcommand{\SL}{\mathrm{SL}}

\newcommand{\PSU}{\mathrm{PSU}}
\newcommand{\SU}{\mathrm{SU}}

\newcommand{\PSp}{\mathrm{PSp}}
\newcommand{\Hom}{\mathrm{Hom}}
\newcommand{\rank}{\mathrm{rank}}
\newcommand{\pos}{\mathrm{pos}}

\newcommand{\Ad}{\mathrm{Ad}}
\newcommand{\ad}{\mathrm{ad}}

\newcommand{\QH}{\mathcal{QH}}

\renewcommand{\hat}{\widehat}

\numberwithin{equation}{section}
\theoremstyle{plain}
\newtheorem{thm}{Theorem}[section]
\newtheorem{cor}[thm]{Corollary}
\newtheorem{conj}[thm]{Conjecture}
\newtheorem{lem}[thm]{Lemma}
\newtheorem{prop}[thm]{Proposition}
\newtheorem{defn}[thm]{Definition}

\newtheorem{bigthm}{Theorem}

\theoremstyle{definition}

\theoremstyle{definition}

\newenvironment{rmenumerate}{\begin{enumerate}}{\end{enumerate}}

\begin{document}

\title{Geometry of compact complex manifolds associated to generalized
  quasi-Fuchsian representations}

\author{David Dumas and Andrew Sanders}

\date{}

\maketitle

\blfootnote{\emph{Date:} June 25, 2019.  (v2: December 21, 2018.  v1: April 4, 2017)}

\section{Introduction}
This paper is concerned with the following general question: which
aspects of the complex-analytic study of discrete subgroups of $\PSL_2\C$
can be generalized to discrete subgroups of other semisimple complex
Lie groups?

To make this more precise, we recall the classical situation that
motivates our discussion.  A torsion-free cocompact Fuchsian group
$\Gamma < \PSL_2\R$ acts freely, properly discontinuously, and
cocompactly by isometries on the symmetric space
$\PSL_2\R / \textnormal{PSO}(2) \simeq \H^{2}$.  The quotient
$S = \Gamma \backslash \H^2$ is a closed surface of genus $g \geq 2$.
When considering $\Gamma$ as a subgroup of $\PSL_2\C$, it is natural
to consider either its isometric action on the symmetric space
$\H^3 \simeq \PSL_2\C / \PSU(2)$ or its holomorphic action on the
visual boundary $\CP^1 \simeq \PSL_2\C / B_{\PSL_2\C}$.  The latter
action has a limit set $\Lambda = \RP^1$ and a disconnected domain of
discontinuity $\Omega = \H \sqcup \bar{\H}$.  The quotient
$\Gamma \backslash \Omega$ is a compact K\"ahler manifold---more concretely, it
is the union of two complex conjugate Riemann surfaces.

Quasiconformal deformations of such groups $\Gamma$ give
\emph{quasi-Fuchsian groups} in $\PSL_2\C$.  Each such group acts on
$\CP^1$ in topological conjugacy with a Fuchsian group, hence the
limit set $\Lambda$ is a Jordan curve, the domain of discontinuity has
two contractible components, and the quotient manifold is a union of
two Riemann surfaces (which are not necessarily complex conjugates of one
another).

If $G$ is a complex simple Lie group of adjoint type (such as
$\PSL_n\C$, $n \geq 2$), there is a distinguished homomorphism
$\iota_{G}:\PSL_2\C \rightarrow G$ introduced by Kostant \cite{KOS59}
and called the \emph{principal three-dimensional embedding}.  Applying
$\iota_{G}$, a discrete subgroup of $\PSL_2\R$ or $\PSL_2\C$ gives
rise to a discrete subgroup of $G$.  When this construction is applied
to a torsion-free cocompact Fuchsian group $\pi_1S \simeq \Gamma$, the
resulting \emph{$G$-Fuchsian representation} $\pi_1S \to G$ lies in
the \emph{Hitchin component} of the split real form $G_\R<G$.
Representations in the Hitchin component have been extensively studied
in recent years, and the resulting rich geometric theory has shown
them to be a natural higher-rank generalization of Fuchsian groups.
In the same way, we propose to generalize the theory of quasi-Fuchsian
groups by studying complex deformations of these $G$-Fuchsian and
Hitchin representations and the associated holomorphic actions on parabolic
homogeneous spaces of $G$.

The existence of domains of proper discontinuity for such actions
follows from a theory developed by Kapovich-Leeb-Porti \cite{KLP13}
and Guichard-Wienhard \cite{GW12}, which applies in the more general
setting of \emph{Anosov representations} of word-hyperbolic groups in
a semisimple\footnote{For this paper, a semisimple Lie
  group $G$ is a real Lie group with finite center, finitely many
  connected components, semisimple Lie algebra, and with no compact
  factors.  For the reader who prefers algebraic groups, one may also
  work with the $\mathbb{K}$-points of a semisimple linear algebraic
  group defined over $\mathbb{K}$ where $\mathbb{K}=\R$ or
  $\mathbb{K}=\C$ depending on the situation.} Lie group $G$.
In fact, a key component of this theory, as developed in \cite{KLP13}, is the construction of
many distinct cocompact domains of proper discontinuity for the action of a given Anosov representation on a
parabolic homogeneous space $G/P$, each labeled by a certain
combinatorial object---\emph{a Chevalley-Bruhat ideal} in the Weyl group of $G$.

Applying this theory to a $G$-Fuchsian representation of a surface
group, or more generally to an Anosov representation of a
word-hyperbolic group in a complex semisimple group $G$, we consider the
compact, complex quotient manifold $\W = \Gamma \backslash \Omega$
associated to a cocompact domain of discontinuity $\Omega \subset G/P$
arising from the construction of \cite{KLP13}.  Concerning such a
manifold, we ask:
\begin{itemize}
\item What is the homology of $\W$?
\item Does $\W$ admit a K\"{a}hler metric?  Is it a projective algebraic variety?
\item What is the Picard group of $\W$?
\item What are the cohomology groups of holomorphic line bundles on $\W$?
\item Does $\W$ admit non-constant meromorphic functions?
\end{itemize}

In considering these questions, our restriction to complex Lie groups
has the simultaneous advantage that it simplifies topological
questions, and that it paves the way for the rich holomorphic geometry of
generalized flag varieties over $\C$ to assume a prominent role.

Our answers to these questions rest on the fact that if $\W$ were
replaced by one of the complex partial flag varieties $G/P$, classical
Lie theory would give a complete answer: The homology
of $G/P$ admits a preferred basis in terms of \emph{Schubert cells},
which are $B$-orbits on $G/P$ where $B<G$ is a Borel subgroup.  The
classification of line bundles on $G/B$ and their sheaf cohomology is
the content of the Borel-Bott-Weil theorem \cite{BOT57}.

In the remainder of this introduction we survey our results, after
introducing enough terminology to formulate them precisely.

Choosing Cartan and Borel subgroups $H < B < G$, we obtain
the Weyl group $W$ and a natural partial order on it, the
\emph{Chevalley-Bruhat order}.  A subset $I \subset W$ which is
downward-closed for this order is a \emph{Chevalley-Bruhat ideal} (or
briefly, an \emph{ideal}).  An ideal $I$ is \emph{balanced} if
$W=I \sqcup w_{0}I$ where $w_{0}\in W$ is the unique element of
maximal length.

Each element of $W$ corresponds to a \emph{Schubert cell} in the
space $G/B$.  The union of the cells corresponding to elements of an
ideal $I$ gives a closed set $\Phi^I \subset G/B$, the \emph{model
  thickening}.  For a general parabolic subgroup $P < G$, there
is a similar construction of a model thickening $\Phi^I \subset G/P$
if we also assume that $I$ is invariant under right multiplication by
$W_P < W$, the Weyl group of the parabolic.

Now let $\pi$ be a word hyperbolic group.  A homomorphism
$\rho: \pi\rightarrow G$ is $B$-Anosov if there exists a
$\rho$-equivariant continuous map
\[
\xi: \partial_{\infty}\pi\rightarrow G/B
\]
which satisfies certain additional properties that are described in
Section \ref{anosov reps}; roughly speaking, these conditions say
that $\rho$ is ``undistorted'' at a large scale; in particular such a
representation is a discrete, quasi-isometric embedding with finite
kernel.  The map $\xi$ is the \emph{limit curve} associated to the
Anosov representation $\rho$.  (Section \ref{anosov reps} also describes a
more general notion of Anosov representation where $B$ is replaced by
an arbitrary symmetric parabolic subgroup of $G$.)

The work of Kapovich-Leeb-Porti \cite{KLP13} establishes that if
$\rho:\pi\rightarrow G$ is $B$-Anosov, then for every balanced and
right-$W_P$-invariant ideal $I\subset W$ one obtains a
$\Gamma:=\rho(\pi)$-invariant open set $\Omega\subset G/P$ upon which
the action of $\Gamma$ is properly discontinuous and cocompact.  The
set $\Omega$ is defined as the complement
$\Omega = (G/P) \setminus \Lambda$ where the \emph{limit set}
$\Lambda$ is a union over points in the limit curve $\xi$ of
$G$-translates of the model thickening $\Phi^{I}$.

Using the continuous variation of the limit curve as a function of the
Anosov representation (established in \cite{GW12}), and the fact that
the Anosov property is an open condition among representations
(ibid.), elementary arguments establish that if $\rho$ and $\rho'$ are in the same path
component of the space of Anosov representations, then the
corresponding compact quotient manifolds are homotopy equivalent.
In fact, we provide a slightly more sophisticated argument which shows that the resulting
compact quotient manifolds are diffeomorphic.  

We focus on the path component of the space of $B$-Anosov
representations $\pi_{1} S \to G$ that contains the $G$-Fuchsian
representations, which we regard as a complex analogue of the Hitchin
component of $G_\R$.  This component also contains the compositions of
quasi-Fuchsian representations with $\iota_G$, which we call
\emph{$G$-quasi-Fuchsian representations}.  Using the invariance of
topological type described above, when studying topological invariants
of quotient manifolds for representations in this component, it
suffices to consider the $G$-Fuchsian case. Concerning homology, we
find:
\begin{bigthm}
\label{introthm:fuchsian-quotient-homology}
Let $G$ be a complex simple Lie group of adjoint type and let
$\rho: \pi_{1}S\rightarrow G$ be a $G$-Fuchsian representation.  Let
$I \subset W$ be a balanced and right-$W_P$-invariant ideal, where
$P < G$ is parabolic.  Then if $\Omega_{\rho}^{I}\subset G/P$ is the
corresponding cocompact domain of discontinuity, the quotient manifold
$\W_{\rho}^{I} = \rho(\pi_{1}S) \backslash \Omega_{\rho}^{I}$
satisfies
\[
H_{*}(\W_{\rho}^{I}, \Z) \simeq H_{*}(S, \Z)\otimes_{\Z} H_{*}(\Omega_{\rho}^{I}, \Z).
\]
\end{bigthm}
\noindent Furthermore, we calculate the homology of the domain of discontinuity
$\Omega_{\rho}^{I}$:
\begin{bigthm}
\label{introthm:fuchsian-domain-homology}
Let $\rho$ and $I$ be as in the previous theorem, and let
$\Phi^I \subset G/P$ be the associated model thickening.  
Then for any integer $k\geq 0$ the homology of the domain of discontinuity
$\Omega_{\rho}^{I}\subset G/P$ fits in a split exact sequence
\[ 0\rightarrow H^{2n-2-k}(\Phi^{I}, \Z)\rightarrow H_{k}(\Omega_{\rho}^{I}, \Z)\rightarrow H_{k}(\Phi^{I}, \Z)\rightarrow 0,
\]
where $n = \dim_\C G/P$ is the complex dimension of the flag variety.
\end{bigthm}
The correspondence between Weyl group elements, Schubert cells, and
cohomology classes in $G/P$ makes the calculation of the outer terms
in the exact sequence above an entirely combinatorial matter.  More
precisely, we find:
\begin{bigthm}
\label{introthm:fuchsian-domain-properties}
The domains $\Omega_{\rho}^{I} \subset G/P$ as above have the following
properties:
\begin{rmenumerate}
\item The odd homology groups of $\Omega_\rho^I$ vanish.
\item The even cohomology groups of $\Omega_\rho^I$ are free abelian.
\item The rank of $H_{2k}(\Omega_\rho^I)$ is equal to $r_k +
r_{n-1-k}$, where $n = \dim_\C G/P$ and where $r_j$ denotes the
number of elements of $I/W_P$ of length $j$ with respect to the
Chevalley-Bruhat order on $W/W_P$.
\item \label{suspected-pd} For each $k\geq 0$ there is a natural isomorphism 
$ H_k(\Omega_{\rho}^I, \Z) \simeq H^{2n-2-k}(\Omega_{\rho}^I,
\Z) $.
\end{rmenumerate}
\end{bigthm}

Taken together, these results are consistent with the possibility that
$\W_{\rho}^{I}$ is a bundle over the surface $S$ with fiber a compact,
oriented manifold of dimension $(2n-2)$ homotopy equivalent to
$\Omega_{\rho}^{I}$; if so, property \ref{suspected-pd} would follow
from Poincar\'{e} duality for this fiber manifold.  We conjecture a weaker
form of this:
\begin{conj}\label{conjecture fiber}
There exists a compact $(2n-2)$-dimensional Poincar\'e duality space
$F_{\rho}^{I}$ homotopy equivalent to $\Omega_{\rho}^{I}$ and a
continuous fiber bundle
\[
F_{\rho}^{I}\rightarrow \W_{\rho}^{I}\rightarrow S.
\]
\end{conj}

In Section \ref{subsec:sl3} we verify this conjecture in the case
$G = \PSL_3\C$.  We have been informed of work in progress by
Alessandrini-Li \cite{alessandrini-li} and
Alessandrini-Maloni-Wienhard \cite{alessandrini-maloni-wienhard} that
provides other examples in which Conjecture \ref{conjecture fiber}
holds.  Some of these results are announced in \cite{alessandrini:survey}.

These homological results also yield a simple formula for the Euler
characteristic of the quotient manifold:
\begin{cor}
\label{introcor:euler}
The Euler characteristic of $\W_{\rho}^{I}$ satisfies
\[
\chi(\W_{\rho}^{I})=\chi(S)\chi(G/P).
\]
\end{cor}
Note in particular that the Euler characteristic is independent of the
choice of balanced ideal $I\subset W$.  It also follows that an
affirmative answer to Conjecture \ref{conjecture fiber} would
necessarily produce a fiber space $F_{\rho}^{I}$ which satisfies
$\chi(F_{\rho}^{I})=\chi(G/P).$

In Section \ref{complex geometry}, we turn to the complex geometry of
quotients.  Here our work parallels the study of quotient manifolds
associated to complex Schottky groups by a number of authors,
e.g.~\cite{larusson98}, \cite{seade-verjovsky}, and especially
\cite{MO15}.  As in \cite{larusson98} and \cite{seade-verjovsky}, one
of our main techniques is to use a bound on the Hausdorff dimension of
the limit set to apply complex-analytic extension results (e.g.~from
\cite{SHI68}, \cite{HAR74}) and show that the quotient manifold
inherits holomorphic characteristics from $G/P$.  The more recent
results of \cite{MO15} in the Schottky case are probably the most
analogous to our study of Anosov representations, though their results
are stated with hypotheses about extensions of sheaves in place of the 
Hausdorff dimension assumptions we use.

In this complex-geometric part of the paper it is natural for us to
work in the more general setting of a complex Lie group $G$ and
$N = G/H$ a complex homogeneous space (where $H<G$ is a closed
complex Lie subgroup).  We say that a complex manifold $\W$ is a
\emph{uniformized $(G,N)$-manifold with data $(\Omega, \Gamma)$} if:
\begin{itemize}
\item There exists a discrete torsion-free group $\Gamma<G$ and a
$\Gamma$-invariant domain of proper discontinuity $\Omega \subset N$
upon which $\Gamma$ acts freely with compact quotient, and
\item There is a biholomorphism $\W\simeq \Gamma\backslash \Omega.$
\end{itemize}
(Such manifolds are sometimes called \emph{Kleinian} in the
literature.)  Note that a uniformized $(G,N)$-manifold is a special
case of a locally homogeneous geometric structure modeled on $(G,N)$,
and that the manifold $\W_\rho^I$ associated to a
right-$W_P$-invariant ideal $I$ is a uniformized $(G,G/P)$-manifold
with data $(\Omega_\rho^I, \rho(\pi))$.  Following terminology from
the study of convex-cocompact group actions, we call
$\Lambda:=N\setminus \Omega$ the \emph{limit set} of $\W$.  Denote by
$m_k(\Lambda)$ the $k$-dimensional Riemannian Hausdorff measure of $\Lambda$.

\begin{restatable}{bigthm}{embeddednoholo}
\label{introthm:embedded-no-holo}
Let $\W$ be a uniformized $(G,N)$-manifold with data
$(\Omega, \Gamma)$ and limit set $\Lambda$.  Suppose that $N$ is
compact and $1$-connected, and that $m_{2n-2}(\Lambda)=0$ where
$n=\dim_{\C}N$.  If $X$ is a Riemann surface and
$X \not \simeq \CP^1$, then every holomorphic map $\W \rightarrow X$
is constant.  More generally, if $Y$ is a complex manifold whose
universal cover is biholomorphic to an open subset of $\C^k$, then any
holomorphic map $\W \rightarrow Y$ is constant.
\end{restatable}

Using a theorem of Eyssidieux, we also show that under mild conditions
on the complexity of $\pi_1\W$, such a uniformized manifold does not
admit a K\"{a}hler metric:

\begin{restatable}{bigthm}{embeddednotkahler}
\label{introthm:embedded-not-kahler}
Let $\W$ be a uniformized $(G,N)$-manifold with data
$(\Omega, \Gamma)$ and limit set $\Lambda$. Suppose that $N$ is
compact and $1$-connected, and that $m_{2n-2}(\Lambda)=0$ where $n = \dim_\C N$.
If $\pi_1\W$ has an infinite linear group (e.g.~a surface group) as a
quotient, then $\W$ does not admit a K\"ahler metric.  In particular,
$\W$ is not a complex projective variety.
\end{restatable}

In order to apply Theorems \ref{introthm:embedded-no-holo} and
\ref{introthm:embedded-not-kahler} to examples arising from Anosov
representations, it is necessary to verify the hypothesis
concerning the Hausdorff measure of the limit set.  We do this in the
technical Section \ref{size limit set}, which relies on a
combinatorial property of balanced ideals in Weyl groups.  Namely,
except for some low rank aberrations, every balanced ideal
$I\subset W$ contains every element $w\in W$ of length at most $2$.
(We note that a similar result was proved in \cite{ST15} for a similar
purpose, but only for a certain class of Chevalley-Bruhat ideals.)
This translates to a lack of high-dimension cells in $\Phi^I$, which
allows us to show that $m_{2n-2}(\Lambda_\rho^I)$ vanishes in the
$G$-quasi-Fuchsian case.  We conclude: 

\begin{restatable}{bigthm}{anosovquotientnegatives}
\label{introthm:anosov-quotient-negatives}
Let $\rho : \pi_1 S \to G$ be a $G$-quasi-Fuchsian representation,
where $G$ is a complex simple adjoint Lie group that is not isomorphic
to $\PSL_2\C$, and let $P < G$ be a parabolic subgroup.  Let
$I\subset W$ be a balanced and right-$W_P$-invariant ideal in the Weyl
group.  Then the associated compact quotient manifold $\W_\rho^I$ has
the following properties:
\begin{rmenumerate}
\item Any holomorphic map from $\W_{\rho}^{I}$ to a manifold
whose universal cover embeds in $\C^k$ (e.g.~any Riemann surface not
isomorphic to $\CP^1$) is constant.  In particular, $\W$ is
not a holomorphic fiber bundle over such a manifold.

\item The complex manifold $\W_{\rho}^{I}$ does not admit a K\"ahler
metric, and in particular it is not a complex projective variety.
\end{rmenumerate}
\end{restatable}

We remark that results announced in a recent preprint of Pozzetti,
Sambarino and Wienhard \cite{PSW19} would allow this theorem to be
extended to an open neighborhood of the space of $G$-quasi-Fuchsian
representations.  We discuss this further in Section \ref{hdim
  subsec}.
 
While Theorems
\ref{introthm:embedded-no-holo}--\ref{introthm:anosov-quotient-negatives}
are essentially negative results---they rule out the use of certain
techniques in understanding these manifolds---our methods also lead to
positive results concerning the behavior of holomorphic line bundles
on uniformized $(G,G/P)$-manifolds $\W \simeq \Gamma \backslash \Omega$.
Specifically, we find that the behavior of such holomorphic line
bundles is closely related to the representation theory of the
discrete group $\Gamma<G$.

Let $\Pic^\Gamma(G/P)$ denote the space of $\Gamma$-equivariant isomorphism classes of
$\Gamma$-equivariant line bundles on $G/P.$ Then there is a homomorphism
\[
p_{*}^{\Gamma}: \Pic^\Gamma(G/P)\rightarrow \Pic(\W),
\]
the \emph{invariant direct image}.  In favorable circumstances,
the extension theorems of Harvey (see \cite{HAR74} and Theorem
\ref{harvey} below) allow us to show that $p_*^\Gamma$ is an
isomorphism.  In fact, we have:
\begin{restatable}{bigthm}{picard}
\label{introthm:picard}
Let $G$ be a connected semisimple complex Lie group, $P<G$ a parabolic subgroup,
and $\W$ a uniformized $(G,G/P)$-manifold with data $(\Omega, \Gamma)$
and limit set $\Lambda$.  Suppose that $m_{2n-4}(\Lambda)=0$ where
$n=\dim_{\C}G/P$.  Then there is a natural isomorphism
\begin{equation}
\label{picard-exact}
 \Pic(\W) \xrightarrow{\simeq}
\Pic^\Gamma(G/P)
\end{equation}
which is split by the invariant direct image homomorphism
$p_{*}^{\Gamma}: \Pic^\Gamma(G/P)\rightarrow \Pic(\W)$.

Moreover, the kernel of the composition
\begin{equation}\label{Picard kernel}
 \Pic(\W) \xrightarrow{\simeq}
\Pic^\Gamma(G/P)\rightarrow \Pic(G/P)
\end{equation}
is naturally isomorphic to $\textnormal{Hom}(\Gamma, \mathbb{C}^{*}).$
\end{restatable}

As before, after excluding some low dimensional cases, this allows us
to compute the Picard group of manifolds arising from
$G$-quasi-Fuchsian representations.
\begin{restatable}{bigthm}{picardqf}
\label{introthm:quasi-fuchsian-picard}
Let $\rho : \pi_1 S \to G$ be a $G$-quasi-Fuchsian representation,
where $G$ is a complex simple adjoint Lie group that is not of type
$A_{1}, A_{2}, A_{3}$ or $B_{2}$.  Let $P<G$ be a parabolic subgroup,
$I\subset W$ a balanced and right-$W_P$-invariant ideal in the Weyl
group, and $\W_{\rho}^{I}$ the uniformized $(G, G/P)$-manifold associated
to these data.  Then, there is a short exact sequence
\begin{equation}
1\rightarrow \textnormal{Hom}(\pi_{1}S, \mathbb{C}^{*})\rightarrow
\Pic(\W_{\rho}^{I})\rightarrow \Pic(G/P)\rightarrow 1.
\end{equation}
\end{restatable}

Having calculated the Picard group, in Section \ref{cohomology line
  bundles} we turn to calculations of sheaf cohomology groups of line
bundles on $\W$ in the image of the invariant direct image
homomorphism.  Here we restrict to the case $P=B$ to simplify the
discussion, though analogous statements could be derived for any
parabolic subgroup.

Recall that a holomorphic line bundle $\mathcal{L}$ on $G/B$ is
$G$-equivariant if it admits an action of $G$ by bundle automorphisms
covering the $G$-action on $G/B.$  Isomorphism classes of $G$-equivariant
bundles on $G/B$ are in bijection with 1-dimensional representations
$B\rightarrow \C^{*}.$  We say a line bundle $\mathcal{L}$ is effective if it admits
a non-zero holomorphic section.

\begin{bigthm}
\label{introthm:cohomology-line-bundles}
Let $\mathcal{L}$ be a $G$-equivariant effective line bundle on $G/B$ and let $\W$ be a uniformized $(G, G/B)$-manifold with data $(\Omega, \Gamma)$ and  limit set
$\Lambda$ satisfying $m_{2n-2k-2}(\Lambda)=0$ for some $k\geq 1$, where $n=\dim_{\C}G/B$.
Then for all $0 \leq i < k$,
\[
H^{i}(\W, p_{*}^{\Gamma}(\mathcal{L}))\simeq
H^{i}(\Gamma, H^{0}(G/B, \mathcal{L})).
\]
\end{bigthm}
In this theorem, the expression
$H^{i}(\Gamma, H^{0}(G/B, \mathcal{L}))$ denotes the group cohomology
of $\Gamma$ with twisted coefficients.  Since $\mathcal{L}$ is
$G$-equivariant and $\Gamma < G$, the space $H^{0}(G/B, \mathcal{L})$
has the structure of a $\Gamma$-module.

When $i$ exceeds the cohomological dimension $\cohdim(\Gamma)$, the
previous theorem becomes the vanishing result:
\begin{equation}
\label{vanishing result}
H^{i}(\W, p_{*}^{\Gamma}(\mathcal{L}))=0 \text{ for }  \cohdim(\Gamma)<i<k.
\end{equation}

We close the discussion of the complex geometry of quotients
with the following theorem regarding the existence of meromorphic
functions on uniformized $(G, G/B)$-manifolds arising from $G$-quasi-Fuchsian
representations.  Recall that an ample line bundle $\mathcal{L}$ on $G/B$ is
one which gives rise to a projective embedding.
\begin{restatable}{bigthm}{meromorphic}
\label{introthm:meromorphic}
Let $\rho : \pi_1 S \to G$ be a $G$-quasi-Fuchsian representation with
image $\Gamma$, where $G$ is a complex simple adjoint Lie group that
is not of type $A_{1}$, $A_{2}$, $A_{3}$ or $B_{2}$. Let $I$ be a
balanced ideal in the Weyl group $W$ of $G$.  Let $\W_\rho^I$ denote
the uniformized $(G, G/B)$-manifold associated to these data.  For any
ample, $G$-equivariant line bundle $\mathcal{L}$ on $G/B$, the following properties
hold:
\begin{rmenumerate}
\item There exists a $k>0$ such that
\[
H^{0}(\W_{\rho}^{I}, p_{*}^{\Gamma}(\mathcal{L}^{k}))\simeq
H^{0}(\Gamma, H^{0}(G/B, \mathcal{L}^{k})) \neq 0.
\]
\item The manifold $\W_{\rho}^{I}$ admits a non-constant meromorphic function.
\end{rmenumerate}
\end{restatable}
The same techniques show that the transcendence degree over $\C$ of
the field of meromorphic functions on $\W_{\rho}^{I}$ is large
whenever the rank of
$H^{0}(\W_{\rho}^{I}, p_{*}^{\Gamma}(\mathcal{L}^{k}))$ is large;
however, whether or not there are any cases where this transcendence
degree is equal to the complex dimension of $\W_{\rho}^{I},$ so that
$\W_{\rho}^{I}$ is Moishezon, is yet to be seen.  In the analogous
setting of quotient manifolds associated to complex Schottky groups,
these questions are studied in \cite{larusson98} and \cite{MO15}.

\subsection{The role of Hausdorff dimension and $G$-quasi-Fuchsian assumptions} \label{hdim subsec}

Most of our results include a hypothesis concerning the Hausdorff
dimension of the limit set or an assumption that the representation
is $G$-quasi-Fuchsian.  We briefly discuss the prospects for weakening
or removing these hypotheses.

In Theorems \ref{introthm:anosov-quotient-negatives} and
\ref{introthm:quasi-fuchsian-picard}, the Anosov representation is
required to be $G$-quasi-Fuchsian, but this hypothesis is only used to
obtain a bound on the Hausdorff dimension of the limit curve.  In a
recent preprint, Pozzetti, Sambarino, and Wienhard \cite{PSW19} have
announced results that in particular imply continuous variation of the
Hausdorff dimension of the limit curve as a function of the
representation, for a particular sub-class of Anosov representations.
This would allow Theorems \ref{introthm:anosov-quotient-negatives} and
\ref{introthm:quasi-fuchsian-picard} to be immediately extended to a
neighborhood of the $G$-quasi-Fuchsian locus.

Theorem \ref{introthm:meromorphic} is also stated for
$G$-quasi-Fuchsian representations, but here that hypothesis is more
fundamental, as it is used to ensure the existence of vectors in
irreducible representations of $G$ which are invariant under a
principal $\textnormal{PSL}(2,\mathbb{C})$.  It seems likely that a
generic uniformized $(G,N)$-manifold has no non-constant meromorphic
functions.

Theorems \ref{introthm:embedded-not-kahler}, \ref{introthm:picard},
and \ref{introthm:cohomology-line-bundles} require specific upper
bounds on the Hausdorff dimension of the limit set, but we do not know
if the threshold dimensions in those statements are optimal.
Producing examples with limit sets of \emph{large} Hausdorff
dimension, as might be used to show the necessity of the hypothesis,
seems to be out of reach of current methods.  Furthermore, the
delicate nature of extension problems in several complex variables
could make analyzing such examples quite challenging.

\subsection{An illustrative example}

In formulating the main results of this paper, we strive for the
maximum level of generality that our arguments allow.  However, in
reading the proofs it may be helpful to have a concrete
example in mind.  While Section \ref{examples} develops various
aspects of certain examples in detail, here we discuss how all of the
main results apply to one class of examples (which is
also discussed in Sections \ref{sec:sln-constructions}--\ref{subsec:incidence} and in \cite[Section 10.2.2]{GW12}).

Consider a torsion-free cocompact Fuchsian group $\Gamma_0 < \SL_2\R$
and fix $n \geq 2$.  Let $\Gamma$ denote the image of $\Gamma_0$ in
$\SL_n\R$ using the $n$-dimensional irreducible representation of
$\SL_2\R$.  Thus $\Gamma$ acts on $\CP^{n-1}$ preserving a rational
normal curve $X$ of degree $n-1$, and it also preserves the set of
real points $X_\R \subset X$.

Let $\F_{1,n-1}$ denote the $\SL_n\C$-homogeneous manifold consisting
of pairs $(\ell,H)$ where $\ell \subset \C^n$ is a line and
$H \subset \C^n$ is a hyperplane containing $\ell$.  Define
$\Lambda_1 \subset \F_{1,n-1}$ as the set of all pairs $(\ell,H)$
where $[\ell] \in X_\R$, and $\Lambda_{n-1}$ as the set of all
$(\ell,H)$ where $[H] \subset \CP^{n-1}$ is an osculating hyperplane
of $X$ at a point of $X_\R$. Then $\Gamma$ acts properly
discontinuously and cocompactly on
$\Omega_{1,n-1} = \F_{1,n-1} \setminus (\Lambda_1 \cup \Lambda_{n-1})$
by \cite[Theorem 8.6 and Section 10.2.2]{GW12}.  As we explain in
Section \ref{subsec:incidence}, the set $\Omega_{1,n-1}$ is the domain
corrsponding to the ideal in $W(\SL_n\C) \simeq S_n$ consisting of
permutations $x$ with $x(1) < x(n)$.  Letting
$M_{1,n-1} = \Omega_{1,n-1} / \Gamma$, we have:
\begin{itemize}
\item
Theorems~\ref{introthm:fuchsian-domain-homology}--\ref{introthm:fuchsian-domain-properties}
allow the computation of the (free abelian) homology of $\Omega_{1,n-1}$; explicitly,
the Betti numbers are
\[ b_{2k}(\Omega_{1,n-1}) = \begin{cases}
2n-2 & \text{ if } k = n-2 \\
\max \left ( 0,n-1- \left|n-k-2\right| \right ) & \text{ else.}
\end{cases} \]
and $b_{2k-1}(\Omega_{1,n-1}) = 0$.  The details of this calculation can be found in
Theorem \ref{thm:example-incidence-betti}.

\item Theorem~\ref{introthm:fuchsian-quotient-homology} then gives the
homology of $M_{1,n-1}$ itself, and in particular implies that
$\chi(M_{1,n-1}) = (2g-2)\chi(\F_{1, n-1}) = (2g-2)(n^2-n)$ (an application of Corollary \ref{introcor:euler}) where $g$ is the genus of the Riemann surface uniformized by $\Gamma_{0}.$

\item For $n>2$, Theorems \ref{introthm:embedded-no-holo}--\ref{introthm:anosov-quotient-negatives}
show that any holomorphic map from $M_{1,n-1}$ to a manifold
uniformized by a domain in $\C^k$ is constant, and in particular that
$M_{1,n-1}$ is not a holomorphic fiber bundle over a Riemann surface
of positive genus.

\item On the other hand, for $n=3$ we show in Theorem
\ref{thm:sl3-fibering} that the
conclusion of Conjecture \ref{conjecture fiber} holds, i.e.~that
$M_{1,2}$ is a fiber bundle over the surface $\mathbb{H}^2 /
\Gamma_0$.  A related special feature of $n=3$ is that $M_{1,2}$ is a
compactification of a finite quotient of $\SL_2\C / \Gamma_0$.

\item For $n>3$, Theorems
\ref{introthm:picard}--\ref{introthm:quasi-fuchsian-picard} show that
the Picard group of $M_{1,n-1}$ is isomorphic to
$\Hom(\Gamma_0,\C^*) \times \Pic(\F_{1,n-1})$.

\end{itemize}
While Theorems
\ref{introthm:cohomology-line-bundles}--\ref{introthm:meromorphic} do
not apply directly to this example, they can be applied to its
natural lift to a domain of proper discontinuity in the full flag variety
$\SL_n\C/B$ to conclude e.g.~vanishing of cohomology of line bundles
on the quotient manifold in large degree (when $n$ is correspondingly
large) and also that the quotient manifold admits meromorphic
functions (again, for $n$ large).

\subsection{Outline}

In Section \ref{back and prelim} we recall some facts from Lie theory
and introduce the notion of an Anosov representation of a word
hyperbolic group.

In Section \ref{KLP machine} we review the geometry of flag varieties
and discuss the construction of Kapovich-Leeb-Porti which produces
domains of proper discontinuity for Anosov representations.  For the
benefit of readers familiar with \cite{KLP13}, we note that in some
cases our notation and terminology are different from that of the
above-cited paper; this is done to adapt their theory to suit the
specific cases we study (i.e.~complex Lie groups).

In Section \ref{size limit set} we derive estimates for the Hausdorff
dimension of the complement of a domain of discontinuity for an Anosov
representation.  While these estimates are essential in Section
\ref{complex geometry}, their derivation represents a technical
excursion into combinatorial and geometric considerations that are not
used elsewhere in the paper.  (A reader might skip this section on
first reading if seeking an efficient route to the results of Section
\ref{complex geometry}.)

Section \ref{topology} contains the main results concerning the
topology of domains of discontinuity and of quotient manifolds,
including the proofs of Theorems
\ref{introthm:fuchsian-quotient-homology},
\ref{introthm:fuchsian-domain-homology}, and
\ref{introthm:fuchsian-domain-properties}.  The results on homology
and cohomology of flag varieties from Section \ref{subsec:pd-flag} are
used extensively here.

In Section \ref{complex geometry} we turn to the complex geometry of
quotients, proving Theorems \ref{introthm:embedded-no-holo},
\ref{introthm:embedded-not-kahler}, \ref{introthm:picard}, and
\ref{introthm:cohomology-line-bundles} on embedded
$(G,G/P)$-manifolds, and their consequences for $G$-quasi-Fuchsian
representations, Theorems \ref{introthm:anosov-quotient-negatives},
\ref{introthm:quasi-fuchsian-picard}, and \ref{introthm:meromorphic}.
The Borel-Bott-Weil theorem and related notions that are used in our
analysis of holomorphic line bundles and sheaves on uniformized
$(G,G/P)$-manifolds are also recalled here.  This section does not use
the results of Section \ref{topology}, and could be read independently
of that one.

Finally, in Section \ref{examples} we present some explicit examples
of ideals in the Weyl group.  We apply the results of Section
\ref{topology} to these examples, in some cases obtaining explicit formulas for the
Betti numbers of these domains and their quotient manifolds.  We also give an
alternative description of the unique cocompact domain of
discontinuity in $G/B$ for a $G$-Fuchsian representation
$\pi_1S \to G$ in the case $G = \PSL_3\C$, showing that it is a
compactification of a finite quotient of the frame bundle of
$S \times \R$.  Using this description, we verify that Conjecture
\ref{conjecture fiber} holds in this case.

\subsection{Acknowledgments}

The authors thank Benjamin Antieau, Izzet Coskun, William Goldman,
Michael Kapovich, Bridget Tenner, Richard Wentworth, and Anna Wienhard
for helpful conversations relating to this work.  They also thank the
Mathematical Sciences Research Institute in Berkeley, California; a
crucial part of this work was completed during a Spring 2015 semester
program which the authors attended.  The second author thanks the
University of Illinois at Chicago where he spent 2013--2016 as a
postdoctoral fellow; he is very grateful for the freedom and
hospitality provided by the mathematics department.  The authors were
supported in part by the U.S.~National Science Foundation, through
individual grants DMS-0952869 and DMS-1709877 (DD) and DMS-1304006 (AS), and the
second author's participation in the 2015 MSRI program was supported
in part by DMS 1107452, 1107263, 1107367, ``RNMS: GEometric structures
And Representation varieties'' (the GEAR Network).

Finally, the authors thank the anonymous referees for their careful
reading of the paper and many helpful suggestions and corrections; in
particular, these suggestions allowed Theorems
\ref{introthm:embedded-no-holo}--\ref{introthm:anosov-quotient-negatives}
to be significantly strengthened.

\section{Lie groups and Anosov representations}\label{back and prelim}

\subsection{Complex semisimple groups}
\label{complex ss groups}

This section serves as a rapid review of the basic Lie theory which we will use
throughout this paper.  

We use the term \emph{complex semisimple Lie group} to mean a complex Lie
$G$ group with finitely many connected components and semisimple Lie
algebra.  If $G$ is connected and its Lie algebra is
simple, we say $G$ is a \emph{complex simple Lie group}.

Let $G$ be a complex semisimple Lie group with Lie algebra $\g$. A
\emph{Cartan subalgebra} $\h \subset \g$ is a maximal abelian subalgebra such
that the linear map $\ad(X):\g\rightarrow \g$ is
diagonalizable for every $X\in \h$. 

There is a unique Cartan subalgebra up to adjoint action of $G$. The
\emph{rank} of $G$ is the dimension (over $\C$) of any Cartan subalgebra.

Given $\alpha \in \h^{*}\backslash \{0\}$, define
\[
\g_{\alpha}:=\{ X\in \g \suchthat \text{ad}(Y)(X)=\alpha(Y)X \text{
  for all } Y\in \h\}.
\]
An element $\alpha \in \h^{*}$ is a \emph{root} if
$\g_{\alpha}\neq \{0\}$ and $\g_\alpha$ is the associated \emph{root
  space}.  The set of all roots is denoted by $\Sigma$.  It is
possible to partition the set of roots as
$\Sigma = \Sigma^+ \sqcup \Sigma^-$ so that $\Sigma^- = -\Sigma^+$ and
so that the sets $\Sigma^\pm$ are separated by a hyperplane in the
$\R$-span of $\Sigma$.  Fix such a partition.  Elements of $\Sigma^+$
are \emph{positive roots}, and those of $\Sigma^-$ are \emph{negative
  roots}.  A positive root $\alpha$ is \emph{simple} if it cannot be
written as a sum of two positive roots.  The set of simple roots is
denoted by $\Delta\subset \Sigma^{+}$.

These data define the \emph{standard Borel
  subalgebra}
\[
\b:= \h \oplus \bigoplus_{\alpha \in \Sigma^{+}} \g_{\alpha},
\]
which is a maximal solvable subalgebra of $\g$.

Next, let $\Theta \subset \Delta$ be a subset of the simple
roots.  Let $\Sigma_\Theta^- \subset \Sigma^-$ denote the set of
negative roots that can be expressed as an integer linear combination
of elements of $\Delta \setminus \Theta$ with non-positive
coefficients.  The subset $\Theta$ defines a
\emph{standard parabolic subalgebra}  via $\p_{\Theta} = \left ( \bigoplus_{\alpha
  \in \Sigma_\Theta^-} \g_\alpha \right ) \oplus \b$.

We define the corresponding Lie subgroups by
\begin{equation*}
\begin{split}
H&=C_{G}(\h), \\ 
B&=N_{G}(\b), \\
P_{\Theta}&= N_{G}(\p_{\Theta}).
\end{split}
\end{equation*}
It is a standard fact that $\mathfrak{h}, \b,$ and $\mathfrak{p}_{\Theta}$ are the Lie
algebras of the above defined Lie groups.

The subgroup $H<G$ is called a \emph{Cartan subgroup}, and is a
maximal torus\footnote{A maximal torus $H<G$ is an abelian
  subgroup which is isomorphic to $(\C^{*})^{\rank(G)}.$} in $G$.
A subgroup $P<G$ is \emph{parabolic} if it conjugate to $P_{\Theta}$ for some
subset of simple roots $\Theta\subset \Delta$. We call
$P_{\Theta}$ a \emph{standard parabolic subgroup}.

Two parabolic subgroups $P^{+}, P^{-}$ are \emph{opposite} if
$P^{+}\cap P^{-}=L$ is a maximal reductive subgroup of both $P^{+}$
and $P^{-}$: that is, the subgroup $L$ is a common \emph{Levi factor} of $P^{+}$ and
$P^{-}.$

Next, choose a maximal compact subgroup $K<G$ with Lie algebra $\mathfrak{k}$ such that
$\mathfrak{k}\cap \h$ is a maximal compact torus inside of $\mathfrak{k}.$  Let
$\g=\mathfrak{k}\oplus \mathfrak{m}$ be the associated Cartan decomposition of $\mathfrak{g}.$
The (real) subspace $\mathfrak{a}:=\h \cap \mathfrak{m}$ is a maximal abelian subspace of
$\mathfrak{m}$ consisting of semisimple elements, called a \emph{Cartan subspace}.
Furthermore, if $\alpha\in \Sigma$ is any root, the restriction of $\alpha$ to $\mathfrak{a}$
is real valued, and this restriction (a \emph{restricted root})
uniquely determines $\alpha$.

A \emph{positive Weyl chamber} $\mathfrak{a}^{+}\subset \mathfrak{a}$ is defined by
$X\in \mathfrak{a}^{+}$ if and only if $\alpha(X)>0$ for all $\alpha\in \Delta.$  
Let $A\subset G$ be defined by $\textnormal{exp}(\overline{\mathfrak{a}}^{+}).$
This gives rise to a Cartan decomposition $G=KAK$ on the group level.

If $g=k_{1}\exp(X)k_{2},$ then the element $X\in \overline{\mathfrak{a}}^{+}$ is uniquely
determined which defines a continuous, proper map

\[
\mu: G\rightarrow \overline{\mathfrak{a}}^{+} 
\]
called the \emph{Cartan projection}.

The \emph{Weyl group} $W$ associated to these data is the group
$N_{K}(\mathfrak{a})/Z_{K}(\mathfrak{a})$ which acts on the Cartan
subspace $\mathfrak{a}$ via the adjoint action, and thus also on the
space $\Hom_\R(\mathfrak{a}, \R)$ containing the simple restricted
roots.  The \emph{restricted simple roots} are the restrictions of the
simple roots $\Delta$ to the Cartan subspace $\mathfrak{a}.$ The Weyl
group is a Coxeter group which is generated by reflections in the
kernels of the restricted simple roots (the \emph{simple root
  hyperplanes}).  The action of $W$ on $\Hom_\R(\mathfrak{a}, \R)$
permutes the restricted roots, and through the bijection of this set
with $\Sigma$, we can regard $W$ as a group of permutations of
$\Sigma$.  Finally, by construction there is an inclusion
$N_{K}(\mathfrak{a})\rightarrow N_{G}(H)$ which induces an isomorphism
$W\simeq N_{G}(H)/H.$ Note that in this isomorphism, the left hand
side acts on restricted roots, while the right hand side acts on
roots.  Since these determine one another, we will freely use this
isomorphism without further comment when it is clear from the context.

As a Coxeter group, $W$ has a unique element of maximal length $w_{0}$
which has order two.  The
\emph{opposite involution} acting on the set of roots is defined
by $\iota(\alpha)=-w_{0}(\alpha).$

A subset $\Theta\subset \Delta$ is symmetric if
$\iota(\Theta)=\Theta.$ A parabolic subgroup is symmetric if and only
if it is conjugate to any (hence all) of its opposite parabolic
subgroups.  This is equivalent to $P$ being conjugate to $P_{\Theta}$
for $\Theta\subset \Delta$ symmetric.  We remark that if all simple
factors of $G$ are of type $A_{1}$, $B_{n\geq 2}$, $D_{2k\geq 4}$,
$E_{7}$, $E_8$, $F_{4}$ or $G_{2},$ then $\iota$ is the identity and
all parabolic subgroups are symmetric.

If $\g$ is a complex semisimple Lie algebra, then a \emph{split real
  form} $\g_{\mathbb{R}}$ is a real form of $\g$ such that the
restriction of the Killing form to $\g_{\R}$ has maximal index.  There
is a single equivalence class of split real forms under the adjoint
$G$-action on $\g$; choosing a representative of this class, we refer
to \emph{the} split real form $\g_\R \subset \g$.  When $G$ is
connected, the connected Lie subgroup $G_{\R}<G$ with lie algebra
$\g_{\R}$ is the split real form of $G.$

\subsection{Principal three-dimensional subgroups}
For more information on the objects in this section, see the
discussion in \cite{ST15} and the original paper of Kostant
\cite{KOS59}.

Let $\mathfrak{g}$ be a complex simple Lie algebra of rank
$\ell$ and Borel subalgebra $\b<\g$. Choose a nilpotent element $e_1\in \b$ which has
$\ell$-dimensional centralizer (a \emph{regular nilpotent}). By the
Jacobson-Morozov theorem (\cite{JAC51} \cite{MOR42}), there exist
elements $x, f_{1} \in \g$ such that the triple $\{f_{1},x, e_{1}\}$
spans a subalgebra $\s$ isomorphic to $\sl_{2}\C$, with
$f_1$, $x$, and $e_1$ respectively corresponding to 
$\left(\begin{smallmatrix}0 & 0\\1 & 0
\end{smallmatrix}\right)$, $\left(\begin{smallmatrix}1 & 0\\0 & -1
\end{smallmatrix}\right)$, and $\left(\begin{smallmatrix}0 & 1\\0 & 0
\end{smallmatrix}\right)$.
Such a subalgebra $\s$ is called a \emph{principal three-dimensional
  subalgebra}.  There is a single conjugacy class of principal
three-dimensional subalgebras under the adjoint $G$-action on $\g$,
corresponding to the single conjugacy class of regular nilpotents.
Abusing terminology, we use this uniqueness to refer to \emph{the}
principal three-dimensional subalgebra of $\g$.

If $G\simeq \mathrm{Aut}_{0}(\g)$ is the adjoint complex simple group
associated to $\g$, and $\s \subset \g$ is the principal
three-dimensional subalgebra, then associated to the isomorphism
$\sl_2\C \simeq \s$ described above is a unique injective homomorphism
\[
\iota_{G}: \PSL_{2}\C\rightarrow G
\]
Moreover, the restriction of $\iota_{G}$ to $\PSL_{2}\R$ takes
values in the split real form of $G$. The image $\mathfrak{S}$ of this
homomorphism is the \emph{principal three-dimensional subgroup} of
$G$.

Given a maximal torus and Borel subgroup
$H_{\mathfrak{S}} < B_{\mathfrak{S}}<\mathfrak{S}$ in the principal
three-dimensional subgroup, there exists a unique maximal torus and
Borel subgroup $H < B < G$ in $G$ such that $H_{\mathfrak{S}} < H$ and
$B_{\mathfrak{S}} < B$.  When considering the principal
three-dimensional subgroup, we always assume that the maximal tori and
Borel subgroups for $\mathfrak{S}$ and $G$ have been chosen in this
compatible way.  We further assume that the isomorphism $\iota_G$ is
chosen so that $H_{\mathfrak{S}}$ and $B_{\mathfrak{S}}$ correspond,
respectively, to the set of diagonal and upper-triangular matrices in
$\PSL_2\C$.  Then, identifying the quotient of $\PSL_2\C$ by its
upper-triangular subgroup with $\CP^1$ we obtain an equivariant
holomorphic embedding
\[
f_{G}: \CP^{1} \simeq \mathfrak{S} / B_{\mathfrak{S}} \rightarrow G/B
\]
called the \emph{principal rational curve}, following \cite{ST15}.  The
principal rational curve can also be characterized as the unique
closed orbit of the action of $\mathfrak{S}$ on $G/B$.

Since $B$ is self normalizing, the space $G/B$ is equivariantly
isomorphic to the space of Borel subgroups of $G$ where $G$ acts on
the latter space by conjugation.  Using this isomorphism, two points
$p,p'\in G/B$ are defined to be \emph{opposite} if the corresponding
Borel subgroups are opposite.  More generally, a pair of points
$p \in G/P^+$ and $p' \in G/P^-$ corresponds to a pair of parabolic
subgroups conjugate, respectively, to $P^+$ and $P^-$; we say in this
case that $p,p'$ are opposite if the corresponding parabolic subgroups
are opposite.

We will need the following essential property of the principal
rational curve.

\begin{prop}\label{prop: opposite}
Given distinct points $z, z'\in \CP^1,$ the images $f_{G}(z),
f_{G}(z')\in G/B$ are opposite.
\end{prop}

\begin{proof}
The statement is invariant under conjugation of $\mathfrak{S}$ by
elements of $G$, hence we can fix a convenient choice of principal
three-dimensional subalgebra $\mathfrak{s} = \Span(e_0,x_0,f_0)$ as in
\cite[Proposition 1.1]{ST15} so that $\alpha(x_0) = 2$ for all
$\alpha \in \Delta$.  Recall that in terms of the derivative of
$\iota_G$, the element $x_0$ is given by $(\iota_G)_*\left(\begin{smallmatrix}
1 & 0 \\
0 & -1
\end{smallmatrix}\right)$.

Let $H_0 \subset \PSL_2\C$ denote the diagonal subgroup.  
Identify the Weyl group of $\PSL_2\C$ with $\Z/2$, with the non-trivial element
represented by $u=
\left(\begin{smallmatrix}
0 & 1 \\
-1 & 0
\end{smallmatrix}\right)$.
Since $u$ normalizes $H_0$, the image $\iota_G(u)$ normalizes 
$\iota_G(H_0)$.  Since $H$ is the unique maximal torus containing
$\iota_G(H_0)$, it follows that $\iota_G(u) \in N_G(H)$.  Thus $\iota_G$
induces a homomorphism $W(\PSL_2\C) = N_{\PSL_2\C}(H_0) / H_0 \to W(G) =
N_G(H)/H$.

Next we claim that the image of $u$ under this map is the longest
element $w_0 \in W = W(G)$.  This element is uniquely characterized by
the condition that it maps every simple root to a negative root.  Note
that $\Ad(\iota_G(u))(x_0) = -x_0$.  Thus for each $\alpha \in \Delta$
we have
\[
\iota_G(u)(\alpha)(x_0) = \alpha(\Ad(\iota_G(u))(x_0)) = \alpha(-x_0)
= -2
\]
It follows that when expressing $\iota_G(u)(\alpha)$ as a linear
combination of the simple roots, there is exactly one non-zero
coefficient, which is equal to $-1$.  Hence $\iota_G(u)(\alpha)$ is a
negative simple root for all $\alpha \in \Delta$, and we conclude
$\iota_G(u)$ represents $w_0$.

Since the longest element of $W$ maps $B$ to an opposite Borel,
it follows from $\iota_G$-equivariance of the map $f_G$ that if
$z_0 \in \CP^1$ is the unique point such that $f_G(z_0) = eB \in G/B$,
then $f_G(z_0)$ and $f_G(uz_0) = \iota_G(u) f_G(z_0)$ are opposite.  Finally, since
$\PSL_2\C$ acts transitively on pairs of distinct points in $\CP^1$,
equivariance of $f_G$ implies that the same condition holds for all
pairs of distinct points.
\end{proof}

\subsection{Anosov representations}\label{anosov reps}

In this subsection we recall the definition of an Anosov
representation and some related notions that are used extensively in
the sequel.  We follow the exposition of \cite{GGKW15} quite
closely.  Additional background on Anosov representations can be found
in \cite{LAB06}, \cite{GW12}, \cite{KLP13}, and \cite{KLP14.1}.  The principal
distinction in our treatment is that we work exclusively with complex semisimple
Lie groups.

Let $d_{\pi}$ denote the word metric on the Cayley graph of a finitely
generated group $\pi$ corresponding to some finite generating set.
Recall that $\pi$ is \emph{word hyperbolic} if this Cayley graph is a
Gromov hyperbolic metric space. Write $|\cdot|$ for the associated
word length function, i.e.~$|\gamma| = d_\pi(e,\gamma)$.  The
\emph{translation length} of $\gamma\in \pi$ is defined by
\[
\ell_{\pi}(\gamma):=\inf_{\beta\in \pi} \: \lvert \beta \gamma \beta^{-1}\rvert.
\]

We denote by $\partial_\infty \pi$ the Gromov boundary of the Cayley
graph of $\pi$; points in $\partial_\infty \pi$ are equivalence
classes of geodesic rays in the Cayley graph.  The $\pi$-action by
left translation on its Cayley graph extends to a continuous action on
$\partial_{\infty}\pi$.  Under this action, each infinite order
element $\gamma\in \pi$ has a unique attracting fixed point
$\gamma^{+} \in \partial_\infty \pi$ and a unique repelling fixed
point $\gamma^{-} \in \partial_\infty \pi$.

Let $(P^{+}, P^{-})$ be a pair of opposite parabolic subgroups of a complex semisimple
group $G$.
Let $\rho:\pi\rightarrow G$ be a homomorphism and suppose there
exists a pair of continuous, $\rho$-equivariant maps
\[
\xi^{\pm}: \partial_{\infty}\pi \rightarrow G/P^{\pm}.
\]
The pair $(\xi^{+}, \xi^{-})$ is \emph{dynamics preserving} for
$\rho$ if for each infinite order element $\gamma\in \pi$ 
the point $\xi^{+}(\gamma^{+})$ (resp. $\xi^{-}(\gamma^{+}))$ is an
attracting fixed point for the action of $\rho(\gamma)$ on $G/P^{+}$
(resp. $G/P^{-}).$ Here, a fixed point $x\in G/P$ is attracting for
$g\in G$ if the linear map given by the differential
\[
dg_{x}: T_{x}G/P\rightarrow T_{x}G/P
\]
has spectral radius strictly less than one.

We now come to the definition of an \emph{Anosov} representation.

\begin{defn}\label{Anosov definition}
Let $(P^+,P^-)$ be a pair of opposite parabolic subgroups of $G$, and
$\rho:\pi\rightarrow G$ a homomorphism.  Then $\rho$ is
\emph{$(P^+,P^-)$-Anosov} if there exists a pair of
$\rho$-equivariant, continuous maps
\[
\xi^{\pm}: \partial_{\infty}\pi\rightarrow G/{P^{\pm}}
\]
such that the following conditions hold:
\begin{rmenumerate}
\item \label{transverse} For all distinct pairs $t,t'
\in \partial_{\infty}\pi,$ the points $\xi^{+}(t) \in G/P^+$ and
$\xi^-(t') \in G/P^-$ are opposite.
\item \label{dynamics preserving} The pair of maps $(\xi^{+},
\xi^{-})$ is dynamics preserving for $\rho$.
\item \label{divergence} Realize $(P^+,P^-)$ as a pair of standard
opposite parabolics $(P_\Theta,P_\Theta^-)$ for suitable choices of
Cartan subalgebra $\h$, system of positive roots $\Sigma^+$, and subset
$\Theta \subset \Delta$.  Then for each $\alpha \in \Theta$, any
sequence $\{\gamma_{n}\}_{n=1}^{\infty} \subset \pi$ with divergent word length
\[
\displaystyle\limsup_{n\rightarrow \infty}\ell_{\pi}(\gamma_n)\rightarrow \infty,
\]
satisfies the following \emph{$\alpha$-divergence condition} of the Cartan
projections of its $\rho$-images:
\[
\displaystyle\limsup_{n\rightarrow \infty} \langle \alpha, \mu(\rho(\gamma_{n}))\rangle= \infty.
\]
Here $\langle \cdot , \cdot \rangle$ denotes the evaluation pairing
$\mathfrak{a}^* \times \mathfrak{a} \to \R$, we view the root $\alpha$ as an
element of $\mathfrak{a}^{*}$ by restriction, and $\mu$ denotes the Cartan projection.
\end{rmenumerate}
\end{defn}

Due to the work of \cite{GW12}, \cite{GGKW15} and \cite{KLP13},
\cite{KLP14.1}, there are now many equivalent definitions of Anosov
representations.  The definition given above (taken from
\cite[Theorem~1.3]{GGKW15}) is the most economical one for our
purposes.  However, Condition~\ref{divergence} from this definition is
evidently quite technical, and the details of this part of the
definition will not be used at all in what follows.  Most readers can
therefore proceed without careful study of this last condition.  In particular, it is
shown in \cite{GW12} that if $G$ is a real algebraic group and the representation is
Zariski dense, then condition~\ref{divergence} is a consequence of 
conditions~\ref{transverse} and \ref{dynamics preserving}.

The maps
$\xi^{\pm}: \partial_{\infty}\pi\rightarrow G/P^{\pm}$ in the
definition above are called the \emph{limit curves} of the Anosov
representation.

If $P$ is a symmetric parabolic subgroup, we can apply the definition
above with $(P^+,P^-)=(P,gPg^{-1})$ (for suitable $g\in G$) 
as the pair of opposite parabolic
subgroups.  In this case both spaces $G/P^\pm$ are canonically and
$G$-equivariantly identified with $G/P$, and the limit maps $\xi^\pm$
are related to one another by this identification.  We therefore
consider such a representation to have a single limit curve
\[
\xi: \partial_{\infty}\pi\rightarrow G/P,\]
and in this situation we simply say that $\rho$ is \emph{$P$-Anosov}.

The following property of Anosov representations follows quickly
from the definitions.

\begin{prop}\label{anosov exchange}
Let $P,Q<G$ be symmetric parabolic subgroups such that $P<Q.$  
If $\rho: \pi\rightarrow G$ is $P$-Anosov, then $\rho$ is also $Q$-Anosov.
Furthermore, if $\xi: \partial_{\infty}\pi\rightarrow G/P$ is the limit curve
for $\rho$ as a $P$-Anosov representation, then $p\circ \xi: 
\partial_{\infty}\pi\rightarrow G/Q$ is the limit curve for $\rho$ as
a $Q$-Anosov representation where $p: G/P\rightarrow G/Q$ is the 
natural projection.
\end{prop}

There is also no loss of generality in considering only $P$-Anosov
representations for symmetric parabolics $P$ rather than the \emph{a
  priori} more general classes of $(P^{+}, P^{-})$-Anosov
representations:
\begin{prop}[{\cite{GW12}}]\label{sym parabolic}
Let $\rho: \pi\rightarrow G$ be $(P^{+}, P^{-})$-Anosov.  Then there
exists a symmetric parabolic subgroup $P<G$ such that $\rho$ is
$P$-Anosov. \noproof
\end{prop}

Furthermore, the following theorem of Guichard-Wienhard establishes
some basic properties of Anosov representations:

\begin{thm}[{\cite{GW12}}]\label{anosov prop}
Let $\rho:\pi \rightarrow G$ be $(P^{+},P^{-})$-Anosov.  Then the following properties are satisfied:
\begin{rmenumerate}
\item For every $\gamma\in \pi,$ the holonomy $\rho(\gamma)$ is conjugate to an element of $L=P^{+}\cap P^{-}.$ 
\item The representation $\rho$ is discrete, has finite kernel, and is
a quasi-isometric embedding.
\item The set $\A$ of all $(P^+,P^-)$-Anosov representations of $\pi$
is an open set in the representation variety $\Hom(\pi, G)$.
\item The map taking a $(P^+,P^-)$-Anosov representation to either of its limit curves,
\begin{equation*}
\begin{split}
\A & \rightarrow C^{0}(\partial_{\infty}\pi, G/P^{\pm}) \\
\rho &\mapsto  \xi_{\rho}^{\pm}
\end{split}
\end{equation*}
is continuous, where $C^{0}(\partial_{\infty}\pi, G/P^{\pm})$ has the uniform topology. \noproof
\end{rmenumerate}
\end{thm}

In the case that $G_{\R}<G$ is a real form of a complex semisimple group $G$ such that
$G_{\R}$ has real rank equal to one, it was also shown in \cite{GW12}
that the Anosov property reduces to the well-known class of
\emph{convex-cocompact} subgroups of $G_{\R}$:
\begin{thm}[{\cite{GW12}}] \label{rank one Anosov}
Suppose $G_{\R}<G$ has real rank one.  Then a representation $\rho:\pi \rightarrow G_{\R}<G$ is Anosov if and only if $\rho$
has finite kernel and its image is convex-cocompact. \noproof
\end{thm}

In particular, if $\Gamma$ is a uniform lattice in a real rank one Lie
group $G_{\R}<G$ (e.g.~a lattice in $\SO(n,1)<\SO(n+1, \C)$ or $\SU(n,1)<\SL(n+1,\C)$), then
the inclusion $\Gamma \into G$ is an Anosov representation.

\subsection{Fuchsian and Hitchin representations}

Let $S$ be a closed, oriented surface of genus at least two.  For a
Lie group $G$ we define the \emph{character space} of $S$ in $G$ to be the
topological space
$$ \bar{\character}(S,G) = \Hom(\pi_{1}S,G) / G$$
where $G$ acts on $\Hom(\pi_{1}S,G)$ by conjugation\footnote{In this
  paper we do not use the closely related notion of a character
  \emph{variety}, and so we avoid discussion of the subtleties
  necessary to define such algebraic or semialgebraic sets.}.

Identify the hyperbolic plane $\mathbb{H}^2$ with the upper half plane
$\H \subset \C$ (which is oriented by its complex structure).  Then
$\PSL_2\R$ is identified with the group of orientation preserving
isometries of $\H^2$.  A \emph{Fuchsian representation} is an
injective homomorphism
\[
\eta: \pi_{1}S\rightarrow \PSL_{2}\R
\]
with discrete image such that the associated homotopy equivalence $S
\simeq \rho(\pi_{1}S)\backslash\mathbb{H}^{2}$ is
orientation preserving.

Let $G$ be a complex simple Lie group of adjoint type and fix a
principal three-dimensional subgroup (with embedding
$\iota_{G}:\PSL_{2}\C\rightarrow G$).  Let $G_{\R}<G$ be a split real
form which contains $\iota_{G}(\PSL_2\R)$.  A representation
$\rho: \pi_{1}S\rightarrow G$ is \emph{$G_\R$-Fuchsian} if there
exists a Fuchsian representation $\eta$ such that $\rho$ is conjugate
to $\iota_{G}\circ \eta.$ The set of conjugacy classes of $G_\R$-Fuchsian
representations forms a connected subset of $\bar{\character}(S,G_\R)$
that is in natural bijection with the Teichm\"uller space of
hyperbolic structures on $S$.

A \emph{$G_\R$-Hitchin} representation is a homomorphism
$\rho: \pi_{1}S\rightarrow G_{\R}$ whose conjugacy class lies in the
same path component of $\bar{\character}(S,G_{\R})$ as the
$G_{\R}$-Fuchsian representations.  Let
$\mathcal{H}(S, G_\R) \subset \bar{\character}(S, G_\R)$ denote the
set of conjugacy classes of $G_\R$-Hitchin representations.

The following theorem organizes the key properties of Hitchin representations
which we will use.
\begin{thm}
\label{thm:hitchin-space-smooth}
When considered as a subset of the $G_\R$-character space, the set
of $G_{\R}$-Hitchin representations
\[
\mathcal{H}(S, G_{\R})\subset \bar{\character}(S, G_\R)
\]
is a smooth manifold diffeomorphic to a Euclidean space of real
dimension $-\chi(S)\dim_{\R}(G_{\R})$ and is a connected component
of $\bar{\character}(S, G_{\R})$. Moreover, every
$G_{\R}$-Hitchin representation is Anosov with respect to a Borel subgroup 
$B<G$ where $G$ is the complexification of $G_{\R}.$ 

Furthermore, when considered as a representation in the complex group
$G$, each Hitchin representation is a smooth point of the complex
affine variety $\Hom(\pi_1S,G)$.
\end{thm} 

\begin{proof}
The statement that
$\mathcal{H}(S, G_{\R})\subset \bar{\character}(S, G_\R)$ is a smooth
manifold of the given dimension was proved by Hitchin in \cite{HIT92}.
When $G_{\R}=\PSL_n\R$, Labourie \cite{LAB06} established that Hitchin
representations are $B$-Anosov.  For general split groups, the
analogous statement was proved by Fock-Goncharov in \cite[Theorem
1.15]{FG06}; also see \cite[Theorem 6.2]{GW12} for further discussion.

By \cite[pg.~204]{GOL84}, a representation
$\rho\in \Hom(\pi_{1}S, G)$ lies in the smooth locus if and only if
it has discrete centralizer.  Hitchin representations are irreducible
(i.e.~not conjugate into a proper parabolic subgroup of $G$, see
\cite[Lemma~10.1]{LAB06}), which implies that their centralizers are finite
extensions of the center of $G$ (\cite[Proposition~15]{Sikora}) and
thus discrete.
\end{proof}

\subsection{Quasi-Fuchsian and quasi-Hitchin representations}\label{quasi-hitchin}
As before, let $S$ be a closed, oriented surface of genus at least
two. A representation $\eta: \pi_{1}S\rightarrow \PSL_2\C$ is
\emph{quasi-Fuchsian} if it is obtained from a Fuchsian representation
by a quasiconformal deformation.  This is equivalent to being a
convex-cocompact representation, or to the existence of a continuous,
equivariant, injective map
$\xi_{\eta}: \partial_{\infty}\pi_{1}S\rightarrow \CP^{1}.$ A
quasi-Fuchsian representation is \emph{Fuchsian} if it is
conjugate to a representation with values in $\PSL_2\R< \PSL_2\C.$ The
space of all quasi-Fuchsian representations up to conjugacy will be
denoted
\[
\mathcal{QF}(S)\subset \bar{\character}(S,\PSL_2\C)
\]
and the set of Fuchsian representations by 
\[
\mathcal{F}(S)\subset \bar{\character}(S,\PSL_2\R).
\]
Now, let $G$ be a complex simple Lie group of adjoint type.  A
\emph{$G$-quasi-Fuchsian representation} $\rho: \pi_{1}S\rightarrow G$ is a
representation which admits a factorization
$\rho = \iota_G \circ \eta$ where $\eta$ is a quasi-Fuchsian
representation.  Similarly, a subgroup $\Gamma<G$ is $G$-quasi-Fuchsian if it is
the image of a $G$-quasi-Fuchsian representation.

The chosen principal three-dimensional embedding
$\iota_{G}: \PSL_2 \C \rightarrow G$ induces a commutative diagram
\begin{equation}
\begin{tikzcd}
\mathcal{F}(S) \arrow{d} \arrow{r}{\iota_{G} \circ}
& \bar{\character}(S, G_{\mathbb{R}}) \arrow{d} \\
\mathcal{QF}(S) \arrow{r}{\iota_{G} \circ}
& \bar{\character}(S, G).
\end{tikzcd}
\end{equation}
Moreover, these maps are independent of the choice of three-dimensional subalgebra and split real form.

We now show that a $G$-quasi-Fuchsian representation is Anosov and identify the limit curve.

\begin{prop}\label{quasifuchsian Anosov}
Every $G$-quasi-Fuchsian representation $\rho$ is $P$-Anosov where
$P<G$ is any symmetric parabolic subgroup.  Furthermore, if $\rho=\iota_{G}\circ \eta$
where $\eta: \pi_{1}S\rightarrow \PSL_2 \C$
is quasi-Fuchsian, and if $\eta$ has limit curve $\xi: \partial_{\infty}\pi_{1}S\rightarrow \CP^{1}$,
then the limit curve of $\rho$ is given by
\[f_{G}\circ\xi: \partial_{\infty}\pi_{1}S\rightarrow G/P
\]
where $f_{G}: \CP^{1}\rightarrow G/P$ is the principal rational
curve. \noproof
\end{prop}

This proposition can be proved using the criterion in \cite{GW12}
regarding when an Anosov representation remains Anosov after composing
with a homomorphism to a larger Lie group, but we include a sketch of
a proof here to give some indication of how Definition~\ref{Anosov
  definition} is applied.

\begin{proof}
Firstly, by Proposition~\ref{anosov exchange}, if we show that the above statement is 
true for a Borel subgroup $P=B$, then the result follows for all other
symmetric parabolics.

By Proposition~\ref{prop: opposite}, the composition
\[
f_{G}\circ\xi: \partial_{\infty}\pi_{1}S\rightarrow G/B
\]
satisfies Condition~\ref{transverse} of Definition~\ref{Anosov
  definition}.  For Condition~\ref{dynamics preserving} of the
definition, we use conjugation in $G$ to affect the same normalization
of $\mathfrak{S}$ considered in Proposition~\ref{prop: opposite},
where $\alpha(x_0) = 2$ for all $\alpha \in \Delta$ and $x_{0}\in \g$
is the semisimple element of the $\mathfrak{sl}_{2}$-triple
generating the principal three dimensional subalgebra.  For any
nontrivial element $\gamma\in \pi_{1}S$ we can assume (after
conjugating $\eta$ in $\PSL_2\C$) that
$\eta(\gamma) = \exp \left ( \zeta \left
( \begin{smallmatrix}1&0\\0&-1\end{smallmatrix} \right ) \right )=
\left ( \begin{smallmatrix}e^\zeta& 0\\0 &
e^{-\zeta} \end{smallmatrix}\right )$
for $\zeta \in \C$ with $\Re(\zeta) > 0$.  Thus $\xi(\gamma_+) = z_0$
satisfies $f_G(z_0) = e B$ and $\rho(\gamma) = \exp(\zeta x_0)$.  Then
\[
T_{eB}(G/B)\simeq \bigoplus_{\alpha\in \Sigma^{-}} \g_{\alpha}
\]
and this is a decomposition into eigenspaces for the action of
$\rho(\gamma)$, where the eigenvalue on $\g_\alpha$ is
$\exp(\alpha(\zeta x_0))$.  Since $\alpha(x_0) = 2$ for
$\alpha \in \Delta$, for $\alpha \in \Sigma^-$ we have
$\alpha(x_0) < 0$ and $|\exp(\alpha(\zeta x_0))| < 1$.  This verifies
that $eB$ is the attracting fixed point for $\rho(\gamma)$, 
and Condition~\ref{dynamics preserving} of Definition~\ref{Anosov
  definition} follows.

Finally, for property \ref{divergence} of Definition~\ref{Anosov
  definition}, we note that for any divergent sequence of regular
semisimple elements $\{g_{n}\}\subset\PSL_2 \C$, their images under
the principal three-dimensional embedding $\iota_{G}$ satisfy
\[
\lim_{n\rightarrow \infty}\langle \mu(\iota_{G}(g_{n})), \alpha\rangle= \infty
\]
for every simple root $\alpha\subset \Delta$.  Since every element in
the image of $\eta$ is regular semisimple, this verifies property
\ref{divergence} and completes the proof.
\end{proof}

Using Theorem~\ref{rank one Anosov}
and the equivalence of quasi-Fuchsian and convex-compact for
representations $\pi_{1}S \to \PSL_2\C$ we have the well-known
corollary (which was part of the initial motivation for the study of
Anosov representations):
\begin{cor}
The set of $B$-Anosov representations $\rho: \pi_{1}S\rightarrow \PSL_2 \C$ is equal
to the set of quasi-Fuchsian representations. \noproof
\end{cor}

Let $P$ be a symmetric parabolic subgroup of $G$. We define the space
of \emph{$(G,P)$-quasi-Hitchin representations}
\[
\QH(S, G, P)\subset \bar{\character}(S, G)
\]
as the connected component of $P$-Anosov representations which
contains the Hitchin representations.  In case $G=\PSL_2\C$ this
reduces to the set of quasi-Fuchsian representations, i.e.~$\mathcal{QF}(S)=\QH(S, \PSL_2 \C, B)$.

For later use, we denote by $\tilde{\QH}(S, G, P) \subset
\Hom(\pi_{1}S, G)$ the preimage of $\QH(S, G, P)$ under the quotient
mapping $\Hom(\pi_{1}S,G) \to \bar{\character}(S,G)$.

\section{Flag varieties and the KLP construction}\label{KLP machine}

In this section, we will explain in some detail the construction of
Kapovich-Leeb-Porti of domains of discontinuity for Anosov
representations.  Our account differs from that of \cite{KLP13} in
that we focus on complex semisimple Lie groups and their associated
flag varieties and avoid the discussion of visual boundaries of
symmetric spaces.  This presentation is tailored to the application of
Sections \ref{size limit set} and \ref{topology}.

\subsection{Length function and Chevalley-Bruhat order}
\label{subsec:bruhat}
References for the following standard material include
\cite{bourbaki:lie-iv} and \cite{bjorner-brenti}.

Let $G$ be a complex semisimple Lie group.  As in Section \ref{back
  and prelim} let $W$ denote the Weyl group of $G$ associated to a
maximal torus $H < G$.  Fix a system $\Delta$ of simple roots and let
$S = \{ r_\alpha \suchthat \alpha \in \Delta\}$ denote the associated
system of reflection generators for $W$.  Then $(W,S)$ is a Coxeter
system and hence gives rise to a partial order $<$ on $W$, the
\emph{Chevalley-Bruhat order}, which can be defined as follows: A word
in $S$ that has minimum length among all words representing the same
element of $W$ is called a \emph{reduced word}.  Given a word $w$ in
$S$, we say that $z$ is a \emph{subword} of $w$ if $z$ is the result
of deleting zero or more letters from arbitrary positions within $w$.
Then $x < y$ if and only if $x$ can be represented by a subword of a
reduced word for $y$.  It can be shown that this definition gives a
partial order on $W$ (and in particular is transitive); see
e.g.~\cite[Definition~2.1.1 and Corollary~2.2.3]{bjorner-brenti}.

Closely related to this partial order on $W$ is the \emph{length function}
\[
\ell : W \to \Z^{\geq 0}
\]  
where $\ell(x)$ is the length of any reduced word for $x$.  It is
immediate that $x < y$ implies $\ell(x) < \ell(y)$.

Inversion in $W$ preserves both of these structures,
i.e.~$\ell(x^{-1}) = \ell(x)$ and $x < y$ if and only if
$x^{-1} < y^{-1}$.  When $a < b$ for $a,b \in W$, we say that $b$
\emph{dominates} $a$.  In the usual way we use $\leq$ to denote the
associated non-strict comparison operation of the Chevalley-Bruhat
order.

The longest element $w_0 \in W$ was introduced in Section \ref{back
  and prelim} and defined relative to its action on the roots;
equivalently, $w_0$ is the unique element of $W$ on which the function
$\ell$ attains its maximum. Multiplication by $w_0$ on the left defines an
antiautomorphism of the Chevalley-Bruhat order and length function; that is, it
inverts length and comparisons:
\begin{equation}
\label{eqn:reverse-length}
\ell(w_0 w) = \ell(w_0) - \ell(w) \;\; \text{ and } \;\; (a < b) \Leftrightarrow (w_0 b < w_0 a).
\end{equation}

Now let $P<G$ be a standard parabolic subgroup. The \emph{Weyl group of
 $P$} is defined as
\[
W_P =(N_{G}(H) \cap P)/H.
\]
Note that $W_P < W = N_G(H) / H$, and for the Borel subgroup we have
$W_B = \{e\}$.  The space $W/W_P$ of left $W_{P}$-cosets inherits a
partial order from that of $W$ as follows: Each coset $w W_P$ has a
unique minimal element, and letting $W^P$ denote the set of such
minimal elements, we have a canonical bijection $W/W_P \simeq W^P$.
Restricting the Chevalley-Bruhat order to $W^P$ gives the desired
partial order on $W/W_P$.  Extending the previous terminology, we also
call this order on $W/W_P$ the Chevalley-Bruhat order, and we call the
resulting rank function the length function on $W/W_P$.  Explicitly,
the latter function is
\[
\ell : W/W_P \to \Z^{\geq 0},  \;\; \ell(wW_P) := \min_{w' \in w W_P} \ell(w').
\]  
We also note that the length function on $W/W_P$ satisfies
\[
\ell(w_0 w W_P) = \ell(w_0 W_P) - \ell(w W_P),
\]
and $\ell(w_0 W_P)$ is the maximum value of $\ell$ on $W/W_P$.

There is a further extension of the Chevalley-Bruhat order for a pair
$P,Q$ of standard parabolic subgroups of $G$:  Each double coset in
$W_P \backslash W / W_Q$ contains a unique minimal element, and
restricting the Chevalley-Bruhat order to the set $W^{P,Q}$ of such
minimal elements gives a partial order on $W_P \backslash W / W_Q$.

\subsection{Chevalley-Bruhat ideals}
\label{subsec:ideals}

A \emph{Chevalley-Bruhat ideal} (or briefly, an \emph{ideal}) is a
subset $I \subset W$ such that if $b \in I$ and $a < b$, then
$a \in I$.  That is, $I$ is \emph{downward closed} for the partial
order.  (In \cite{KLP13} ideals are called \emph{thickenings},
though several other objects are given that name as well; we reserve
the term thickening for a subset of the flag variety that is defined
below.)  Associated to any element $x \in W$ there is the
\emph{principal ideal} defined as
$\langle x \rangle = \{ w \in W \suchthat w \leq x \}$.  It is easy to
see that every ideal $I \subset W$ is a union of principal ideals, and
in fact has a unique minimal description
$I = \bigcup_{i=1}^r \langle x_i \rangle$ as a union of
principal ideals.  The elements $x_i$ appearing in this minimal
presentation are exactly those which lie in $I$ but are not dominated
by any element of $I$.  We call $\{x_1, \ldots, x_r\}$ the
\emph{minimal generating set} of $I$.

If $I \subset W$ is an ideal, then $I^{-1} = \{ x^{-1} \suchthat x \in
I\}$ is also an ideal.  The complement of a non-empty ideal is never an
ideal, however if we define
\[
I^\perp = w_0(W \setminus I)
\]
then, by the antiautomorphism property of $w \mapsto w_0 w$, we find that
$I^\perp$ is an ideal. We call this the \emph{orthogonal} of $I$.  Note
that it is always the case that 
\[
W = I \sqcup w_0 I^\perp.
\]

Following the terminology of \cite{KLP13}, we say that an ideal
$I\subset W$ is \emph{slim} if $I\subset I^\perp$, \emph{fat} if
$I \supset I^\perp$, and \emph{balanced} if $I = I^\perp$
(equivalently, if it is both fat and slim). Note in particular that a
balanced ideal satisfies $\card{I} = \frac{1}{2} \card{W}$, and that
for slim ideals, this cardinality condition is equivalent to being
balanced.

\subsection{Flag variety and Schubert cells}
We now discuss the cell structures of flag varieties in relation to
the Weyl group and the Chevalley-Bruhat order; this material is
standard and can be found in e.g. \cite{BGG73}
\cite{fulton:young-tableaux} \cite{lakshmibai-gonciulea} \cite{CG10}.

Let $B<G$ be the Borel subgroup associated to the choice $\Delta$ of
simple roots fixed above.  The homogeneous space $G/B$ is the
\emph{full flag variety} of $G$.  If $P \subset G$ is a parabolic
subgroup, then $G/P$ is the \emph{partial flag variety} associated to $P$.
All flag varieties are smooth projective varieties over $\C$, and in
particular are compact oriented manifolds.

The full flag variety $G/B$ has a natural decomposition into a disjoint
union of $B$-orbits called \emph{Schubert cells} 
\[
\{ C_w = BwB \suchthat w \in W \}.
\]  
Each $C_w$ is diffeomorphic to $\C^{\ell(w)}$.  The closure $X_w = \bar{C_w}$ is a
\emph{Schubert variety} and can be described as the union of the cells
that are dominated by $w$ in the Chevalley-Bruhat order:
\[
X_w = \{ C_{w'} \suchthat w' \leq w \}.
\]  
Therefore, there is a bijection between $W$ and the set of Schubert
cells, where ideals $I\subset W$ correspond to unions of Schubert
varieties.  In topological terms, ideals $I\subset W$ are in bijection
with closed, cellular subcomplexes of $G/B$ with respect to the
cellular structure given by the Schubert cells.  In algebraic terms,
Schubert varieties are irreducible projective subvarieties of $G/B$.

For a parabolic $P < G$ containing $B$, we have the projection $\pi :
G/B \to G/P$.  Under this projection, the Schubert cell
decomposition of $G/B$ projects to a cell decomposition of $G/P$, and
the projection of a Schubert cell $C_w$ to $G/P$ depends only on the
coset $wW_P \in W/W_P$.  Thus, the cells in $G/P$ are indexed by the
coset space $W/W_P$, or by the collection of coset representatives
$W^P$.  We define
\begin{equation*}
\begin{split}
C_{w W_P} &:= \pi(C_w)\\
X_{w W_P} &:= \overline{C_{wW_P}}.
\end{split}
\end{equation*}
The set $X_{w W_P}$ is called a Schubert variety in $G/P$, and is an
irreducible projective subvariety.  As before, the Chevalley-Bruhat
order (now on $W/W_P$) is equivalent to the inclusion partial order on
these Schubert varieties.  Note that the real dimension of $G/B$ is
$2\ell(w_0)$, while that of $G/P$ is $2 \ell(w_0 W_P)$.

The Schubert cells are defined as $B$-orbits in flag varieties of $G$.
In what follows, we will also need to understand the structure of the
$P$-orbits on $G/Q$ for $P, Q$ parabolic subgroups.  We
summarize the results in the following (see \cite{PER02} and
\cite{PEC14}):
\begin{thm}\label{P orbits}\mbox{}
\begin{rmenumerate}
\item Every $P$-orbit in $G/Q$ can be written as $PwQ$ for some $w\in W$.
\item This description gives a bijection between the set of
$P$-orbits in $G/Q$ and the double cosets $W_{P}\backslash W/W_{Q}$, where $W_{P}$ and $W_{Q}$ denote the
Weyl groups of $P$ and $Q$.
\item The inclusion partial order on closures of $P$-orbits in $G/Q$
corresponds, under this bijection, to the Chevalley-Bruhat order on
$W_P \backslash W / W_Q$.
\item Each $P$-orbit is a union of $B$-orbits; specifically, we have
\begin{equation}
\label{eqn:p-orbit-decomp}
PwQ=\bigcup_{(w_{P},w_{Q})\in W_{P}\times W_{Q}} Bw_{P}ww_{Q}Q.
\end{equation}
\end{rmenumerate} \noproof
\end{thm}

\subsection{Homology and cohomology of the flag variety}
\label{subsec:pd-flag}
First, we fix the following notation for the rest of the paper: If $E$
is a set, then $\Z_E$ denotes the free abelian group on $E$, i.e.~the
set of all formal finite linear combinations of elements of $E$ with
integer coefficients.  Of course if $E$ is itself a group, then $\Z_E$
is the underlying abelian group of the integral group ring of $E$.
However, we will not use any ring structure on $\Z_E$ in the sequel.
Also, we observe that any function $E \to \Z$ gives $\Z_E$ the
structure of a graded abelian group.

As in the previous section, let $P$ be a parabolic subgroup of $G$, a
complex semisimple Lie group.  The integral homology $H_*(G/P, \Z)$ is
naturally isomorphic to $\Z_{W/W_P}$ with grading given by twice the
length function, $2\ell$.  This can be seen using cellular homology
for the Schubert cell decomposition of $G/P$; then $\Z_{W/W_P}$ with
grading $2 \ell$ is the cellular chain complex, and the boundary maps
are zero since all cells have even dimension.  Concretely, in this
isomorphism the element $wW_P \in W/W_P$ corresponds to the cell
$C_{wW_p}$ (in the cellular resolution), or to the fundamental class
$[X_{wW_P}]\in H_{2\ell(wW_{P})}(G/P, \Z)$ of the Schubert variety
$X_{wW_P}$.

Correspondingly, the universal coefficients theorem identifies
$H^*(G/P, \Z)$ with the dual abelian group $\Z^{W/W_P}$ of $\Z_{W/W_P}$; here the
Kronecker function
\[
\delta_{w
  W_P} : W/W_P \to \Z
  \]
   corresponds to a cohomology class $[X^{w
  W_P}]$, and these form the dual basis to 
  \[
  \{[X_{wW_P}] \suchthat
wW_P \in W/W_P \}.
\]

In terms of these models, the Poincar\'e duality isomorphism is given
by left multiplication by $w_0$ (see e.g.~\cite{BGG73}),
\begin{equation*}
\begin{split}
PD: H_k(G/P) &\to H^{2n-k}(G/P)\\
[X_{w W_P}] &\mapsto [X^{w_0 w W_P}]
\end{split}
\end{equation*}
where $n = \dim_\C G/P$.  Equivalently, the intersection pairing
\[
\langle \cdot ,\cdot \rangle :
H_*(G/P) \times H_{*}(G/P) \to \Z
\]
is given by
\[
\langle [X_{w W_P}] , [X_{w' W_P}] \rangle = \begin{cases}
1 & \text{ if } w^{-1} w_0 w' \in W_P\\
0 & \text{ otherwise }.
\end{cases}
\]

\subsection{Relative position}\label{relative position}

In this subsection, we give a more algebraic exposition of \cite[Section 3.3]{KLP13}.

There is a combinatorial, $W$-valued invariant associated to a
pair of points $p,q \in G/B$ called the \emph{relative position} and
denoted by $\pos(p,q)$.  It can be defined as follows: Choose an
element $g \in G$ such that $g \cdot p = eB$. Then $g \cdot q$ lies in
the Schubert cell $C_w \subset G/B$ for a unique $w \in W$, and we define
$\pos(p,q) = w$.  One can check that this is independent of the choice
of $g$.

To generalize this construction, let $P$ and $Q$ be standard parabolic
subgroups of $G$ corresponding to subsets
$\Theta_P, \Theta_Q \subset \Delta$, so that in particular
$B < P \cap Q$ and we have natural surjections $G/B \to G/P$ and
$G/B \to G/Q$.  Given $p \in G/P$ and $q \in G/Q$ we can select
respective preimages $\tilde{p},\tilde{q} \in G/B$ and consider their
relative position $\pos(\tilde{p},\tilde{q}) \in W$.  While this
element will depend on the choices of preimages, its double coset in
$W_P \backslash W / W_Q$ depends only on $p$ and $q$; we therefore
define the relative position of $p$ and $q$ by
\[
\pos_{P,Q}(p,q) = W_P \left( \pos(\tilde{p},\tilde{q}) \right) W_Q \in W_P \backslash W / W_Q.
\]
Our previous definition is the special case $\pos_{B,B} = \pos$.  It
is immediate from the definition that the relative position is
$G$-invariant in the sense that
\begin{equation}
\label{pos G invariant}
\pos_{P,Q}(p, q)=\pos_{P,Q}(g(p), g(q))
\end{equation}
for all $g\in G$.  Moreover, from its construction the relative
position function is closely tied to the decompositions of $G/P$ and
$G/Q$ into Schubert cells.  We summarize its key properties in the
following proposition, which follows easily from Theorem~\ref{P orbits}:
\begin{prop}
\label{prop:rel pos}
Suppose $p \in G/P$, $q \in G/Q$, and $g \in G$ satisfies
$g \cdot p = eP$.  Then we have $\pos_{P,Q}(p,q) = W_P w W_Q$ if and
only if $g \cdot q$ is contained in the $P$-orbit on $G/Q$ which is
labeled by the double coset $W_P w W_Q$ in the sense of Theorem~\ref{P
  orbits}.(ii).  Thus the level set
$\{ q \in G/Q \suchthat \pos_{P,Q}(p,q) = W_P w W_Q\}$ is a $gPg^{-1}$-orbit
on $G/Q$.  Moreover, the closure of this $gPg^{-1}$-orbit is given by
the sublevel set
\[
\{q' \in G/Q \suchthat \pos_{P,Q}(p,q') \leq
\pos_{P,Q}(p,q) \}
\]
where $\leq$ is the Chevalley-Bruhat order on $W_P \backslash W /
W_Q$.

In particular the Schubert cell in $G/Q$ labeled
by coset $wW_Q$ is given by the level set
\[ C_{w W_Q} = \{ q \suchthat \pos_{B,Q}(eB,q) = w W_Q \}\] and the corresponding
Schubert variety $X_{w W_Q}$ is the sublevel set
\[ X_{w W_Q}= \overline{C_{w W_Q}} = \{ q \suchthat \pos_{B,Q}(eB,q) \leq w W_Q \}.\]
\noproof
\end{prop}
\noindent This proposition shows that the ideal $I$ in the Weyl group $W$,
which corresponds to a closed union of Schubert varieties, equally
corresponds to a union of sublevel sets of the relative position
function over the generators of the ideal.

\subsection{Parabolic pairs and thickenings}

We have considered pairs of standard parabolic subgroups $(P,Q)$ and
the corresponding $W_P \backslash W / W_Q$-valued relative position
function.  

Now fix such a pair $(P_A,P_D)$ of parabolics with $P_A$ symmetric,
and consider $P_A$-Anosov representations $\rho : \pi \to G$.  (Recall
that by Proposition~\ref{sym parabolic} there is no loss of generality
in requiring $P_A$ to be symmetric.)  We consider the action of $\pi$
on the partial flag variety $G/P_D$ induced by $\rho$, with the goal
of finding a domain $\Omega \subset G/P_D$ on which the action is
properly discontinuous.  Thus the notation for the parabolics
signifies that $P_A$ is the ``Anosov parabolic'', while $P_D$ is the
``domain parabolic''.

We make corresponding abbreviations $W_A := W_{P_A}$ and
$W_D = W_{P_D}$ for the Weyl groups, and abbreviate the 
relative position function $\pos_{P_A,P_D}$ by $\pos_{A,D}$.

We say that an ideal $I \subset W$ has
\emph{type $(P_A,P_D)$} if $I$ is left $W_A$-invariant and right
$W_D$-invariant.  Equivalently $I$ is a union of double cosets
$W_A w W_D$.  Let $I \subset W$ be such an ideal.  We can define the
associated union of $P_{A}$-orbits
\[ \Phi^I := \bigcup_{W_{A}wW_{D} \in W_{A}\backslash I/W_D}
P_{A}wP_{D}\subset G/P_{D} \]
which we call the \emph{model thickening} associated to $I$.  (In
\cite[Section 3.4.2]{KLP13} this is called a \emph{thickening at
  infinity}.)  By Theorem~\ref{P orbits} the set $\Phi^{I}$ is a union
of Schubert cells, and since $I$ is an ideal, the set $\Phi^{I}$ is in
fact a finite union of Schubert varieties.  In particular it is a
closed set.

In the sequel, the sets obtained from $\Phi^{I}$ by applying an
element of $G$ play a key role.  It is evident from the definition of
$\Phi^{I}$ that the set $g \cdot \Phi^{I}$ depends only on the coset
$g P_A$.  Thus for any $p \in G/P_A$ we have a well-defined subset of
$G/P_D$,
\[
\Phi^I_p := g \cdot \Phi^I \text{ for any } g \in G \text{ such that }
g P_A = p.
\]
We call $\Phi_p^I$ the \emph{thickening of $p$} associated with $I$.
This set can also be characterized in terms of relative position;
using $G$-invariance of the relative position function and
Proposition~\ref{prop:rel pos}, it follows that
\[
\Phi_p^I = \{ q \in G/P_D \suchthat \pos_{A,D}(p,q)
\in W_{A}\backslash I/W_{D} \}.
\]
It is immediate from the definition that the construction of
$\Phi^{I}_p$ is compatible with the $G$-action in the following sense:

\begin{prop}\label{prop: thicktranslate}
For $g\in G$ and $p\in G/P_{A},$ the thickenings satisfy
$\Phi_{g(p)}^{I}=g \cdot \Phi_{p}^{I}$. \noproof
\end{prop}

\subsection{Limit sets and domains}\label{limit sets}
Let $P_A$ and $P_D$ be parabolic subgroups, with $P_A$ symmetric. For
any subset $V \subset G/P_A$, define the thickening of $V$, denoted
$\Phi_V^I$, as the union of the thickenings of its points:
\[ \Phi_V^I = \bigcup_{p \in V} \Phi_p^I \]
Let $\rho: \pi\rightarrow G$ be a $P_{A}$-Anosov representation with
limit curve $\xi : \partial_\infty \pi \to G/P_A$, and let $I$ be an
ideal of type $(P_A,P_D).$ The \emph{limit set} of $\rho$ relative to
$I\subset W$ is defined as the thickening of the limit curve, i.e.
\[
\Lambda^I_\rho := \Phi_{\xi(\partial_\infty \pi)}^I = \bigcup_{t \in \partial_{\infty}\pi}
\Phi_{\xi(t)}^I\subset G/P_{D}
\]
The complement
\[
\Omega^I_\rho := G/P_D \setminus \Lambda_{\rho}^I
\]
is the associated \emph{domain}, which by the equivariance of $\xi$ is a
$\rho(\pi)$-invariant open set.  Let $\Gamma:=\rho(\pi).$

The paramount result of \cite{KLP13} establishes that if $I$ is
balanced, then the complement of the limit set furnishes a cocompact
domain of proper discontinuity for the action of $\Gamma$ on $G/P_D$.
More generally:
\begin{thm}[\cite{KLP13}]\mbox{}
\label{thm: klp-domain}
\begin{rmenumerate}
\item If $I$ is a slim, then the action of $\Gamma$ on $\Omega^I_\rho$ is
properly discontinuous.
\item If $I$ is fat, then the action of $\Gamma$ on $\Omega^I_\rho$ is
cocompact.\noproof
\end{rmenumerate}
\end{thm}

In this construction, there remains the question of whether
the domain $\Omega_{\rho}^{I}$ could be empty.  In \cite{KLP13}
and \cite{GW12}, various conditions are obtained ensuring the
non-emptiness of the domains. In our primary applications, we will
show that the corresponding domains are non-empty.

Regarding the structure of the limit set, the same authors show:
\begin{thm}[{\cite[Lemma~3.38 and Lemma~7.4]{KLP13}}]
\label{thm:klp-local-triviality}
If $I$ is a slim ideal of type $(P_A,P_D)$, then the set
$\Lambda^I_\rho$ is a locally trivial topological fiber bundle over
$\partial_{\infty}\pi$ with typical fiber $\Phi^I$.

More generally, if $V \subset G/P_A$ is a compact set consisting of
pairwise opposite points, then the set $\Phi_V^I$ is a locally trivial
topological fiber bundle over $V$, where the projection
$p : \Phi_V^I \to V$ is given by $p(\Phi_x^I) = x$.  In particular,
the thickenings $\{ \Phi_x^I \suchthat x \in V\}$ are pairwise
disjoint. 
\end{thm}

It will be important in what follows to know that this bundle is
trivial for $G$-Fuchsian representations (which, we recall, are
defined when $G$ is simple and of adjoint type).  This follows from
similar considerations as those used in the proof of the theorem
above.

\begin{lem}
\label{thm:triviality}
Let $G$ be a complex simple Lie group of adjoint type.
If $\rho:\pi_{1}S\rightarrow G$ is $G$-Fuchsian and $I$ is a slim ideal of type $(P_A,P_D)$,
then there is a homeomorphism $\Lambda^I_\rho \simeq \Phi^I \times S^1$.
\end{lem}

\begin{proof}
Recall that a locally trivial fiber bundle over $S^1$ is trivial if
and only if it extends over the closed $2$-disk.  We show that
$\Lambda_{\rho}^I$ admits such an extension.

By Proposition \ref{prop: opposite}, the entire principal curve in
$G/B$ consists of pairwise opposite points.  Under the projection
$G/B \to G/P_A$, opposite Borel subgroups map to opposite parabolics,
hence the principal curve $X := f_{G}(\CP^{1}) \subset G/P_A$ has the
same property.  By Theorem~\ref{thm:klp-local-triviality}, the set
$\Phi_X^I$ is a fiber bundle over $X$.  By Theorem~\ref{quasifuchsian
 Anosov}, the limit curve of a $G$-Fuchsian representation is the
image of the limit curve of the associated Fuchsian group, which is
simply the extended real line in the principal curve:
\[
\xi(\partial_\infty \pi_{1}S) = f_{G}(\RP^{1}) \subset G/P_{A}
\]
Denoting the image as $X_{\R} := f_{G}(\RP^{1}) \subset X$,
the limit set $\Lambda_{\rho}^I$ is
\[
\Lambda_{\rho}^I = p^{-1}(X_\R) \subset \Phi_X^I
\]
where
\[
p: \Phi_{X}^{I}\rightarrow X
\]
is the aforementioned projection.

We have therefore described the bundle $\Lambda_{\rho}^I$ over base
$S^1\simeq \RP^{1} \simeq X_\R$ as the restriction to the equator of a
bundle over $S^2 \simeq \CP^{1} \simeq X$.  Since $S^1$ bounds a disk
in $\CP^{1}$, the Lemma follows.
\end{proof}

For later use, we record that the domains constructed in Theorem~\ref{thm: klp-domain}
for a $G$-Fuchsian representation are invariant under the full group $\iota_{G}(\PSL_2\R).$  
\begin{prop}\label{PSL invariance}
Let $G$ be a complex simple Lie group of adjoint type and $I \subset W$ an ideal of type
$(B, P_{D}).$  If $\rho: \pi_{1}S\rightarrow G$ is a $G$-Fuchsian representation,
then the domain $\Omega_{\rho}^{I}\subset G/P_{D}$ is invariant under $\iota_{G}(\PSL_2\R)$.
\end{prop}

\begin{proof}
Since the limit curve $\xi(\partial_\infty \pi_{1}S) =
f_G(\RP^1)$ in this case is an orbit of $\iota_G(\PSL_2\R)$ on
$G/P_A$, this is immediate from Proposition~\ref{prop:
  thicktranslate}.
\end{proof}

\section{Size of the limit set}\label{size limit set}

We now consider combinatorial properties of Weyl ideals and
apply them to estimate the Hausdorff dimension of the limit sets
described above.  The results of this section are not used in Section
\ref{topology}, however they are essential to the complex geometry
results of Section \ref{complex geometry}.

\subsection{Weyl ideal combinatorics}
As before we refer the reader to \cite{bourbaki:lie-iv} or
\cite{bjorner-brenti} for more detailed discussion of the Coxeter
group structure of the Weyl group $W$.  We will also use the
classification of complex simple Lie algebras into Cartan types
$A$--$G$ as described for example in
\cite[Section~VI.2]{bourbaki:lie-iv}.

As in the previous section we assume $G$ is a complex semisimple Lie
group, hence $\g$ decomposes as a direct sum of simple Lie algebras,
which we call the \emph{simple factors}.  There is a corresponding
direct product decomposition of the Weyl group $W = W(G)$.

Our goal in this section is to show:
\begin{thm}\mbox{}
\label{thm:short-small}
Let $I \subset W$ be a fat ideal.
\begin{rmenumerate}
\item If $G$ has no factors of type $A_1$, then $I$ contains each
element $w \in W$ with $\ell(w) \leq 1$.
\item If $G$ has no factors of type $A_1$, $A_2$, $A_3$, or $B_2$,
then $I$ contains each element $w \in W$ with $\ell(w) \leq 2$.
\end{rmenumerate}
\end{thm}

Note that by the exceptional isomorphisms, this also excludes types
$B_{1}, C_{1}, C_{2}$ and $D_{3}.$ In terms of the classical matrix
groups, representatives of the excluded types are given by
$A_1 = \sl_2\C$, $A_2 = \sl_3\C$, $A_3 = \sl_4\C$,
$B_2 = \mathfrak{so}_5\C$.

Toward the proof of the theorem, we introduce the following
terminology: An element $x \in W$ will be called \emph{small} if
$x \leq w_0x$, where $w_0 \in W$ is the longest element (as in
Section~\ref{complex ss groups}).

\begin{lem}
\label{lem:small-in}
If $I \subset W$ is a fat ideal and $x \in W$ is small, then $x \in I$.
\end{lem}

\begin{proof}
Suppose for contradiction that $x$ is small, $I$ is a fat ideal, but
that $x \not \in I$.  Then $w_0 x \in w_0(W \setminus I)$, and since $I$
is fat we have $w_0(W \setminus I) \subset I$, thus $w_0x \in I$.
Since $x$ is small we have $x < w_0x$, and $I$ is an ideal, so we find
$x \in I$, a contradiction.
\end{proof}

Theorem~\ref{thm:short-small} will follow from showing that elements
of $W$ of small length (i.e.~``short'' elements) are small.  To do
this we will require some additional properties of the length
function and Chevalley-Bruhat order on $W$, which we now state.

First, we need a construction of reduced words representing $w_0$.
The description of these will involve a positive integer associated to
$W$, the \emph{Coxeter number}, which is defined as the order in $W$
of any element that is the product of all of the simple root
reflections (in some order).  We denote the Coxeter number by $h$, and
abusing the terminology we will also refer to it as the Coxeter number
of $G$ or $\g$.  (Further discussion of the Coxeter number can be
found in e.g.~\cite[Section V.6.1]{bourbaki:lie-iv}.)

\begin{lem}[{Bourbaki \cite[pp.~150--151]{bourbaki:lie-iv}}]
\label{lem:longest-word}
Suppose $G$ is simple and has Coxeter number $h$.  Let $S = S' \sqcup
S''$ be a partition such that each of $S',S''$ generates an abelian
subgroup of $W$.  Let $a$ (resp.~$b$) denote the product of the
elements of $S'$ (resp.~$S''$).  Then:
\begin{rmenumerate}
\item If $h$ is even, then $w_0 = (ab)^\frac{h}{2}$ is a reduced word.
\item If $h$ is odd, then $w_0 = (ab)^\frac{h-1}{2}a$ is a reduced word.\noproof
\end{rmenumerate}
\end{lem}

Note that the order in the product $a$ does not matter since elements
of $S'$ commute, and similarly for $b$.  Partitions
$S = S' \sqcup S''$ of the type considered here always exist, as
each Dynkin diagram admits a $2$-coloring and non-adjacent vertices
correspond to commuting simple root reflections.

Lemma~\ref{lem:longest-word} also gives reduced words for $w_0$ when
$G$ is semisimple, by taking a product $\prod_i w_0^{(i)}$ of words
of type (i) or (ii) for the longest elements $w_0^{(i)}$ of the Weyl
groups of the simple factors.

Next, we need the following relation between a reduced word for an
element $x \in W$ and for its product $xs$ with a simple root
reflection:

\begin{lem}\mbox{}
\label{lem:exchange}
Suppose $x \in W$ and $s \in S$ satisfy $\ell(xs) = \ell(x)-1$, and
that
$$x = s_1 \cdots s_{\ell(x)}$$
is a reduced word for $x$.  Then for some
$k \in \{1, \ldots, \ell(x)\}$ we have that
$$xs = s_1 \cdots \widehat{s_k} \cdots s_{\ell(x)},$$
and furthermore, $s_k$ is conjugate to $s$. \noproof
\end{lem}

Proofs of these standards facts about Coxeter groups can be found 
for example in \cite[Corollary~1.4.4]{bjorner-brenti}.  Note that
these properties are often stated in terms of left multiplication by a
reflection; the version for right multiplication stated above is
equivalent, however, since the inversion map $w \mapsto w^{-1}$ is an
automorphism of the Chevalley-Bruhat order.

Combining the previous lemmas we can now establish the key
combinatorial property that underlies Theorem~\ref{thm:short-small}:
\begin{lem}\mbox{}
\label{lem:deletion}
\begin{rmenumerate}
\item If each simple factor of $G$ has Coxeter number at least
$3$, then each element of $S$ is small.
\item If each simple factor of $G$ has Coxeter number at least
$5$, then for any $s,t \in S$ the element $st \in W$ is small.
\end{rmenumerate}
\end{lem}

\begin{proof}
First suppose $G$ is simple with Coxeter number $h \geq 3$ and let
$s \in S$.  Note that $\ell(w_0s) = \ell(w_0) - 1$ by
\eqref{eqn:reverse-length}.  Apply Lemma~\ref{lem:longest-word} to a
partition of $S$ with $s \in S'$ to obtain a reduced expression of the
form $w_0 = a b a z$, where $z$ is a (possibly empty) alternating
product of $a$ and $b$.  The simple root reflection $s$ appears at
least twice in this word (once in each copy of $a$), hence by
Lemma~\ref{lem:exchange} we find that $s$ appears at least once in a
reduced expression for $w_0s$.  This shows $s < w_0s$ and thus $s$ is
small.

Now suppose $G$ is simple with Coxeter number $h \geq 6$. (The case
$h=5$ is considered separately below.)  Let $s,t \in S$.  We will show
$st$ is small.  If $s=t$ then $s^2=e$ and this is trivial, so we
assume $s \neq t$.  Then $\ell(st) = 2$,
$\ell(st w_0) = \ell(w_0) - 2$, and $\ell(tw_0) = \ell(w_0) -1$.
Proceeding as before and using $h \geq 6$ we obtain a reduced
expression $w_0 = abababz$, where we can assume $s$ appears in product
$a$.  Applying Lemma~\ref{lem:exchange} twice we find that a reduced
word for $w_0st$ can be obtained from this one for $w_0$ by deleting
two letters, and each such deletion may alter one of the copies of $a$
or $b$ in this word.  However, this leaves at least one unaltered copy
$a$ to the left of an unaltered copy of $b$.  That is, $ab$ is a
subword of a reduced expression for $w_0st$.

The simple root reflection $t$ appears in either $a$ or $b$.  If it
appears in $b$, then $st$ is evidently a subword of $ab$.  If $t$
appears in $a$, then $s$ and $t$ commute and one of the equivalent
words $st = ts$ is a subword of $ab$.  Thus in either case we conclude
$st < w_0st$, hence $st$ is small.

If $G$ is simple and $h=5$ then $G$ is of type $A_4$, hence
$W \simeq S_5$.  In this case it can be checked directly that the nine
non-trivial elements which are products of pairs of simple root
reflections are small.  We omit the details of this verification.

Finally suppose $G$ is semisimple.  We have a reduced expression for
$w_0$ that is a product over the simple factors.  If each simple
factor has Coxeter number at least $3$, we find as before that the
reduced expression for $w_0$ can be constructed to use a given simple
root reflection $s$ at least twice, and hence that $s$ is small.  If
each simple factor has Coxeter number at least $5$, and if $s,t$ are
simple root reflections ($s \neq t$), then a reduced word for $w_0st$
is obtained by deleting two letters from the word for $w_0$, and the
deleted letters are respective conjugates of $s$ and $t$.  If $s$ and
$t$ lie in the same simple factor of $W$, then the deleted letters are
both in the corresponding factor of $w_0$, and the argument above in
the simple case shows that $st$ is a subword of the result.  If $s$
and $t$ lie in distinct simple factors (and hence commute), we recall
that each can be assumed to appear at least twice in its factor and
hence each appears at least once after the deletion.  Thus $st = ts$
is also a subword of a reduced expression for $w_0st$ in this case.
We have therefore shown $st$ is small.
\end{proof}

Using this lemma, the proof of Theorem~\ref{thm:short-small} is straightforward:

\begin{proof}[Proof of Theorem~\ref{thm:short-small}.]
The elements $x \in W$ with $\ell(x) \leq 1$ are the simple root
reflections and the identity element.  The only simple Lie algebra of
Coxeter number less than $3$ is $A_1$, hence if $G$ has no simple
factors of this type then Lemma~\ref{lem:deletion}(i) shows that the
simple root reflections are small.  The identity element is also
small.  By Lemma~\ref{lem:small-in} we find that these elements lie in
any fat ideal $I \subset W$, and part (i) of the theorem follows.

In exactly the same way, part (ii) follows from Lemma
\ref{lem:deletion}(ii) because the elements $x \in W$ with
$\ell(x) \leq 2$ are the products of at most two simple root
reflections, and because the only simple Lie algebras with Coxeter
number less than $5$ are $A_1$, $A_2$, $A_3$, and $B_2$.
\end{proof}

\subsection{Hausdorff dimension of limit sets}

Now we will bound the Hausdorff dimension of the limit
set of an Anosov representation
in terms of the Hausdorff dimension of its limit curve and the
combinatorial size of the ideal defining the thickening.

All of the sets for which we discuss dimension are closed subsets of
compact manifolds.  When regarding such sets as metric spaces (for
example when computing dimensions) we always consider them to be
equipped with the distance obtained by restricting the distance induced
by an arbitrary Riemannian metric on the ambient manifold.  
Since any two Riemannian metrics on a
compact manifold are bi-Lipschitz, our results will not depend on the
particular metric chosen.

Let $P_{A}<G$ be a symmetric parabolic subgroup of a complex
semisimple Lie group $G$. Let $V\subset G/P_{A}$ be a closed subset
consisting of pairwise opposite points.  The property of a pair of points being
opposite is an open condition since it coincides with the unique open
orbit of $G$ acting diagonally on $G/P_A \times G/P_A$.  (Here we are
using the fact that $P_{A}$ is symmetric so that it is conjugate to
any of its opposite parabolic subgroups.)

Let $W$ be the Weyl group of $G$.  We begin with the following general
fact, which is a straightforward generalization of Theorem
\ref{thm:klp-local-triviality}:

\begin{prop}\label{prop:lip}
Let $P_{D}<G$ be a parabolic subgroup and let $I\subset W$ be a slim
ideal of type $(P_{A}, P_{D}).$ Let $V\subset G/P_A$ denote a compact
subset consisting of pairwise opposite points.  Then the fiber bundle $p:
\Phi_V^I \to V$ admits Lipschitz local parameterizations; that is, each point
$x \in V$ has a neighborhood $U_x$ such that there exists a Lipschitz
homeomorphism
\[ U_x \times \Phi^I \to p^{-1}(U_x). \]
\end{prop}

In fact, this proposition follows easily from the proofs of
\cite[Lemmas~3.39 and Lemma~7.4]{KLP13} which we stated
as Theorem~\ref{thm:klp-local-triviality} above.  We will simply
recall enough of the construction used by those authors to make the
Lipschitz property evident.

\begin{proof}
Note that the set $\Phi^I$ is compact.  For $x \in V$ let $U_x$ be a
relatively compact neighborhood of $x$ in $V$ over which there exists
a smooth section $s: U_x \rightarrow G$ of the quotient map
$G \to G/P_A$, and choose such a section. In the proof of
\cite[Lemma~7.4]{KLP13} it is shown that the map
\begin{equation*}
\begin{split}
U_x \times \Phi^{I} &\rightarrow p^{-1}(U_x) = \Phi^I_{U_x} \\
(x, y) &\mapsto  s(x)(y)
\end{split}
\end{equation*}
gives a local trivialization of the bundle $\Phi^I_{V} \to V$.
However, as it is the restriction of the smooth action map
$G \times G/P_D \to G/P_D$ to the relatively compact set
$s(U_x) \times \Phi^I_{V}$, this map is also Lipschitz.
\end{proof}

We now come to the main result of this section.

\begin{thm}\label{thm: hdimbound}
Let $P_{A}, P_{D}<G$ be a pair of parabolic subgroups with $P_{A}$
symmetric.  Let $\rho: \pi\rightarrow G$ be a $P_{A}$-Anosov
representation of a word hyperbolic group with limit curve
$\xi: \partial_{\infty}\pi\rightarrow G/P_A$. Let $I\subset W$ be a
slim ideal of type $(P_{A}, P_{D}).$ Then the limit set
$\Lambda_{\rho}^{I}\subset G/P_D$ satisfies
\[
\hdim(\Lambda_{\rho}^{I})\leq \hdim\left(\xi(\partial_{\infty}\pi)\right)+ 2\max_{w\in I/W_{D}}\ell(w).
\]
Here, the Hausdorff dimensions are computed with respect to any
Riemannian metrics on $G/P_A$ and $G/P_D$, and $\ell$ denotes the
length function associated to the Chevalley-Bruhat order on $W/W_D$.
\end{thm}

\begin{proof}
Recall $\Lambda_\rho^I = \Phi_{\xi(\partial_\infty \pi)}^I$ and 
$\xi(\partial_\infty \pi)$ is a compact set consisting of pairwise
opposite points (by Theorem~\ref{thm:klp-local-triviality}).  Applying
Proposition~\ref{prop:lip} we obtain a finite open cover $\{U_{i}\}$ of
$\partial_\infty \pi$ by sets whose images by $\xi$ are trivializing
open sets for the bundle $\Lambda_\rho^I$, and over which this bundle has
Lipschitz parameterizations.  Since Lipschitz maps do not increase
Hausdorff dimension, and since Hausdorff dimension is finitely stable,
we find
\begin{equation}
\label{eqn:reduction-to-product}
\hdim(\Lambda_{\rho}^{I}) \leq \max_i \,\hdim(\xi(U_i) \times \Phi^I).
\end{equation}
On the other hand, the Hausdorff dimension of a product can be bounded
in terms of the Hausdorff dimension and upper Minkowski dimension (also known
as upper box counting dimension) of the factors \cite[Formula~7.3]{Falconer}:
\[
\hdim(\xi(U_i) \times \Phi^I) \leq \hdim(\xi(U_i)) + \umdim(\Phi^I)
\]
However, $\Phi^I$ has a finite stratification by manifolds (the
Schubert cells corresponding to elements of $I$), and hence its upper
Minkowski dimension is equal to the maximum real dimension of these
manifolds (see e.g. \cite[Section~3.2]{Falconer}), which is
$2\max_{w\in I/W_{D}}\ell(w)$.  Also, since $\xi(U_i)$ is a subset of
$\xi(\partial_\infty\pi)$ we have
$\hdim(\xi(U_i)) \leq \hdim(\xi(\partial_\infty\pi))$.  We conclude
\[
\hdim(\xi(U_i) \times \Phi^I) \leq
\hdim\left(\xi(\partial_{\infty}\pi)\right) + 2\max_{w\in
  I/W_{D}}\ell(w).
\]
Substituting this bound into \eqref{eqn:reduction-to-product}, the
Theorem follows.
\end{proof}

We note that in case the right hand side of the bound from Theorem
\ref{thm: hdimbound} is less than the real dimension of $G/P_D$
itself, it follows that the limit set has positive ``Hausdorff
codimension'' and that $\Omega_\rho^I$ is non-empty.  We state the
resulting criterion separately:

\begin{thm}
  \label{thm:general-nonempty}
Let $\rho:\pi\rightarrow G$ be a $P_{A}$-Anosov representation for a symmetric parabolic
subgroup $P_{A}<G$, with limit curve $\xi: \partial_{\infty}\pi\rightarrow G/P_{A}$.  Suppose $I\subset W$ is a balanced ideal of type $(P_{A}, P_{D})$ with corresponding domain $\Omega_{\rho}^{I} \subset G/P_{D}$.
Let $n=\dim_{\C} G/P_D$. Then:
\begin{rmenumerate}
\item If $\hdim \xi(\partial_{\infty}\pi) <4$ and $G$ is not
  isomorphic to $\PSL_2\C$, then the domain $\Omega_{\rho}^{I}$ is
  non-empty.
\item If $\hdim \xi(\partial_{\infty}\pi) <6$ and $G$ is not isomorphic to types $A_{1}, A_{2}, A_{3}$ or $B_{2},$ then $\Omega_{\rho}^{I}$ is non-empty.
\item If
$\hdim \xi(\partial_{\infty}\pi) < 2 \left ( n - \max_{w\in I/W_{D}} \ell(w)\right)$,
then $\Omega_{\rho}^{I}$ is non-empty.
\end{rmenumerate}
\end{thm}

\begin{proof}
  For (iii), the assumption on $\hdim \xi(\partial_{\infty}\pi)$ is
  exactly what is needed so that Theorem~\ref{thm: hdimbound} gives
  $\hdim(\Lambda_\rho^I) < 2n = \hdim(G/P_D)$, so the complement of
  $\Lambda_\rho^I$ is non-empty.

  For (ii), by Theorem \ref{thm:short-small} the exclusion of these types gives $\max_{w\in I/W_{D}} \ell(w) \leq n-3$, and thus $2 \left ( n - \max_{w\in I/W_{D}} \ell(w)\right) \geq 6$.  Therefore this case follows from (iii).

  For (i), Theorem \ref{thm:short-small} similarly gives $2 \left ( n - \max_{w\in I/W_{D}} \ell(w)\right) \geq 4$ and hence the claim again follows from (iii).
\end{proof}

Note that the hypothesis $G \not \simeq \PSL_2\C$ in part (i) of
Theorem \ref{thm:general-nonempty} is necessary, as the example of a
cocompact lattice in $\PSL_2\C$ acting on $\CP^1$ with empty domain
of discontinuity shows.

Our main application of Theorem \ref{thm: hdimbound} will be to estimate the Hausdorff
dimension of limit sets for $G$-quasi-Fuchsian groups.  We find:
\begin{thm}\label{thm: QFbound}
Let $G$ be a complex simple Lie group of adjoint type and rank at least two
with Weyl
group $W$.  Let $\rho: \pi_{1}S \rightarrow G$ be a
$G$-quasi-Fuchsian representation and $I\subset W$ a balanced ideal of
type $(B, P_{D})$.  Let $n$ denote the complex dimension of
$G/P_{D}$.  Then the limit set $\Lambda_{\rho}^{I}\subset G/P_{D}$
satisfies
\[
m_{2n-2}(\Lambda_{\rho}^{I})=0.
\]
Furthermore, if $G$ is not of type $A_{2}, A_{3}$ or $B_{2},$ then
\[
m_{2n-4}(\Lambda_{\rho}^{I})=0.
\]
Here $m_k$ denotes the $k$-dimensional Hausdorff measure associated to
any Riemannian metric on $G/P_D$.
\end{thm}

\begin{proof}
By Theorem \ref{thm:short-small}, the hypotheses imply
$\max_{w\in I/W_{D}}\ell(w)\leq n-2$.  As the limit curve of a
quasi-Fuchsian group is a quasi-circle in $\CP^1$, its Hausdorff
dimension is strictly less than $2$.  By Theorem \ref{quasifuchsian
  Anosov}, the limit curve of a $G$-quasi-Fuchsian group is the image
of such a quasi-circle by the smooth embedding
$f_G : \CP^{1}\rightarrow G/P_D$, hence $\xi(\partial_\infty \pi_{1}S)$
also has Hausdorff dimension less than $2$.  Applying Theorem
\ref{thm: hdimbound} gives
\[
\hdim(\Lambda_{\rho}^{I})<2+2(n-2)=2n-2.
\]
and thus $m_{2n-2}(\Lambda_{\rho}^{I})=0$.

If we also exclude types $A_{2}, A_{3}$ and $B_{2}$, then Theorem
\ref{thm:short-small} gives $\max_{w\in I/W_{D}}\ell(w)\leq n-3$, and
proceeding as above we find $m_{2n-4}(\Lambda_{\rho}^{I})=0$.
\end{proof}
We note that, in particular, the domains in these cases considered in
Theorem \ref{thm:short-small} are always non-empty.

\section{Topology}\label{topology}

We now begin one of our central investigations of the paper---studying
the topology of the domains and quotient manifolds for
$G$-quasi-Hitchin representations.  We do this by first reducing to
the $G$-Fuchsian case (in Sections \ref{components}--\ref{families})
and then studying the Fuchsian case in Sections
\ref{subsec:homological-balance}--\ref{subsec:homology-m}.

\subsection{Anosov components}\label{components}
Let $\pi$ be a finitely generated group and $G$ a complex semisimple
Lie group.  By choosing a finite generating set of $\pi$, the set
$\Hom(\pi,G)$ can be identified with a complex affine subvariety of
$G^{N}$ for some $N \in \mathbb{N}$.  Thus $\Hom(\pi,G)$ has both the
Zariski topology and the compact-open topology of maps from the
discrete space $\pi$ to the manifold $G$, the latter of which we will
call the analytic topology.  Throughout this section, we use
\emph{component} to mean a connected component of a set with respect
to the analytic topology.

Let $P_A$ be a symmetric parabolic subgroup of $G$.  Given a
$P_A$-Anosov representation $\rho :\pi \rightarrow G,$ let
$\A(\rho, P_{A}) \subset \Hom(\pi, G)$ denote the connected
component of the set of $P_{A}$-Anosov representations that contains
$\rho$.  We call $\A(\rho,P_A)$ the \emph{Anosov component}
of $\rho$.

For example, the quasi-Hitchin set $\widetilde{\mathcal{QH}}(S,G,P_{A})$ for a complex
simple adjoint group $G$, as defined in
Section \ref{quasi-hitchin}, is equivalently described as the Anosov
component $\A(\rho,P_{A})$ of any $G$-Fuchsian representation $\rho :
\pi_{1}S \to G$.  

\subsection{Constant diffeomorphism type}\label{families}

Next we show that the diffeomorphism type of the compact quotient
manifold associated to a balanced ideal is constant on each Anosov
component:

\begin{thm}\label{thm:const-diffeo}
Let $P_A, P_D$ be parabolic subgroups of $G$, with $P_A$ symmetric,
and let $I \subset W$ be a balanced ideal of type $(P_A,P_D)$.  Let
$\rho : \pi \to G$ be a $P_A$-Anosov representation.  Then for any
$\rho' \in \A(\rho,P_A)$, the quotient manifolds $\W_\rho^I$ and
$\W_{\rho'}^I$ are diffeomorphic.
\end{thm}

In a similar spirit, in \cite{GW12} it was shown that the
homeomorphism type is constant on Anosov components for the quotients
of the domains of discontinuity constructed by those authors.  The
argument given there is quite general, however, and would also apply
in the present situation.  We give a detailed argument in order to
emphasize the smoothness of the resulting map.

In preparation for the proof, we define a \emph{smooth $1$-parameter
  family of representations} to be a collection
$\{ \rho_t \in \Hom(\pi,G) \suchthat t \in [0,1] \}$ such that for
each $\gamma \in \pi$ the map $[0,1] \to G$ defined by
$t \mapsto \rho_t(\gamma)$ is smooth.  This is equivalent to requiring
that $t \mapsto \rho_t$ defines a smooth map of $[0,1]$ into $G^N$
that takes values in the subvariety $\Hom(\pi,G) \subset G^N$.

\begin{lem}\label{lem:interval-family-smoothly-trivial}
Let $P_A$, $P_D$, and $I$ be as in Theorem \ref{thm:const-diffeo}.  If
$\rho$ is a smooth $1$-parameter family of representations, and if for
each $t \in [0,1]$ the representation $\rho_t : \pi \to G$ is
$P_A$-Anosov, then the quotient manifolds $\W_{\rho_0}^I$ and
$\W_{\rho_1}^I$ are diffeomorphic.
\end{lem}

\begin{proof}
First, the domains $\Omega_{\rho_t}^I$ can be assembled into a family; 
define the set $\Tilde{\V} \subset [0,1] \times G/P_D$ by
\[
\tilde{\V}:=\{ (t, x) \suchthat x \in \Omega_{\rho_t}^{I}\}.
\]
By Proposition~\ref{anosov prop}(iv) this is an open subset of
$[0,1]\times G/P_D$.  Let $\tilde{\Pi} : \tilde{V} \to [0,1]$ denote the
projection on to the first factor, so that
$\tilde{\Pi}^{-1}(t) = \{ t \} \times \Omega^I_{\rho_t}$.

The group $\pi$ acts smoothly and properly discontinuously on $\tilde{\V}$ by
\[
\gamma\cdot (t, x)=(t, \rho_t(\gamma)(x)).
\]
Let $\V:=\tilde{\V}/\pi$ denote the quotient by this action, which is
a smooth manifold (with boundary).  Since
$\tilde{\Pi}(\gamma \cdot (t,x)) = \tilde{\Pi}(t,x) = t$, there is an
induced smooth map $\Pi: \V \to [0,1]$ such that
$\Pi^{-1}(t) = \{ t \} \times \W_{\rho_t}^I$.  By compactness of
$\W_{\rho_t}^I$, the map $\Pi$ is proper.  Also, the map $\Pi$ is a
submersion, since its lift to the cover $\Tilde{V}$ is the projection
of the product manifold $[0,1]\times G/P_D$ onto its first
factor.

By Ehresmann's Lemma \cite{Ehresmann}, a proper smooth submersion is a
smoothly locally trivial fiber bundle.  Thus the fibers of $\Pi$ are
pairwise diffeomorphic.
\end{proof}

\begin{proof}[Proof of Theorem \ref{thm:const-diffeo}]
We abbreviate $\A = \A(\rho,P_A)$.  Recall that $\Hom(\pi,G)$ is a
complex affine algebraic variety, and by Proposition~\ref{anosov
  prop}(iv) we have that $\A$ is an open subset of $\Hom(\pi,G)$ in
the analytic topology.

Consider the equivalence relation on $\A$ given by diffeomorphism of
quotient manifolds, i.e.~$\rho' \sim \rho''$ if and only if
$\W_{\rho'}^I$ is diffeomorphic to $\W_{\rho''}^I$.  We will show
show that $\A$ consists of a single equivalence class.

First, let $H$ be an irreducible component of $\Hom(\pi,G)$ and let
$\B$ be a component of $\A \cap H$, so that $\B$ is a connected open
subset of $H$.  The singular locus $H^{\sing}$ of $H$ is a proper
algebraic subvariety, and its complement $H^{\smooth}$ is a connected
complex manifold that is dense in $H$. In the analytic topology, a
subvariety of an irreducible algebraic variety over $\C$ does not locally
separate, and so $\B \cap H^{\smooth}$ is also a connected complex
manifold.  Any two points of $\B \cap H^{\smooth}$ are therefore
joined by a smooth path, and Lemma
\ref{lem:interval-family-smoothly-trivial} shows that
$\B \cap H^{\smooth}$ lies in a single equivalence class.

By Milnor's Curve Selection Lemma \cite[Section~3]{Milnor68} for any
$x \in H^\sing$ there exists a smooth path $\gamma : [0,1] \to H$ so
that $\gamma(0) = x$ and $\gamma(t) \in H^{\smooth}$ for $t>0$.  Thus
for any $x \in \B \cap H^{\sing}$ we have such a path with
$\gamma(t) \in \B \cap H^{\smooth}$ for $0 < t \leq \epsilon$ (using
that $\B$ is open in $H$).  Applying Lemma
\ref{lem:interval-family-smoothly-trivial} to such paths, we find that
each $x \in \B \cap H^{\sing}$ lies in the same equivalence class as
$\B \cap H^{\smooth}$.  That is, $\B$ consists of a single equivalence
class.

Now for any point $x \in \A$, let $H_1, \ldots, H_k$ be the irreducible
components of $\Hom(\pi,G)$ that contain $x$.  The argument above
gives neighborhoods $\B_i$ of $x$ in $\A \cap H_i$ such that each
$\B_i$ lies in a single equivalence class.  Thus the union
$\bigcup_i \B_i$ also lies in a single equivalence class, and it
contains a neighborhood of $x$ in $\A$.

This shows that the equivalence classes in $\A$ are open.  Since $\A$
is connected, there is only one equivalence class.
\end{proof}

Since the set $\widetilde{\mathcal{QH}}(S,G,P_{A})$ is the Anosov
component of a $G$-Fuchsian representation (for $G$ simple and
adjoint), we have the immediate corollary:

\begin{cor}\label{cor: hitchin diff quotient}
The quotient manifold $\W_\rho^I$ obtained from any
$\rho \in \widetilde{\mathcal{QH}}(S,G,P_A)$ is diffeomorphic to the
corresponding quotient manifold for a $G$-Fuchsian representation.
\end{cor}

\subsection{Homology and cohomology of thickenings}
\label{subsec:homological-balance}
Starting toward our study of the topology of $G$-Fuchsian quotient
manifolds associated to a Chevalley-Bruhat ideal $I$, we
begin by considering the topology of the model thickening
$\Phi^I \subset G/P_D$.
\begin{lem}
\label{lem:homology-of-phi}
Let $I \subset W$ be a right $W_D$-invariant ideal.  
Then in the Schubert cell
basis for $H_*(G/P_D)$, the map 
\[
i : H_*(\Phi^I) \to H_*(G/P_D)
\] induced
by the inclusion $\Phi^I \into G/P_D$ corresponds to the natural
embedding of free abelian groups
\[
\Z_{I/W_D} \into \Z_{W/W_D}.
\]
\end{lem}

\begin{proof}
The model thickening $\Phi^I$ is a closed set that is a union of
Schubert cells, hence it is a subcomplex of the cell structure on
$G/P_D$. Using the labeling of cells by $W_D$-cosets, the natural map
$\Z_{I/W_D} \into \Z_{W/W_D}$ becomes the map on cellular chain
complexes induced by the inclusion of $\Phi^I$.  Since the boundary
maps of these chain complexes vanish identically (as there are no
odd-dimensional cells), this is naturally isomorphic to the induced
map on homology.
\end{proof}

Taking duals, Lemma~\ref{lem:homology-of-phi} identifies
the cohomology pullback map associated to the inclusion $\Phi^I \into
G/P_D$ with the natural surjective map $\Z^{W/W_D} \to \Z^{I/W_D}$.

Next, we show that the pair of orthogonal ideals $I,I^\perp$
corresponds naturally to a splitting of the homology $H_*(G/P_D)$ as a
direct sum.

\begin{lem}
\label{lem:homology-splitting}
For each right $W_D$-invariant ideal $I$ there is a split exact
sequence
\[
0 \to H_*(\Phi^I) \xrightarrow{i} H_*(G/P_D) \to
H^{2n-*}(\Phi^{I^\perp}) \to 0
\]
where $i$ is the map induced by $\Phi^I \into G/P_D$ and $n = \dim_\C G/P_D$.
\end{lem}

\begin{proof}
Splitting is automatic since $H^{2n-*}(\Phi^{I^\perp})$ is free
abelian (by the previous lemma).  To construct the exact sequence, let
$j : H_*(G/P_D) \to H^{2n-*}(\Phi^{I^\perp})$ denote the composition of
the Poincar\'e duality map with the pullback map on cohomology from
the inclusion $\Phi^{I^\perp} \to G/P_D$.  As a composition of an
isomorphism and a surjection (the latter using the previous lemma), we
see $j$ is itself surjective.  Its kernel consists of classes that are
orthogonal (with respect to the intersection pairing) to
$H_{2n-*}(\Phi^{I^\perp})$.  Identifying $H_{2n-*}(\Phi^{I^\perp})$ with
the subgroup $\Z_{I^\perp/W_D}$ of $\Z_{W/W_D}$, the description of the
intersection pairing from Section \ref{subsec:pd-flag} shows that this
subgroup pairs non-trivially with basis elements in $w_0 I^\perp$, and
is zero otherwise.  That is, the orthogonal is $\Z_{(W \setminus w_0I^\perp)/W_D}$.
Recalling that $I^\perp = w_0 (W \setminus I)$ and $w_0^2 = e$ we see
that this is simply $\Z_{I/W_D} \simeq i(H_*(\Phi^I))$ as required.
\end{proof}

We remark that this lemma essentially describes the (co)homologial consequence
of the disjoint union decomposition $(W/W_D) = (I/W_D) \sqcup
(w_0I^\perp/W_D)$.  In case $I$ is slim the description of the
intersection pairing on $G/P_D$ from Section \ref{subsec:pd-flag}
shows that the image of $H_*(\Phi^I)$ is an isotropic space for this
pairing (i.e.~the restriction of the intersection form vanishes
identically).  Therefore, for a \emph{balanced} ideal $I$ the exact
sequence of Lemma~\ref{lem:homology-splitting} represents an associated
``Lagrangian splitting'' of the homology $H_*(G/P_D)$.

\subsection{Homology of domains of proper discontinuity}
\label{subsec:omega-pd}
We now turn to the topology of domains $\Omega_\rho^I$.

\begin{thm}
\label{thm:omega-homology}
Let $G$ be a complex simple Lie group of adjoint type and let
$\rho\in \tilde{\QH}(S, G, P_A)\subset \Hom(\pi_{1}S,G)$.
If $I$ is a slim ideal of type $(P_A, P_{D})$ with associated model
thickening $\Phi^I$ and domain $\Omega_{\rho}^I \subset G/P_D$, then
there is a split short exact sequence
\begin{equation}
\label{exact seq domain}
0\rightarrow H^{2n-2-k}(\Phi^{I}, \Z)\rightarrow H_{k}(\Omega_{\rho}^{I}, \Z)\rightarrow  H_{k}(\Phi^{I^{\perp}}, \Z) \rightarrow 0
\end{equation}
where $n = \dim_\C G/P_{D}$.  In particular, the homology groups of $\Omega_{\rho}^{I}$ are free abelian.  In addition:
\begin{rmenumerate}
\item The odd homology groups of $\Omega_{\rho}^I$ vanish,
\item If $I$ is balanced, then the homology of $\Omega_{\rho}^I$ satisfies
\[ H_k(\Omega_{\rho}^I, \Z) \simeq H^{2n-2-k}(\Omega_{\rho}^I,
\Z).
\]
\end{rmenumerate}
\end{thm}

Observe that when applied to a balanced ideal $I$, this theorem
incorporates the results stated as Theorems
\ref{introthm:fuchsian-domain-homology} and
\ref{introthm:fuchsian-domain-properties} in the introduction, with
the exception of statement (iii) of Theorem
\ref{introthm:fuchsian-domain-properties}.

In the proof, we will omit the $\Z$-coefficients to simplify notation.

\begin{proof}
By Corollary~\ref{cor: hitchin diff quotient} it suffices to consider
the case when $\rho$ is $G$-Fuchsian.  Assume this from now on.
Poincar\'{e}-Alexander-Lefschetz duality yields a canonical
isomorphism
\begin{equation}
\label{PAL iso}
H^{2n-j}(G/P_D, \Lambda_{\rho}^{I})\simeq H_{j}(\Omega_{\rho}^{I}).
\end{equation}
Since the cohomology of $G/P_D$ vanishes in odd degrees, the long exact sequence
in cohomology of the pair $(G/P_D,\Lambda_{\rho}^I)$
decomposes into five-term sequences centered on the even degree
cohomology groups of $G/P_D$:
\begin{multline}
0\rightarrow H^{2n-2j-1}(\Lambda_{\rho}^I) \rightarrow H^{2n-2j}(G/P_D,\Lambda_{\rho}^I)
\rightarrow\\
\rightarrow H^{2n-2j}(G/P_D) \xrightarrow{*} H^{2n-2j}(\Lambda_{\rho}^I)\rightarrow H^{2n-2j+1}(G/P_D,\Lambda_{\rho}^I)
\rightarrow 0.
\label{exact sequence}
\end{multline}
Using Theorem \ref{thm:triviality}, the K\"unneth Theorem implies
$H^{2n-2j}(\Lambda_{\rho}^I) \simeq H^{2n-2j}(\Phi^I)$. 
Post-composing with this isomorphism, the map labeled
$(*)$ becomes the pullback map on cohomology of degree $(2n-2j)$
induced by the inclusion $\Phi^I \into G/P_D$.  Taking the dual of the
exact sequence from Lemma~\ref{lem:homology-splitting}, we find that
this map is surjective with kernel isomorphic to
$H_{2j}(\Phi^{I^\perp})$.

By the surjectivity of $(*)$ and the Poincar\'{e}-Alexander-Lefschetz 
isomorphism \eqref{PAL iso}, the exactness of \eqref{exact sequence} at the right implies that
\[
0=H^{2n-2j+1}(G/P_D,\Lambda_{\rho}^I)\simeq H_{2j-1}(\Omega_{\rho}^{I})
\]
which is statement (i) of the theorem.  Since the (co)homology of
$\Phi^{I}$ and $\Phi^{I^{\perp}}$ vanish in odd degrees (by Lemma
\ref{lem:homology-of-phi}), this also trivially verifies the existence
of the exact sequence \eqref{exact seq domain} when the degree is odd.

For even degrees, since the map labeled by $(*)$ has kernel isomorphic to $H_{2j}(\Phi^{I^\perp})$,
the five-term exact sequence restricts to a short exact sequence
\begin{equation}
\begin{split}
\label{ses1}
0\rightarrow H^{2n-2j-1}(\Lambda_{\rho}^I)&\rightarrow H^{2n-2j}(G/P_D,\Lambda_{\rho}^I)
\rightarrow H_{2j}(\Phi^{I^\perp})\rightarrow 0.
\end{split}
\end{equation}
The K\"unneth Theorem, Theorem \ref{thm:triviality},
and the vanishing of the odd-dimensional cohomology of $\Phi^I$ 
imply $H^{2n-2j-1}(\Lambda_{\rho}^I) \simeq H^{2n-2j-1}(\Phi^I \times
S^1) \simeq H^{2n-2j-2}(\Phi^I).$  Using
this isomorphism to replace the initial term in \eqref{ses1} and the Poincar\'{e}-Alexander-Lefschetz duality isomorphism \eqref{PAL iso} to replace the central term with
$H_{2j}(\Omega_{\rho}^{I})$ yields the desired short exact sequence
\[
0\rightarrow H^{2n-2j-2}(\Phi^I)\rightarrow H_{2j}(\Omega_{\rho}^{I})\rightarrow
H_{2j}(\Phi^{I^\perp})\rightarrow 0.
\]
Since $H_{2j}(\Phi^{I^\perp})$ is a free abelian group, the
sequence splits. 

Finally, statement (ii) follows immediately by taking the dual of the
exact sequence \eqref{exact seq domain} and applying the universal
coefficients theorem.
\end{proof}

As a corollary of this result, we find a simple formula for the Betti
numbers of the domain of discontinuity, which we state only for the case
when $I$ is balanced.  Note that Lemma~\ref{lem:homology-of-phi} shows that
$b_{2k}(\Phi)$ is the number of elements of $I/W_D$ of length $k$.
Thus if $I = I^\perp$, the theorem above gives:
\begin{cor}
\label{cor:omega-betti}
Under the hypotheses of Theorem \ref{thm:omega-homology}, if $I$ is a
balanced ideal, then the Betti numbers of the domain of discontinuity
in $G/P_D$ are given by
\[ b_{2k}(\Omega_\rho^I) = r_k + r_{n-1-k} 
\]
where $r_k$ is the number of elements of $I/W_D$ of length $k$ and $n
= \ell(w_0 W_D) = \dim_\C G/P_D$. \noproof
\end{cor}

As this corollary is statement (iii) of Theorem
\ref{introthm:fuchsian-domain-properties} , we have now completed the
proofs of Theorems \ref{introthm:fuchsian-domain-homology} and
\ref{introthm:fuchsian-domain-properties}.  Using the corollary above
to calculate the Euler characteristic of $\Omega_\rho^I$, we also
obtain:
\begin{cor}
\label{cor:chi-omega}
Under the hypotheses of Theorem \ref{thm:omega-homology}, if $I$ is a
balanced ideal, then the Euler characteristic of the domain of
discontinuity is given by
\[
\chi(\Omega^I_\rho) = \chi(G/P_D) = \card{W/W_D}.
\]
\end{cor}

\begin{proof}
Since $\Omega^I$ has only even-dimensional homology, the Euler
characteristic is the sum of its Betti numbers.  Using the formula of
Corollary \ref{cor:omega-betti}, each term $r_k$ appears twice in this
sum, hence $\chi(\Omega_{\rho}^I) = 2 \card{I/W_D}$.  Since a balanced
ideal satisfies $2 \card{I} = \card{W}$, a balanced
$W_D$-invariant ideal satisfies $2 \card{I/W_D} = \card{W/W_D}$, and
the desired formula for $\chi(\Omega_\rho^I)$ follows.
\end{proof}

\subsection{Homology of quotient manifolds}
\label{subsec:homology-m}
Next we show that Serre spectral sequence for the covering
$\Omega_\rho^I \to \W_\rho^I$ degenerates, yielding:
\begin{thm}
\label{thm:homology-product}
Let $G$ be a complex simple Lie group of adjoint type and
$\rho\in \tilde{\QH}(S, G, P_A)$ where $P_A < G$ is a symmetric
parabolic subgroup.

If $I$ is a balanced ideal of type $(P_{A}, P_{D})$ with
associated domain $\Omega_{\rho}^I \subset G/P_D,$ let $\W_{\rho}^{I}$
denote the compact quotient manifold.  Then, there is an isomorphism of
graded abelian groups
\[ H_*(\W_{\rho}^{I}, \Z) \simeq H_*(S, \Z) \tensor
H_*(\Omega_{\rho}^I, \Z). \]
\end{thm}

As in Corollary \ref{cor:omega-betti}, this shows
$H_k(\W_{\rho}^{I}, \Z)$ is free abelian for each $k$ and its rank is
computable from the combinatorial data of the ideal $I$ and the length
function $\ell$ on $W/W_D$.  Also, using Corollary \ref{cor:chi-omega}
we obtain the result stated in the introduction as Corollary \ref{introcor:euler}:

\begin{cor}
For $\W_{\rho}^{I}$ as above we have $\chi(\W_{\rho}^{I}) = \chi(S)
\chi(G/P_D)$, and so in particular $\chi(\W_{\rho}^{I})=(2-2g)\lvert W/W_{D}\rvert<0$
where $g\geq 2$ is the genus of $S.$
\noproof
\end{cor}
This corollary indicates the importance of the (co)homology calculation since we cannot 
distinguish the quotient manifolds for different choices of ideals $I\subset W$ using 
the Euler characteristic.

\begin{proof}[Proof of Theorem \ref{thm:homology-product}.]
As before, Corollary~\ref{cor: hitchin diff quotient} reduces the
statement to the case of $G$-Fuchsian $\rho$.  Let
$E^2_{p,q} = H_{p}(S, H_{q}(\Omega_{\rho}^{I}, \Z))$ denote the $E^2$
page of the Serre spectral sequence for homology of the regular
covering $\Omega_\rho^I \to \W_\rho^I$.  Since $S$ is a
$K(\pi_{1}S, 1)$, there is an isomorphism
\[
E^2_{p,q} \simeq H_{p}(\pi_{1}S, H_{q}(\Omega_{\rho}^{I}, \Z)_{\rho})
\]
where the right hand side is group homology, and where the $\pi_{1}S$-action
on $H_{q}(\Omega_{\rho}^{I}, \Z)$ is prescribed by $\rho$.
Furthermore, we claim
\begin{equation}
\label{eqn:e2-tensor}
H_p(\pi_{1}S, H_q(\Omega_{\rho}^I, \Z)_{\rho}) \simeq 
H_p(S, \Z) \tensor H_q(\Omega_{\rho}^I, \Z).
\end{equation}
which follows if we show $H_*(\Omega_{\rho}^I, \Z)$ is a trivial
$\pi_{1}S$-module.
However, by Proposition~\ref{PSL invariance} the domain
$\Omega_{\rho}^I$ associated to a $G$-Fuchsian representation is
invariant under the action of the real principal three-dimensional
subgroup $\iota_{G}(\PSL_2\R):=\mathfrak{S}_\R$ on $G/P_D$.  Since $\mathfrak{S}_\R$ is a
connected Lie group, the action of any element of this group on
$\Omega_{\rho}^I$ is homotopic to the identity and hence acts
trivially on $H_*(\Omega_{\rho}^I, \Z)$.  Since
$\rho(\pi_{1}S) \subset \mathfrak{S}_\R$, this gives the desired triviality of the
$\pi_{1}S$-module $H_*(\Omega_{\rho}^I, \Z)$.

Next, we claim that the spectral sequence degenerates at the
$E^{2}$-page.  First, from \eqref{eqn:e2-tensor} we find $E_{p,q}^{2} = 0$
if $p > 2$ (since $S$ has real dimension $2$) or if $q$ is odd (by vanishing of odd
homology of $\Omega_{\rho}^I$).  The condition on $p$ leaves the
$E^2$-differentials $\partial^2_{p,q} : E_{p,q}^{2} \to E_{p-2,q+1}^{2}$
as the only potentially non-trivial maps, however these change the
parity of $q$ and hence either the domain or codomain is trivial.
Thus all differentials vanish at the $E^2$-page.

Finally, since all groups on the $E^{2}$-page are free abelian (which
follows from the homology of both $\Omega_{\rho}^{I}$ and $S$ being free
abelian), there is no extension problem to solve and we conclude that
$H_*(\W_{\rho}^{I}, \Z)$ is isomorphic to the total complex
of the $E^2$-page, which by \eqref{eqn:e2-tensor} is simply
$H_*(S, \Z) \tensor H_*(\Omega_{\rho}^I, \Z)$.
\end{proof}

\section{Complex geometry}\label{complex geometry}
In this section, we will study some fundamental features of the
complex geometry of the manifolds $\W_{\rho}^I$ arising from quotients
of domains in flag varieties by images of Anosov representations.  As
mentioned in the introduction, it is natural to work in a slightly
more general setting.

Recall that if $N = G/H$ is a complex homogeneous space of $G$, then
we say a complex manifold $\W$ is a uniformized $(G,N)$-manifold with data
$(\Omega, \Gamma)$ and limit set $\Lambda:=N\setminus \Omega$
if $\Gamma < G$ acts freely, properly
discontinuously, and cocompactly on $\Omega \subset N$ and there is
a biholomorphism $\W \simeq \Gamma \backslash \Omega$.  For example,
 if $\rho:\pi \rightarrow G$ is
$P_A$-Anosov (with $\pi$ torsion-free) and $I$ is a balanced ideal of type $(P_A,P_D)$, then the
manifold $\W_\rho^I$ is a uniformized $(G,G/P_D)$-manifold with data
$(\Omega_\rho^I, \rho(\pi))$ and limit set $\Lambda_{\rho}^{I}$.
 
\subsection{Non-existence of K\"ahler metrics and maps to Riemann surfaces}
Let $m_\alpha$ denote the $\alpha$-dimensional Hausdorff measure on
$N$ associated to any Riemannian metric.  As in Section \ref{size
  limit set} the particular metric will not matter.

The following classical extension theorem in several complex variables
is due to Shiffman:
\begin{thm}[{\cite[Lemma 3]{SHI68}}]
\label{shiffman}
Let $Z$ be a complex manifold of dimension $n$ and $A \subset Z$ a
closed set satisfying $m_{2n-2}(A) = 0$.  Then any holomorphic
function on $Z \setminus A$ extends to a unique holomorphic function
on $Z$. \noproof
\end{thm}

An immediate consequence of this extension theorem adapted to our
situation is:
\begin{lem}
\label{lem:adapted-shiffman}
Let $\W$ be a uniformized $(G,N)$-manifold with data
$(\Omega, \Gamma)$ and limit set $\Lambda$.
Suppose that $N$ is compact and connected, and that $m_{2n-2}(\Lambda)=0$ where
$n=\dim_{\C}N$.  Then any holomorphic map $\Omega \to \C^k$ is
constant. \noproof
\end{lem}

Using this theorem, we now prove Theorem \ref{introthm:embedded-no-holo}
from the introduction.  We recall the statement:

\embeddednoholo*

\begin{proof}
By the Koebe-Poincar\'{e} uniformization theorem, a Riemann surface
$X \not \simeq \CP^1$ has universal cover biholomorphic to a domain
in $\C$, so it suffices to prove the second assertion.

Since $N$ is $1$-connected, the condition $m_{2n-2}(\Lambda)=0$
implies that $\Omega$ is also $1$-connected (see e.g.~\cite[Chapter 7]{HW41})
and hence is biholomorphic to the universal cover of $\W$.  Using 
the Hausdorff dimension assumption again,
Lemma~\ref{lem:adapted-shiffman}
shows that every holomorphic map $\Tilde{\W} \to \C^k$ is constant.

If $Y$ is a complex manifold whose universal cover is biholomorphic
to a domain in $\C^k$, then lifting a holomorphic map $f : \W \to Y$
to the universal covers gives a map
$\Tilde{f} : \Tilde{\W} \to \Tilde{Y} \subset \C^k$, which is
therefore constant, and $f$ is constant as well.
\end{proof}

Next, we establish the obstruction to the existence of K\"ahler
metrics which was stated in the introduction:

\embeddednotkahler*

\begin{proof}
As in the preceding proof, we conclude that $\Tilde{\W} \simeq \Omega$
has no non-constant holomorphic maps to $\C^k$.  However, Eyssidieux
shows in \cite{eyssidieux} that if the fundamental group of a compact
\emph{K\"{a}hler} manifold has an infinite linear quotient, then its
universal cover admits a non-constant map to $\C^k$ for some $k$.  Therefore
$\W$ is not K\"{a}hler.
\end{proof}

Applying these theorems to the study of manifolds which are quotients
by $G$-quasi-Fuchsian groups and using the Hausdorff dimension bounds
of Section \ref{size limit set}, we now give the proof of:

\anosovquotientnegatives*

Note that for consistency of notation with the introduction, we are
now considering the parabolic pair $(P_A,P_D) = (B,P)$.

\begin{proof}
By Theorem \ref{thm: QFbound}, for such $\rho$ and $I$ the limit set
satisfies $m_{2n-2}(\Lambda_\rho^I)=0$.  The flag variety $G/P$ is
compact and $1$-connected.  Thus, statement (i) follows from Theorem
\ref{introthm:embedded-no-holo}.

Since $\Omega_{\rho}^{I} \to \W^I_\rho$ is a $\pi_1S$-covering, we
have a surjection $\pi_1 \W^I_\rho \to \pi_1S$.  Since $\pi_1S$ is an
infinite linear group, statement (ii) follows from Theorem
\ref{introthm:embedded-not-kahler}.
\end{proof}

\subsection{Picard group}

The following theorem of Harvey is an analogue of Shiffman's extension
theorem (Theorem \ref{shiffman}) for holomorphic line bundles and
their cohomology:
\begin{thm}[{\cite[Theorems 1 and 4]{HAR74}}]\label{harvey}
Let $Y$ be a complex manifold of dimension $n$ and $A \subset Y$ a
closed subset satisfying $m_{2n-4}(A)=0$.  Then, every holomorphic
line bundle $L \rightarrow (Y \setminus A)$ extends uniquely to a
holomorphic line bundle on $Y$.

Furthermore, if $m_{2n-2k-2}(A)=0$, then the inclusion map
$(Y \setminus A) \into Y$ induces an isomorphism
\[ H^i(Y,L) \to H^i(Y \setminus A,L) \]
for all $0 \leq i \leq k$.
\end{thm}

Let $\W$ be a uniformized $(G,N)$-manifold with data $(\Omega, \Gamma).$
A line bundle $\mathcal{L}$ on $N$ is \emph{$\Gamma$-equivariant} if
it carries an action of $\Gamma$ by bundle automorphisms lifting the
action of $\Gamma$ on $N$.

Let $p : \Omega \to \Omega/\Gamma \simeq \W$ be the covering map.
Given a $\Gamma$-equivariant line bundle $\mathcal{L}$ on $N$, there
is a naturally associated line bundle $p_*^\Gamma \mathcal{L}$ on $\W$
which, as a sheaf, is defined by setting $p_*^\Gamma \mathcal{L}(U)$
to be the space of $\Gamma$-invariant sections of
$\restrict{\mathcal{L}}{p^{-1}(U)}$.  This prescription defines the
\emph{invariant direct image} homomorphism
\begin{equation}
\label{descent}
p_{*}^{\Gamma}: \Pic^{\Gamma}(N)\rightarrow \Pic(\W)
\end{equation}
where $\Pic(\W)$ is the Picard group of isomorphisms classes of
holomorphic line bundles on $\W$, and where $\Pic^{\Gamma}(N)$ is the
group of $\Gamma$-equivariant isomorphism classes of
$\Gamma$-equivariant line bundles on $N$.

Using Theorem \ref{harvey} we obtain a sufficient condition 
for the homomorphism \eqref{descent} to admit a section:
\begin{prop} \label{class lb} Let $\W$ be a uniformized $(G,N)$-manifold
with data $(\Omega, \Gamma)$ and limit set $\Lambda$.  Suppose that
$m_{2n-4}(\Lambda)=0$ where $n=\dim_{\C}N$.  Then for any
holomorphic line bundle $L$ on $\W$, we have:
\begin{rmenumerate}
\item The pullback of $L$ to $\Omega$ extends uniquely to a $\Gamma$-equivariant line bundle on $N$.  
\item If $N$ is compact and connected, and if the pullback of $L$ to
$\Omega$ is holomorphically trivial, then
$L\simeq \Omega\times_{\chi} \mathbb{C}$ where
$\chi: \Gamma\rightarrow \C^{*}$ is a homomorphism.
\end{rmenumerate}
\end{prop}

\begin{proof}
As before let $p : \Omega \to \W$ denote the quotient by $\Gamma$.
Under the given hypotheses, Theorem \ref{harvey} shows that
$p^{*}L$ extends uniquely to a holomorphic line bundle
$\mathcal{L}$ on $N$.  By the uniqueness of the extension,
$\mathcal{L}$ is $\Gamma$-equivariant, and (i) follows.

Suppose $p^{*}L$ is holomorphically trivial.  Then the canonical $\Gamma$-action
on $p^{*}L$ is transported by the trivialization to a
holomorphic function $q_\gamma : \Omega \to \C^*$.  By Shiffman's
extension theorem (Theorem \ref{shiffman}) $q_\gamma$ extends
holomorphically to $N$.  Therefore, if $N$ is compact and connected,
this map is constant. Thus the map
$\chi : \Gamma \to \C^*$, $\chi(\gamma) = q_\gamma$ is a homomorphism
such that $L \simeq \Omega \times_\chi \C$, and (ii) follows.
\end{proof}

Using the previous theorem, we can now establish the classification of
holomorphic line bundles on uniformized $(G, G/P)$-manifolds with
sufficiently ``small'' limit sets which was given in the introduction;
we recall the statement:

\picard*

\begin{proof}
Let $L$ be a holomorphic line bundle on $\W$. By Proposition
\ref{class lb}(i), the pullback $p^{*}L$ extends to a
$\Gamma$-equivariant holomorphic line bundle $\mathcal{L}$ on $G/P$.
It is easily checked that the \emph{lift-extend} map
$\textnormal{Pic}(\mathcal{W})\rightarrow
\textnormal{Pic}^\Gamma(G/P)$ thus constructed is a homomorphism.
Since $p^{*}\circ p_{*}^{\Gamma}(\mathcal{L})=\mathcal{L}$, the
lift-extend homomorphism is surjective and split by the invariant
direct image.

Next, suppose $\mathcal{L}$ is a $\Gamma$-equivariant line bundle on $G/P$ and 
$\phi: \mathcal{L}\rightarrow G/P\times \mathbb{C}$ is an isomorphism.  Then, there exists
a holomorphic automorphic function $j: \Gamma\times G/P\rightarrow \mathbb{C}^{*}$ 
and a $\Gamma$-action on $G/P\times \mathbb{C}$ specified by $\gamma\cdot
(x, v)=(\gamma\cdot x, j(x,\gamma)v)$ for which $\phi$ is $\Gamma$-equivariant.
Since $G/P$ is compact and connected, $j(-, \gamma): G/P\rightarrow \mathbb{C}^{*}$ is constant, and therefore $j\in \textnormal{Hom}(\Gamma, \mathbb{C}^{*})$ is a character.  This proves that
the kernel of \eqref{Picard kernel} contains $\textnormal{Hom}(\Gamma, \mathbb{C}^{*}).$

Finally, if $\chi\in \textnormal{Hom}(\Gamma, \mathbb{C}^{*})$, then $p^{*}(\Omega\times_{\chi} \mathbb{C})\simeq \Omega\times \mathbb{C}$, and therefore $\textnormal{Hom}(\Gamma, \mathbb{C}^{*})$ contains the kernel of $\eqref{Picard kernel},$ completing the proof.
\end{proof}

The term $\Pic^\Gamma(G/P)$ appearing in Theorem
\ref{introthm:picard} is often easy to compute in practice.  For
example, if $G$ is simply connected then every line bundle on $G/P$ is
$G$-equivariant, and hence $\Gamma$-equivariant by restriction.  In this case, 
there is a short exact sequence
\begin{equation}\label{pic exact seq}
1\rightarrow \textnormal{Hom}(\Gamma, \mathbb{C}^{*})\rightarrow \Pic(\W)\rightarrow \Pic(G/P)\rightarrow 1,
\end{equation}
which is split by the invariant direct image.

Finally, we prove Theorem \ref{introthm:quasi-fuchsian-picard} from the introduction.
\picardqf*

\begin{proof}
Any quasi-Fuchsian representation
$\eta: \pi_{1}S\rightarrow \PSL(2, \mathbb{C})$ can be lifted to a
representation
$\widetilde{\eta}: \pi_{1}S\rightarrow \SL(2, \mathbb{C})$ (see
e.g.~\cite{culler}).  Such a lift $\widetilde{\eta}$ determines lift
$\widetilde{\rho}: \pi_{1}S\rightarrow \widetilde{G}$ of $\rho$, where
$\widetilde{G}$ is the simply connected cover of $G$.

The covering map $\widetilde{G}\rightarrow G$ induces an equivariant
biholomorphic map $\widetilde{G}/\widetilde{P}\simeq G/P$ where
$\widetilde{P}<\widetilde{G}$ is the corresponding parabolic subgroup.
Therefore, if
$\widetilde{\Omega}_{\widetilde{\rho}}^{I} \subset
\widetilde{G}/\widetilde{P}$ is the corresponding domain whose
quotient by $\widetilde{\rho}(\pi_{1}S)$ is denoted
$\widetilde{\W}_{\widetilde{\rho}}^{I},$ then there is an induced
biholomorphic map
$\widetilde{\W}_{\widetilde{\rho}}^{I}\simeq \W_{\rho}^{I}.$

By Theorem \ref{thm: QFbound}, the exclusion of types $A_1$, $A_2$,
$A_3$, and $B_2$ guarantees that the hypotheses of Theorem
\ref{introthm:picard} are met.  Hence, by \eqref{pic exact seq} and
Theorem \ref{introthm:picard} there is an exact sequence
\[
1\rightarrow \textnormal{Hom}(\pi_1S, \mathbb{C}^{*})\rightarrow \Pic(\widetilde{\W}_{\tilde{\rho}}^I)\rightarrow \Pic(\tilde{G}/\tilde{P})\rightarrow 1.
\]
Since $\widetilde{G}/\widetilde{P}\simeq G/P$ and
$\widetilde{\W}_{\widetilde{\rho}}^{I}\simeq \W_{\rho}^{I}$, this
gives the desired exact sequence.
\end{proof}

\subsection{Cohomology of holomorphic line bundles}\label{cohomology line bundles}

Next we consider the calculation of cohomology of line bundles on
uniformized $(G, G/P)$-manifolds where $G$ is a connected complex semisimple Lie
group.  We will restrict to the case $P=B$ to
simplify the discussion.

Our results are based on reducing these calculations to the
Borel-Bott-Weil theorem, whose statement we recall before proceeding.
Fix a Cartan subalgebra $\h \subset \g$ and a system of simple roots
$\Delta \subset \h^*$; let $L \subset \h^*$ denote the lattice of algebraically integral
weights and $\weylvec \in \h^*$ half the sum of the positive roots.  Finally, let 
$L^{\textnormal{an}}\subset L$ denote the sub-lattice of analytically integral weights
consisting of those $\lambda\in L$ which integrate to a character $\Tilde{\lambda}:B\rightarrow \C^{*}.$  Note that $L^{\textnormal{an}}=L$ if $G$ is simply connected.

To each $\lambda \in L^{\textnormal{an}}$ there is an associated right
action of $B$ on $G \times \C$ given by
$(g,t)\cdot b = (gb, \Tilde{\lambda}(b) t)$.  We denote by
$\L_\lambda$ the quotient of $G \times \C$ by this action of $B$.  The
projection $G \times \C \to G$ is $B$-equivariant and hence descends
to a map $\pi : \L_\lambda \to G/B$, which gives $\L_\lambda$ the
structure of a $G$-equivariant holomorphic line bundle over $G/B$.
Define $\mathcal{L}^\lambda := \mathcal{L}_{\weylvec - \lambda}$.

The co-roots $\{H_{\alpha}\}_{\alpha\in \Delta}\subset \h$ are
elements uniquely defined by the conditions
$H_{\alpha}\in [\g_{-\alpha}, \g_{\alpha}]$ and
$\alpha(H_{\alpha})=2.$ A weight $\lambda\in L$ is \emph{dominant} if
$\lambda(H_{\alpha}) \geq 0$ for all $\alpha \in \Delta$,
\emph{strictly dominant} if $\lambda(H_{\alpha}) > 0$ for all
$\alpha \in \Delta$, and \emph{regular} if its $W$-orbit contains a
strictly dominant weight.

The Borel-Bott-Weil theorem is the following:
\begin{thm}[\cite{BOT57}]
\label{thm:BBW}
The map $\lambda \mapsto \mathcal{L}_\lambda$ is an isomorphism of abelian
groups $L^{\textnormal{an}} \simeq \Pic^{G}(G/B)$.  Furthermore, the cohomology of
$\mathcal{L}^\lambda$ satisfies:
\begin{rmenumerate}
\item If $\lambda$ is not regular, then $H^{i}(G/B,
\mathcal{L}^{\lambda})=0$ for all $i\geq 0$.  
\item If $\lambda$ is regular, let $w \in W$ be the unique element
such that $w(\lambda)$ is strictly dominant.  Then
$H^i(G/B,\mathcal{L}^{\lambda}) =0$ for all $i \neq \ell(w)$, while
$H^{\ell(w)}(G/B,\mathcal{L}^{\lambda}) \neq 0$ and as a $G$-module
this cohomology space is dual to the irreducible representation of $G$
with highest weight $w(\lambda)-\weylvec.$  \noproof
\end{rmenumerate}
\end{thm}
Expositions of this Theorem and associated background material can be
found in \cite{baston-eastwood}, \cite{jantzen} (focusing on algebraic
groups), or \cite{sepanski} (focusing on compact groups).

Returning to our discussion of a uniformized $(G,G/B)$ manifold $\W$, we
can cast the problem of determining cohomology of a line bundle on
$\W$ in the more general framework of relating the cohomology of a
locally free sheaf $\mathcal{F}$ on $Y$ and that of the pullback
$p^*\mathcal{F}$ to the universal cover $\tilde{Y}$.  Here the
Grothendieck spectral sequence (\cite{GRO57}) can be applied to the
composition of the $\Gamma$-invariants and global sections functors,
giving a cohomology spectral sequence with $E_2$-page
\begin{equation}
\label{spec sequence}
E_2^{p,q} = H^{p}(\Gamma, H^{q}(\tilde{Y}, p^{*}\mathcal{F}))
\end{equation}
and which converges to the cohomology of $\mathcal{F}$. Using this
spectral sequence, we show:
\begin{thm}\label{thm: sheaf cohomology}
Let $G$ be a connected semisimple complex Lie group, $B<G$ a
Borel subgroup, and $\W$ a uniformized $(G,G/B)$-manifold with data
$(\Omega, \Gamma)$ and limit set $\Lambda$.  Suppose that
$m_{2n-2k-2}(\Lambda)=0$ where $n=\dim_{\C}G/B$ and $k\geq 1$.

Let $\lambda \in L^{\textnormal{an}}$ be an algebraically integral weight and let
$p_{*}^{\Gamma}: \Pic^{G}(G/B)\rightarrow \Pic(\W)$ denote the invariant
direct image functor.  
\begin{rmenumerate}
\item \label{not-regular} If $\lambda$ is not regular, then
$H^{i}(\W, p_{*}^{\Gamma}(\mathcal{L}^{\lambda}))=0$ for all
$0\leq i< k.$

\item \label{regular-long} If $\lambda$ is regular and $w(\lambda)$ is
dominant for $w\in W$ with $\ell(w)>k$, then
$H^{i}(\W, p_{*}^{\Gamma}(\mathcal{L}^{\lambda}))=0$ for all $0\leq i< k.$

\item \label{regular-short} If $\lambda$ is regular and $w(\lambda)$
is dominant for $w\in W$ with $\ell(w)< k$, then 
\[
H^{i}(\W, p_{*}^{\Gamma}(\mathcal{L}^{\lambda}))\simeq
\begin{cases}
0 & 0\leq i<\ell(w) \\
H^{i-\ell(w)}\left(\Gamma, H^{\ell(w)}(G/B, \mathcal{L}^{\lambda})\right) & \ell(w)\leq i < k \\
\end{cases}
\]
In particular, the group
\[
H^{\ell(w)}(\W, p_{*}^{\Gamma}(\mathcal{L}^{\lambda}))\simeq
H^{0}\left(\Gamma, H^{\ell(w)}(G/B, \mathcal{L}^{\lambda})\right)
\]
is equal to the space of $\Gamma$-invariants in the dual of the
irreducible $G$-representation with highest weight $w(\lambda)-\weylvec$.

\item \label{dominant} In particular, if $\lambda$ is a regular,
dominant weight then
\[
H^{i}(\W, p_{*}^{\Gamma}(\mathcal{L}^{\lambda}))\simeq
H^{i}\left(\Gamma, H^{0}(G/B, \mathcal{L}^{\lambda})\right)
\]
for all $0\leq i< k$.
\end{rmenumerate}
\end{thm}

Note that statement \ref{dominant} of this theorem is exactly Theorem
\ref{introthm:cohomology-line-bundles} from the introduction, since
effective $G$-equivariant line bundles on $G/B$ are exactly those of the form
$\mathcal{L}^\lambda$ for regular, dominant $\lambda\in L^{\textnormal{an}}$.

\begin{proof}
By Harvey's extension theorem (Theorem \ref{harvey}), the
hypothesis on Hausdorff dimension gives an isomorphism
\[
H^{i}(G/B, \mathcal{L}^{\lambda})\simeq H^{i}(\Omega, \mathcal{L}^{\lambda})
\]
for all $0\leq i\leq k.$ Since $k\geq 1$, the same hypothesis ensures
that $\Omega$ is simply connected, and thus is the universal cover of
$\W$.  Thus the spectral sequence \eqref{spec sequence} applies and 
its $E_2$-page is determined up to the $k$-th row:

\hspace{-14mm}\begin{tikzpicture}
\matrix (m) [matrix of math nodes,
             nodes in empty cells,
             nodes={minimum width=10ex,
                    minimum height=10ex,
                    outer sep=-5pt},
             column sep=1ex, row sep=1ex,
             text centered,anchor=center]{
    k   &  H^{0}(\Gamma, H^{k}(G/B, \mathcal{L}^{\lambda})) &  H^{1}(\Gamma, H^{k}(G/B, \mathcal{L}^{\lambda}))  & \cdots & H^{\cohdim(\Gamma)}(\Gamma, H^{k}(G/B, \mathcal{L}^{\lambda}))  \\
    \vdots&  \vdots   &  \vdots         & \vdots         & \vdots      \\                   
     1      &  H^{0}(\Gamma, H^{1}(G/B, \mathcal{L}^{\lambda}))  & H^{1}(\Gamma, H^{1}(G/B, \mathcal{L}^{\lambda})) &  \cdots  & H^{\cohdim(\Gamma)}(\Gamma, H^{1}(G/B, \mathcal{L}^{\lambda}))     \\
   0      &  H^{0}(\Gamma, H^{0}(G/B, \mathcal{L}^{\lambda}))    & H^{1}(\Gamma, H^{0}(G/B, \mathcal{L}^{\lambda})) & \cdots  & H^{\cohdim(\Gamma)}(\Gamma, H^{0}(G/B, \mathcal{L}^{\lambda}))   \\
  \quad\strut &  0 &  1  &  \cdots &  \cohdim(\Gamma)  \\};
\draw[thick] (m-1-1.north east) -- (m-5-1.east) ;
\draw[thick] (m-5-1.north) -- (m-5-5.north east) ;
\end{tikzpicture}

\noindent Here $\cohdim(\Gamma) \in \Z^{\geq 0}$ denotes the cohomological
dimension of $\Gamma$; by definition of this integer, entries in the
$E_2$ page to the right of those indicated here are zero.  Meanwhile,
entries above the $k$-th row involve groups of the form
$H^{j}(\Omega, \mathcal{L}^{\lambda})$ we do not know how to compute.

The entire proposition now follows simply by applying the
Borel-Bott-Weil theorem.  For instance, if $\lambda$ is not regular,
then all the coefficients appearing in the above rectangle of the
$E_{2}$-page vanish, which immediately yields statement
\ref{not-regular}.  The same is true if $\lambda$ is regular, but the
$w\in W$ such that $w(\lambda)$ is dominant satisfies $\ell(w)>k$,
from which statement \ref{regular-long} follows.

In the case that $\ell(w) < k$, only the $\ell(w)$-th row is non-zero,
so all relevant differentials are zero.  Using the description of
the entries in this row from the Borel-Bott-Weil theorem, statements
\ref{regular-short} and \ref{dominant} follow.  This completes the
proof.
\end{proof}

We now explain a connection between these computations and classical
questions in geometric invariant theory (a theme which is also
explored in \cite{KLP13} and \cite{ST15}).  Note that the complex semisimple
group $G$ is an affine algebraic group over $\C$. For a $G$-equivariant
line bundle $\mathcal{L},$ the representation
$\nu$ of $G$ on $H^{0}(G/B, \mathcal{L})$ is a rational
representation.  Therefore, given a subspace
$V\subset H^{0}(G/B, \mathcal{L})$, its stabilizer
\[
\{g\in G \suchthat \nu(g)s-s=0 \text{ for all } s\in V\}
\]
is Zariski closed.  We record this in the following proposition.

\begin{prop}
\label{prop:stab-zariski-closed}
Let $G$ be a connected complex semisimple Lie group and
$\lambda \in L^{\textnormal{an}}$ a regular dominant weight.  If $\Gamma<G$ is a subgroup with
Zariski closure $Q<G,$ then
\[
H^{0}\left(\Gamma, H^{0}(G/B, \mathcal{L}^{\lambda})\right)
=H^{0}\left(Q, H^{0}(G/B, \mathcal{L}^{\lambda})\right),
\]
where the right hand side is the space of $Q$-invariant sections of $\mathcal{L}^{\lambda}.$\noproof
\end{prop}
This leads to the following result.

\begin{thm}\label{thm: zariski closure}
Let $G$ be a connected semisimple complex Lie group, $B<G$ a
Borel subgroup, and $\W$ a uniformized $(G,G/B)$-manifold with data
$(\Omega, \Gamma)$ and limit set $\Lambda$.  Let $Q < G$ denote the
Zariski closure of $\Gamma$.  Suppose that $m_{2n-4}(\Lambda)=0$
where $n=\dim_{\C} G/B$.

Let $\lambda \in L^{\textnormal{an}}$ be a regular dominant weight and let
$p_{*}^{\Gamma}: \Pic^{G}(G/B)\rightarrow \Pic(\W)$ denote the invariant
direct image homomorphism.  Then:
\[
H^{0}(\W, p_{*}^{\Gamma}(\mathcal{L}^{\lambda})) \simeq H^{0}\left(Q, H^{0}(G/B, \mathcal{L}^{\lambda})\right),
\]
where the latter is the space of $Q$-invariant sections.  In
particular, if $\Gamma$ is Zariski dense in $G$, then $H^{0}(\W,
p_{*}^{\Gamma}(\mathcal{L}^{\lambda})) = 0$.
\end{thm}

\begin{proof}
The isomorphisms
\[
H^{0}(\W, p_{*}^{\Gamma}(\mathcal{L}^{\lambda})) \simeq
H^{0}\left(\Gamma, H^{0}(G/B, \mathcal{L}^{\lambda})\right) = H^{0}\left(Q, H^{0}(G/B, \mathcal{L}^{\lambda})\right)
\]
follow from  Theorem 
\ref{thm: sheaf cohomology} and Proposition
\ref{prop:stab-zariski-closed}, respectively.  If $Q=G,$ then the
irreducibility $H^{0}(G/B, \mathcal{L}^{\lambda})$ as a
$G$-representation implies that the space of $G$-invariants is
trivial.
\end{proof}

In the ensuing applications, we will give explicit examples where
$H^{0}(\W, p_{*}^{\Gamma}(\mathcal{L}^{\lambda}))$ is non-vanishing.

\subsection{Applications}
\label{subsec:meromorphic}
We will now present some applications of the previous calculations: in
particular we show that, excluding some low dimensional cases, every
manifold arising from a $G$-quasi-Fuchsian representation admits a
meromorphic function.  In this section, we will return to the notation
$\mathcal{L}_{\lambda}=G\times_{\lambda}\C$ and note that
$\mathcal{L}_{k\lambda}=\mathcal{L}_{\lambda}^{k}$ where the latter
is the $k$-th tensor power.
Given a subgroup $H<G$, we say that
$\mathcal{L}_{\lambda}$ is \emph{twice $H$-ample} if some power
$\mathcal{L}_{k\lambda}$ admits a pair of non-proportional
$H$-invariant sections.

We begin with the following, which follows quickly from results in \cite{ST15}.
\begin{thm}\label{S-ample}
Let $G$ be an adjoint complex simple Lie group not of type $A_{1}$,
$A_{2}$, or $B_{2}$ with principal three-dimensional embedding
$\iota_{G}: \PSL_2\C \rightarrow G$.  Let
$\mathfrak{S} = \iota_G(\PSL_2\C)$.  Then every ample, $G$-equivariant line bundle
$\mathcal{L}$ on $G/B$ is
twice $\mathfrak{S}$-ample.
\end{thm}

\begin{proof}
First, recall that ample, $G$-equivariant line bundles on $G/B$ are of the form $\mathcal{L}_{-\lambda}$
for $\lambda\in L^{\textnormal{an}}$ some regular, dominant weight.
Consider the graded ring $R(\lambda)=\bigoplus_{k>0} H^{0}(G/B, \mathcal{L}_{-k\lambda})$
and the subring $R(\lambda)^{\mathfrak{S}}$ of $\mathfrak{S}$-invariant elements.  
Define the subset $Y(\lambda)\subset G/B$ by
\[
Y(\lambda):=\{x\in G/B \suchthat s(x)=0\ \text{for every}\ s\in R(\lambda)^{\mathfrak{S}} \}.
\]
Under the hypotheses, it is shown in \cite{ST15} that the complex
codimension of $Y(\lambda)$ is at least two.  Since the vanishing
locus of a non-zero holomorphic section has complex codimension one,
this implies that there exists a pair of $\mathfrak{S}$-invariant
sections $s_{i}\in H^{0}(G/B, \mathcal{L}_{-k_{i}\lambda})$ for $i=1,2$
with distinct vanishing loci.  Then $s_1^{k_2}$ and $s_2^{k_1}$ are
non-proportional sections of $\mathcal{L}_{-(k_1+k_2)\lambda}$.
\end{proof}

Specializing now to the case of $G$-quasi-Fuchsian representations,
this leads to a proof of the following theorem stated in the
introduction:
\meromorphic*

\begin{proof}
As before we have $\mathcal{L} \simeq \mathcal{L}_{-\lambda}$ for
$\lambda\in L^{\textnormal{an}}$ regular and dominant.  By Theorem \ref{thm: QFbound},
the exclusion of types $A_{1}$, $A_{2}$, $A_{3}$ and $B_{2}$ implies
that $m_{2n-4}(\Lambda_\rho^I)=0$.  By Theorem \ref{thm: sheaf
  cohomology}(iv),
\[
H^{0}(\W_{\rho}^{I}, p_{*}^{\Gamma}(\mathcal{L}_{-k\lambda}))\simeq H^{0}(\Gamma, H^{0}(G/B, \mathcal{L}_{-k\lambda}))
\]
which is non-vanishing for some $k>0$ by Theorem \ref{S-ample}.  Here,
we have used that every $\mathfrak{S}$-invariant section is
$\Gamma$-invariant (since $\Gamma \subset \mathfrak{S}$).  Thus
statement (i) follows.

By Theorem \ref{S-ample} the ample bundle $\mathcal{L}_{-\lambda}$ is twice
$\Gamma$-ample, thus there exists $k>0$ such that
$\mathcal{L}_{-k\lambda}$ has a pair of $\Gamma$-invariant
non-proportional holomorphic sections.  The quotient of these sections
is a non-constant $\Gamma$-invariant meromorphic function on $G/B$,
hence its restriction to $\Omega_\rho^I$ descends to a non-constant
meromorphic function on $\W_{\rho}^{I}$, giving (ii).
\end{proof}

As a final application of our sheaf cohomology calculations, we
consider the Kodaira dimension of uniformized $(G, G/B)$-manifolds.
Recall that a compact complex manifold $Y$ with canonical bundle
$K_Y$ is said to have \emph{Kodaira dimension $-\infty$}, denoted
$\kappa(Y) = -\infty$, if $H^{0}(Y, K_{Y}^{d})$ vanishes for all
$d>0$.
Because the flag variety $G/B$ is rational, it has $\kappa(G/B) =
-\infty$.  The same holds for uniformized $(G,G/B)$-manifolds with
sufficiently small limit sets:
\begin{thm}
Let $G$ be a connected semisimple complex Lie group of rank at
least two.  Let $B<G$ a Borel subgroup, and $\W$ a uniformized
$(G,G/B)$-manifold with data $(\Omega, \Gamma)$ and limit set
$\Lambda$.  Suppose $m_{2n-4}(\Lambda)=0$ where $n$ is the complex
dimension of $G/B$.  Then $\kappa(\W)=-\infty$.
\end{thm}

\begin{proof}
The canonical line bundle of $G/B$ is isomorphic to
$\mathcal{L}^{-\delta}=\mathcal{L}_{2\delta}$ where as before $\delta$ is half the sum of the positive
roots.  Therefore, we have
\[
K_{\W}^{d} \simeq p_{*}^{\Gamma}(\mathcal{L}^{(1-2d)\delta}).
\]
For any integer $d>0$, the weight $(1-2d)\delta$ is regular and $w_0((1-2d)\delta)$ is
dominant, where $w_{0}$ is the longest element of the Weyl group.
Therefore, by Theorem \ref{thm: sheaf cohomology}(ii) we have
\[
H^0(\W, K_\W^d) \simeq H^{0}(\W, p_{*}^{\Gamma}(\mathcal{L}^{(1-2d)\delta}))=0
\]
for all $d>0$ provided that $\ell(w_{0})>1$, which is the case since
the rank of $G$ is at least two.
\end{proof}
Note that the corresponding statement fails for $G \simeq \SL_2\C$
since Riemann surfaces of higher genus can be obtained as uniformized
$(G,G/B) = (\SL_2\C, \CP^1)$ manifolds, and the canonical bundle of
such a Riemann surface has non-trivial sections.

\section{Examples and complements}\label{examples}

In this final section we return to the topological considerations of
Section \ref{topology} and discuss some specific examples of balanced
ideals, domains, and quotient manifolds for various complex simple Lie
groups $G$ and parabolic pairs $(P_A,P_D)$.  (The survey \cite{KL17}
also gives examples of balanced ideals, including some that belong to
the infinite families constructed below.)

\subsection{The lower half of \texorpdfstring{$W$}{W}}

Certain ideals can be constructed easily from the length function on
the Weyl group $W$.  Since $x < y$ implies $\ell(x) < \ell(y)$, the
set
\[ W_{\leq L} := \{ x \suchthat \ell(x) \leq L \} \]
is an ideal in $W$ for any integer $L$, and this ideal is minimally generated
by $\ell^{-1}(L)$.  Generalizing this, if $J$ is a subset of
$\ell^{-1}(L+1)$, then $W_{\leq L} \cup J$ is also an ideal, and the
minimal generating set of this ideal contains $J$.

This construction can always be used to produce a balanced
ideal. Define the \emph{lower half} of $W$ to be the ideal
\[ I_{\half} = W_{\leq \half\ell(w_0)}.\]
Since $\ell(w_0x) = \ell(w_0) - \ell(x)$, it is immediate that this
ideal is balanced if $\ell(w_0) = \dim_\C G/B$ is odd, which is the
case for all simple $G$ of type $B_{n} = \PO_{2n+1}\C$,
$C_n = \PSp_{2n}\C$, or $E_7$, and for type $A_n = \PSL_{n+1}\C$ when
$n$ is $1$ or $2$ mod $4$.

In such cases, considering $I_{\half}$ as an ideal of type $(B,B)$, it
gives a model thickening $\Phi_{\half} := \Phi^{I_{\half}}
\subset G/B$ and domain of
discontinuity $\Omega_\half \subset G/B$ for $B$-Anosov
representations.  Suppose $\ell(w_0) = 2k+1$ for $k \in \Z$.  Then the
model thickening has the same Betti numbers as $G/B$ itself in the
range $1\ldots 2k$, i.e.
\[ r_i = b_i(\Phi_\half) = b_i(G/B) = \card{\ell^{-1}(i)} \text{ for } i \leq
2k.\]
Applying Corollary \ref{cor:omega-betti} gives a particularly simple expression
for the Betti numbers of the domain of discontinuity:
\[
b_i(\Omega_\half) = \begin{cases}
b_i(G/B) & \text{if }i < 2k\\
2 b_{2k}(G/B) & \text{if }i = 2k\\
b_{4k-i}(G/B) &\text{if }i>2k 
\end{cases}
\]
By Theorem \ref{thm:homology-product} there is a corresponding
formula for the homology of the compact quotient manifolds.

If $\ell(w_0) = 2k$ is even, the construction can be modified to
produce a balanced ideal.  Note that the ``middle'' length
$W_{\text{mid}} := \ell^{-1}(k)$ is mapped to itself under left
multiplication by $w_0$.  Let $J \subset W_{\text{mid}}$ be a subset
containing one element of each $w_0$-orbit.  Then the set
\[ I_{\half,J} = W_{\leq (k-1)} \cup J \]
is a balanced ideal whose minimal generating set contains $J$.  (In
some examples, $I_{\half,J}$ is in fact generated by $J$, while in
other cases there are additional generators of length $k-1$.)

Since there are $2^{\card{W_\text{mid}}/2}$ such sets $J$, this gives a
large collection of balanced ideals, all of which have the same number
of elements of each length.  The corresponding generalizations of the
Betti number formulas given above are
\[
r_i = b_i(\Phi_{\half,J}) = \begin{cases}
b_i(G/B) & i < 2k\\
\frac{1}{2}b_i(G/B) & i = 2k\\
0 & \text{else}
\end{cases}
\]
and by Corollary \ref{cor:omega-betti},
\begin{equation}
\label{eqn:omega-betti-odd-lower-half}
 b_i(\Omega_{\half,J}) = \begin{cases}
b_i(G/B) & \text{if }i < 2k-2\\
b_{2k-2}(G/B) + \frac{1}{2}b_{2k}(G/B)& \text{if }i \in \{ 2k-2, 2k\} \\
b_{4k - 2 - i}(G/B) & \text{if }i > 2k.
\end{cases}
\end{equation}

\subsection{Constructions for \texorpdfstring{$\PSL_n \C$}{PSL(n,C)}}
\label{sec:sln-constructions}

In preparation for the next two types of examples, we recall how some
of the combinatorial and Lie-theoretic notions specialize to the case
$G = A_{n-1} = \PSL_n\C$; general references for this material
include \cite{bjorner-brenti} (concerning Weyl groups), 
\cite{lakshmibai-gonciulea} \cite{BRI05} (concerning flag varieties),
and \cite{fulton:young-tableaux} (concerning both).

We choose the Borel $B < G = \PSL_n\C$ consisting of the
upper-triangular matrices.  The manifold $G/B$ is $G$-equivariantly
identified with the set of complete flags $F = (F_1, \ldots F_{n-1})$,
i.e.~$F_1 \subset \ldots \subset F_{n-1} \subset \C^n$ and
$\dim_\C F_k = k$.  We denote by $E$ the standard flag of $\C^n$ in
which $E_k = \Span \{ e_1, \ldots, e_k \}$, which corresponds to
$eB \in G/B$; here $e_1, \ldots, e_n$ is the standard ordered basis of
$\R^n$.

Standard parabolic subgroups $P < G$ are stabilizers of partial
flags within $E$, with associated quotients $G/P$ parameterizing all
flags of that type. An example we will focus on is $P_{1,n-1}$, the
\emph{incidence parabolic}, which is defined as the stabilizer of
$(E_1,E_{n-1})$.  Thus $G/P_{1,n-1}$ is the set of pairs $(\ell,H)$ of
a line and a containing hyperplane.

The Weyl group $W = W(\PSL_n\C)$ is isomorphic to the symmetric
group $S_{n}$, with the roots (respectively, simple roots) of $G$
corresponding to transpositions (respectively, transpositions of
adjacent elements).  We identify a permutation $x\in S_n$ with the
tuple $(x(1), x(2), \ldots, x (n))$.

The Weyl group $W_{1,n-1}$ of $P_{1,n-1}$ consists of permutations
$w \in S_n$ with $w(1) = 1$ and $w(n) = n$.  Thus, the cosets space
$W / W_{1,n-1}$ consists of classes of permutations
$W(i,j) = \{ (i,*, \ldots, *, j) \} \subset S_n$ for $i \neq j$.

The Chevalley-Bruhat order has a simple description in terms of permutations.
For $w \in S_n$ we define the set of \emph{ascents} of $w$ to be
\[A(w) := \{ i \suchthat w(i) < w(i+1) \}.\]
This is a subset of $\{1,2,\ldots,n-1\}$.  We also denote by $w_{i,j}$
the $j$-th smallest element of the set $\{ w(1), \ldots,
w(i) \}$.  Then:

\begin{thm}[{\cite[Theorem~2.6.3(iii)]{bjorner-brenti}}]
\label{thm:rank-comparison}
Elements $x,y \in S_n$ satisfy $x \leq y$ if and only if $x_{i,j} \leq
y_{i,j}$ for all $i \in A(y)$ and all $j \leq i$. \noproof
\end{thm}

Note that this characterizes elements of the ideal
$\langle y \rangle = \{ x \suchthat x \leq y\}$ by an explicit set of
inequalities.

There is a corresponding formula for the length of an element
$w \in S_n$ as its number of \emph{inversions} (see
\cite[Proposition~1.5.2]{bjorner-brenti}):
\begin{equation}
\label{eqn:length-formula}
\ell(x) = \bigcard{ \{(i,j) \suchthat i<j \text{ and } \sigma(i) >
\sigma(j) \} }.
\end{equation}
Thus the longest element is $w_0 = (n,n-1,\ldots,1)$.

The Schubert variety $X_w = \overline{BwB} \subset G/B$ is defined by an explicit set
of dimension inequalities depending on the permutation $w$; precisely,
we have:
\begin{thm}[{\cite[Section 10.5]{fulton:young-tableaux}}]
\label{thm:schubert-variety-as-flags}
The Schubert variety $X_w$ consists of the flags $(F_1, \ldots, F_n)$ such that
\[ \dim(F_p \cap E_q) \geq \bigcard{ \{ (i,j) \suchthat i \leq p, \: w(j) \leq q
\}}. \tag*{\noproof} \]
\end{thm}

Finally, we note that the partial flag variety
$G/P_{1,n-1} = \{ (\ell,H)\}$ can be embedded as a hypersurface in
$\CP^{n-1} \times (\CP^{n-1})^*$, which we call the \emph{incidence
  variety}, consisting of pairs of a vector $x \in \C^n$ and a linear
form $\xi \in (\C^n)^*$ such that $\xi(x) = 0$, modulo the action of
$\C^* \times \C^*$.  Here $(x,\xi)$ corresponds to the flag
$(\C \cdot x, \ker \xi)$.  Using the theorem above, one can check that
in this realization the Schubert variety
$X_{W(i,j)} \subset G/P_{1,n-1}$ is cut out by the equations
$x_{i+1} = \ldots = x_{n} = \xi_1 = \ldots = \xi_{j-1} = 0$.

\subsection{The \texorpdfstring{$(1,n-1)$}{(1,n-1)}-examples}
\label{subsec:incidence}

In this section we describe how certain domains studied by
Guichard-Wienhard in \cite[Section 10.2.2]{GW12} are represented in
the Kapovich-Leeb-Porti formalism (i.e.~by Chevalley-Bruhat ideals),
and what is obtained by applying the results of Section \ref{topology}
to these examples.

We define the \emph{incidence ideal} to be the subset of $S_n$ given
by 
\[
I_{1,n-1} = \{ x \in S_n \suchthat x(1) < x(n) \}.
\]
Equivalently, this is a union of $W_{1,n-1}$ cosets,
$I_{1,n-1} = \bigcup_{i < j} W(i,j)$.

For $1 \leq k \leq n-1$, let $z_k \in S_n$ be defined by
\begin{equation*}
z_k(i) = \begin{cases}
k & \text{ if } i=1\\
k+1 & \text{ if } i = n\\
n-i+2 & \text{ if } 1 < i \leq n-k\\
n-i & \text{ otherwise.}
\end{cases}
\end{equation*}
Equivalently (and perhaps more transparently) $z_k$ is defined by
the unique tuple $(k, \ldots, k+1)$ in which the omitted
elements appear in decreasing order.  Note that $z_k \in I_{1,n-1}$,
and that $z_k$ is the unique longest element in the coset $W(k,k+1)$.

\begin{thm}
The set $I_{1,n-1} \subset S_n$ is a balanced and right
$W_{1,n-1}$-invariant ideal of the Chevalley-Bruhat order on $S_n$.  It is
minimally generated by $\{ z_1, z_2, \ldots, z_{n-1}\}$.
\end{thm}

\begin{proof}
Since $(w_0 x)(i) = n + 1 - x(i)$ it is immediate that left
multiplication by $w_0$ exchanges $I_{1,n-1}$ with its complement.
Thus if this set is an ideal, then it is balanced.  We have already
seen that $I_{1,n-1}$ is a union of left $W_{1,n-1}$-cosets (and hence
right-$W_{1,n-1}$-invariant).

Next, we claim that the Chevalley-Bruhat order satisfies
\begin{equation}
\label{eqn:1n-technical}
x \leq z_k \;\; \text{ if and only if } \; x(1) \leq k \text{ and } x(n) > k.
\end{equation}
Before proving this, we derive the rest of the statements of the
Theorem from it.  An element $x \in W$ satisfies the right hand side
of \eqref{eqn:1n-technical} for some $k$ if and only if $x(1) < x(n)$,
hence the condition above is equivalent to the statement that
$I_{1,n-1}$ is the union of the principal ideals $\langle z_k \rangle$
for $k = 1, \ldots, n-1$, and in particular is an ideal.  It is
straightforward to calculate from \eqref{eqn:length-formula} that
$\ell(z_k) = \frac{1}{2}(n-1)(n-2)$ for all $k$, so these elements are
pairwise incomparable and of maximal length within $I_{1,n-1}$.  This
shows $\{ z_1, z_2, \ldots, z_{n-1}\}$ is the minimal generating set.

Finally we prove \eqref{eqn:1n-technical} using Theorem
\ref{thm:rank-comparison}.  First suppose that $1 < k < n-1$. Then
$A(z_k) = \{ 1,n-1 \}$ and we find $x \leq z_k$ if and only if
\[ x(1) = x_{1,1} \leq (z_k)_{1,1} = z_k(1) = k \] and
\[ x_{n-1,j} \leq (z_k)_{n-1,j} \: \text{ for } \: j \leq n-1.\]
Since $\{x(1), \ldots, x(n-1) \} = \{1, \ldots, n\} \setminus x(n)$
(and similarly for $z_k$), the second set of inequalities is
equivalent to $x(n) \geq z_k(n) = k+1$, or equivalently $x(n) > k$, as
desired.  The cases $k=1$ and $k=n-1$ are similar, except that $z_k$
then has only one ascent.  We omit the straightforward verification
that the argument above still applies in these cases.
\end{proof}

Using the right-invariance of $I_{1,n-1}$ we can apply the
Kapovich-Leeb-Porti construction with $P_A = B$ and $P_D = P_{1,n-1}$
to obtain a limit set $\Lambda_{1,n-1} := \Lambda^{I_{1,n-1}}_\rho$
and cocompact domain of discontinuity
$\Omega_{1,n-1} := \Omega^{I_{1,n-1}}_{\rho}$ in the incidence variety
$G/P_{1,n-1}$ for a $B$-Anosov representation $\rho : \pi \to G$ of a
word hyperbolic group $\pi$.

Applying Theorem \ref{thm:schubert-variety-as-flags} to $z_k$ we find
that the associated Schubert variety $X_{z_k} \subset G/B$ is
characterized by dimension inequalities $\dim(F_1 \cap E_k) \geq 1$
and $\dim(E_k \cap F_{n-1}) \geq k$.  Projecting to $G/P_{1,n-1}$ we
obtain the Schubert variety
\[ X_{W(k,k+1)} = X_{z_kW_{1,n-1}} = \{ (F_1,F_{n-1}) \suchthat F_1
\subset E_k \subset F_{n-1} \}.\]
Taking the union of these sets over $k$ gives the model thickening
$\Phi_{1,n-1} := \Phi^{I_{1,n-1}}$ in $G/P_{1,n-1}$, and the limit set itself is given
by
\[ \Lambda_{1,n-1} = \bigcup_{t \in \partial_\infty \pi} \{
(F_1,F_{n-1}) \suchthat \exists k, \: F_1 \subset \xi_k(t) \subset
F_{n-1} \},\]
where $\xi_k(t)$ is the $k$-dimensional component of the flag
corresponding to $\xi(t) \in G/B$.  This is the domain constructed in
\cite[Section 10.2.2]{GW12}.  Using the results of Section \ref{topology} we can now
derive a closed formula for the Betti numbers of
$\Omega_{1,n-1}$ in the case of a $G$-Fuchsian representation.
\begin{thm}
\label{thm:example-incidence-betti}
This domain of discontinuity $\Omega_{1,n-1} \subset G/P_{1,n-1}$ in
the incidence variety associated to a $G$-Fuchsian representation
$\rho : \pi_{1}S \to \PSL_n\C$ satisfies
\[ b_{2k}(\Omega_{1,n-1}) = \begin{cases}
2n-2 & \text{ if } k = n-2 \\
\max \left ( 0,n-1- \left|n-k-2\right| \right ) & \text{ else.}
\end{cases} \]
Hence its Poincar\'e polynomial is
\[ p(x) = \sum_i b_{i}x^i = \frac{\left (1-t^{2(n-1)} \right )^2}{\left
  (1-t^2 \right )^2} + (n-1)t^{2n-4}. \]
\end{thm}

\begin{proof}
Recall that $r_k$ is the number of elements of $I/W_{1,n-1}$ of length
$k$, and that $I/W_{1,n-1}$ consists of the cosets $W(i,j)$ with $i<j$.
By \eqref{eqn:length-formula}, the element of $W(i,j)$ of minimal length
\[ (i,1,2,\ldots, \hat{i}, \ldots, \hat{j}, \ldots, n,j) \in W(i,j) \]
has length $n+i-j-1$, hence $r_k$ is the number of pairs $(i,j)$ with
$1\leq i<j\leq n$ and $n+i-j-1=k$.  Such pairs exist for $0 \leq k
\leq n-2$, and enumerating them we find
\[ r_k = \begin{cases}
k+1 & \text{ if } 0 \leq k \leq n-2\\
0 & \text{ else.}
\end{cases} \]
Since $\dim_\C \F_{1,n-1} = 2n-3$, Corollary \ref{cor:omega-betti}
gives $b_{2k}(\Omega) = r_k + r_{2n-4-k}$.  Substituting the formula
for $r_k$ we find that for all $k$ except $n-2$, only one of the
terms is non-zero.  Considering the various cases for $k$ we find
\[ b_{2k}(\Omega_{1,n-1}) = \begin{cases}
k+1 & \text{ if } 0 \leq k < n-2\\
2n-2 & \text{ if } k = n-2\\
2n - 3 - k & \text{ if }  n-2<k\leq 2n-4 \\
0 & \text{ if } k> 2n-4
\end{cases}
\]
which is easily seen to be equivalent to the formula in the theorem.
We omit verification of the corresponding closed form for $p(x)$.
\end{proof}

\subsection{The \texorpdfstring{$2n$}{2n} examples: Principal balanced ideals}

All of the ideals discussed so far in this section have large minimal
generating sets; this follows, for example, from their having many
elements of maximal length.  In this subsection we describe a family
of examples of balanced ideals that are also principal, i.e.~generated
by a single element.  In more geometric terms, these correspond to
model thickenings given by a single Schubert variety.

Let $G = \PSL_{2n}\C$, so that $W \simeq S_{2n}$.  We have:
\begin{thm}
The set 
$I_{2n} := \{ w \in S_{2n} \suchthat w(2n) > n \}$
is a principal, balanced ideal.  In fact,
$I_{2n} = \langle \lambda \rangle$
where
$ \lambda = (2n, 2n-1, \ldots, \widehat{n+1}, \ldots, 2, 1, n+1).$
\end{thm}

\begin{proof}
Since $(w_0 x)(i) = 2n + 1 - x(i)$, it is immediate from the definition
that $I_{2n}$ and its complement are exchanged by left multiplication
by $w_0$.  Thus if $I_{2n}$ is an ideal, it is balanced, and it
suffices to show $I_{2n} = \langle \lambda \rangle$.

Examining the explicit form of $\lambda$ we see there is a single
ascent, $A(\lambda) = \{2n-1\}$.  Applying Theorem
\ref{thm:rank-comparison} and computing $\lambda_{2n-1,j}$ we find
that $x \in \langle \lambda \rangle $ if and only if
\begin{equation}
\begin{split}
\label{eqn:rank-ineq}
x_{2n-1, j} &\leq j \;\;\;\;\;\;\;\;\; \text{ for } j \leq n \hfill\\
x_{2n-1, j} &\leq j+1 \;\;\; \text{ for }  j > n
\end{split}
\end{equation}
But note that
$\{ x(1), \ldots, x(2n-1) \} = \{ 1, \ldots, 2n\} \setminus \{
x(2n)\}$, hence for all $x$ we have
\[ x_{2n-1,j} = \begin{cases}
j & \text{if } j < x(2n)\\
j+1 & \text{if } j \geq x(2n)\\
\end{cases}
\]
Comparing this to \eqref{eqn:rank-ineq}, we see that $x \in
\langle \lambda \rangle$ if and only if $x(2n) < n$, as desired.
\end{proof}

As mentioned above, because $I_{2n}$ is principal, the associated
model thickening $\Phi_{2n} := \Phi^{I_{2n}} \subset G/B$ is the
Schubert variety $X_\lambda$.  While Schubert varieties can in general
have singularities, this one is smooth: This is immediate from the
pattern avoidance criterion of Lakshmibai-Sandhya
\cite{LS90}, or it can be verified from the
description of $X_\lambda$ using dimension inequalities for flags.
The latter will give a more detailed description and allow us to
compute the Poincar\'e polynomial of $\Omega_{2n} := \Omega^{I_{2n}}$:

\begin{thm}
The domain of discontinuity $\Omega_{2n}$ has Poincar\'e polynomial
\[ \frac{(1+t^{2n-2})(1-t^{2n})}{(1-t^2)^{2n-1}} \prod_{i=1}^{2n-2}
  \left (1-t^{2(i+1)} \right ). \]
\end{thm}

\begin{proof}
For brevity, in this proof we denote by $\F(m)$ the full flag variety of $\C^m$
and by $\F(i_1, \ldots, i_k; m)$ the variety of partial flags in
$\C^m$
with components of dimensions $i_1 < i_2 < \ldots < i_k$.  Each such
space is a smooth manifold.  We write $p[X](t)$ for the Poincar\'e
polynomial of a space $X$.

The projection $\pi: (F_1, \ldots, F_{m-1}) \mapsto (F_1,\ldots,F_k)$
is a smooth fibration of $\F(m)$ over $\F(1, \ldots, k; m)$ with fiber
diffeomorphic to $\F(m-k)$.  Furthermore, applying the Serre spectral
sequence shows that this bundle is homologically trivial.  Thus the
Poincar\'{e} polynomial of the base of this bundle satisfies
\begin{equation}
\label{eqn:flag-fibration}
p[\F(1,\ldots,k;m)] = \frac{p[\F(m)]}{p[\F(m-k)]}.
\end{equation}

Applying Theorem \ref{thm:schubert-variety-as-flags} to the
permutation $\lambda$ we find
\[ \Phi_{2n} = X_\lambda = \{ (F_1, \ldots, F_{2n-1}) \suchthat F_n \subset
E_{2n-1} \}.\]
Considering the fibration $\F(2n) \to \F(1\,\ldots,n;2n)$ (i.e.~taking
$m=2n$ and $k=n$ above), this description of $\Phi_{2n}$ is equivalent
to identifying it with the preimage $\pi^{-1}(Y)$ of
$Y = \{ (F_1, \ldots, F_n) \suchthat F_n \subset E_{2n-1} \} \simeq
\F(1, \ldots, n; 2n-1)$.
Thus $\Phi_{2n}$ is a smooth fiber bundle over
$\F(1, \ldots, n; 2n-1)$ with fiber $\F(n).$  Again applying the Serre
spectral sequence shows this bundle is homologically trivial and we obtain
\[p[\Phi_{2n}] = p[\F(1, \ldots, n; 2n-1)] p[\F(n)]\]
Using \eqref{eqn:flag-fibration} with $m=2n-1$ and $k=n$ we find
$p[\F(1, \ldots, n; 2n-1)] = p[\F(2n-1)] / p[\F(n-1)]$ and thus
\[ p[\Phi_{2n}] = \frac{p[\F(2n-1)] p[\F(n)]}{p[\F(n-1)]}.\]
Substituting the classical formula for the Poincar\'e polynomial of
the flag variety itself (see e.g. \cite{macdonald}),
\[ p[\F(m)](t) = (1-t^2)^{1-n} \displaystyle\prod_{i=1}^{m-1} (1-t^{2(i+1)}),\]
and simplifying we obtain
\[p[\Phi_{2n}](t) = \frac{(1-t^{2n})}{(1-t^2)^{2n-1}}
\prod_{i=1}^{2n-2} (1-t^{2(i+1)}).\]

It follows from \eqref{eqn:length-formula} that $\ell(\lambda)
= \ell(w_0) - n$.  Since it is a smooth manifold, the model thickening
$\Phi_{2n}$ satisfies Poincar\'{e} duality in this dimension.
In terms of the number $r_k$ of elements of $I$ of length
$k$, this means
\[ r_{k} = r_{L-n-k} \]
where $L = \ell(w_0)$, and the formula of Corollary
\ref{cor:omega-betti} simplifies in this case to
\[ b_{2k}(\Omega_{2n}) = r_k + r_{k-(n-1)}. \]
Returning to Poincar\'e polynomials, this shows
\[ p[\Omega_{2n}](t) = (1+t^{2n-2}) p[\Phi_{2n}](t), \]
and substituting the expression for $p[\Phi_{2n}](t)$ obtained above,
the theorem follows.
\end{proof}

\subsection{Homotopy types}

For most complex adjoint groups $G$ there are many balanced ideals in
$I \subset W$; it is natural to ask whether these correspond to topologically
distinct quotient manifolds $\W^I$.  We will verify this for two of the
Chevalley-Bruhat ideal examples studied thus far, applied to $G$-Fuchsian representations:

\begin{thm}
Let $G = \PSL_{2n}\C$ where $n = 2j+1$, $j \in \Z$.  Let
$I_\half ,I_{2n} \subset W$ denote, respectively, the lower half and
principal balanced ideals constructed above. Let $\rho : \pi_{1}S \to G$
be a $G$-Fuchsian representation.  Then the quotient manifolds
$\W_\half$ and $\W_{2n}$ associated to $\rho$ are not homotopy
equivalent.
\end{thm}

\begin{proof}
In this case $L = \ell(w_0) = 2k+1$ where $k = j(4j+3)$.  By Corollary
\ref{cor:omega-betti} we have for any balanced ideal $I$ that
\[ b_{2k}(\Omega^I) = 2 r_k(I) = 2 | \ell^{-1}(k) \cap I |. \]
Applying this to $I_\half$ and using
\eqref{eqn:omega-betti-odd-lower-half} we have
\[ b_{2k}(\Omega_\half) = 2 b_{2k}(G/B) = 2 |\ell^{-1}(k)|. \]
Now consider the element $\mu \in S_{2n}$ given by the tuple
\[ \mu = (2j, \ldots, j+1, 4j+2, j, \ldots 1, 4j+1, \ldots, 2j+1 ),\]
where in this expression $a\ldots b$ denotes the integers between $a$
and $b$ in decreasing order.  Then $\mu \not \in I_{2n}$ since
$\mu(2n) = n = 2j+1$.  A straightforward application of
\eqref{eqn:length-formula} shows $\ell(\mu) = k$.  As
$\mu \in \ell^{-1}(k) \setminus I_{2n}$ we have
$\card{\ell^{-1}(k) \cap I_{2n}} < 2\card{\ell^{-1}(k)}$, which by the
formulas above gives
\begin{equation}
\label{eqn:omega-betti-different}
b_{2k}(\Omega_{2n}) < b_{2k}(\Omega_\half).
\end{equation}

Applying Theorem \ref{thm:homology-product}, and using the vanishing of odd
homology groups of $\Omega^I$ from Theorem \ref{thm:omega-homology},
we have for any balanced ideal $I$ that
\[ b_{2k+1}(\W^I) = b_1(S) b_{2k}(\Omega^I). \]
Combining this with \eqref{eqn:omega-betti-different} we find
$b_{2k+1}(\W_{2n}) < b_{2k+1}(\W_\half)$, and these manifolds are
not homotopy equivalent.
\end{proof}

\subsection{The \texorpdfstring{$\PSL_3\C$}{PSL(3,C)} case}
\label{subsec:sl3}

In this final subsection, we consider $G = \PSL_3\C$ and give an
alternative description of the limit set and domain of discontinuity
in $G/B$ for a $G$-Fuchsian group.  This allows us to verify
Conjecture \ref{conjecture fiber} in this case.  Chronologically, our
study of this example preceded the other results of this paper, and
indeed, the main results of Sections \ref{topology}--\ref{complex
  geometry} resulted from an attempt to generalize aspects of the
picture described below to other complex Lie groups.

For $G = \PSL_3\C$ there is unique balanced ideal
$I = I_{\half} = I_{1,2}$ in the Weyl group $W \simeq S_3$.  Here
$I = \{ e, \alpha_1, \alpha_2 \}$ where $\alpha_i$ are the simple root
reflections, or in the permutation model,
$I = \{ (1,2,3), (2,1,3), (1,3,2) \}$.  Since $I$ is fixed we write
$\Phi,\Lambda,\Omega,\W$, for the model thickening, limit set, domain,
and quotient manifold, dropping the decoration by $I$ from our notation.

Let $\rho : \pi_{1}S \to \PSL_3\C$ be a $\PSL_3\C$-Fuchsian
representation, and in the rest of this section let
$\F = G/B = \{ (\ell,H) \suchthat \ell \subset H \}$ denote the flag
variety.  Let $X \subset \F$ denote the principal curve and
$\tilde{\phi} = f_{\PSL_3\C} : \CP^1 \to X$ its holomorphic
parameterization.  Let $Y \subset \CP^2$ denote the projection of the
principal curve under the map $(\ell,H) \mapsto \ell$, and
$\phi : \CP^1 \to Y$ the composition of $\tilde{\phi}$ with the same
projection.

In what follows we regard an element $\ell \in \CP^2$ as a point in
a complex surface, rather than as a $1$-dimensional subspace of a
$3$-dimensional vector space.  Also, we identify the symmetric product
$\Sym^d(\CP^1)$ with the set of effective divisors of degree $d$ on
$\CP^1$, so for example an element of $\Sym^2(\CP^1)$ is
expressible as $p+q$, for $p,q \in \CP^1$.

There is a biholomorphic map $\CP^2 \simeq \Sym^2(\CP^1)$ which maps
$\ell \in \CP^2$ to $p+q$ if $\ell$ lies on distinct tangent lines
$T_{\phi(p)}Y$ and $T_{\phi(q)}Y$, and to $2p$ if $\ell = \phi(p)$.
Dually there is an identification $(\CP^2)^*$ with $\Sym^2(\CP^1)$,
where we regard $H \in (\CP^2)^*$ as a projective line in $\CP^2$, and
map $H$ to the sum (with multiplicity) of the $\phi$-preimages of its
intersection points with $Y$.

Since $P_{1,n-1} = B$ for this group, following the discussion at the
end of Section \ref{sec:sln-constructions} we have the embedding
$\F \into \CP^2 \times (\CP^2)^*$.  Composing with the maps introduced
above we then have $\F \into \Sym^2(\CP^1) \times \Sym^2(\CP^1)$.  It
is easy to check that the principal curve $X\subset \F$ maps to the
set $\{ (2p,2p) \suchthat p \in \CP^1 \}$ and that
$\tilde{\phi}(p) = (2p,2p)$.  Recall that the limit curve of $\rho$ is
the circle $\tilde{\phi}(\RP^1)\subset \F$.

In order to give a geometric description of the limit set and domain
of discontinuity, we further identify $\CP^1$ with the boundary at
infinity of the $3$-dimensional real hyperbolic space $\H^3$, for
example using stereographic projection\footnote{More intrinsically, we
  could view $\H^3 \simeq \SL_2 \C / \SU(2)$ as the space of hermitian
  forms on the vector space $H^0(Y,\O(1))$ that induce a
  given volume form---a space which is compactified by $Y$ itself.} to
map $\CP^1$ to the unit sphere in $\R^3$ considered as the boundary of
the unit ball model of $\H^3$.  Let $\gamma_{p,q}$ denote the
hyperbolic geodesic with ideal endpoints $p,q \in \CP^1$.

\begin{lem}
\label{lem:divisor-flag-model}
A point $x$ in $\Sym^2(\CP^1) \times \Sym^2(\CP^1)$ lies in
the image of $\F$ if and only if it satisfies one of the following
mutually exclusive conditions:
\begin{itemize}
\item $x = (p+q,r+s)$ where $p,q,r,s$ are pairwise distinct and the
hyperbolic geodesics $\gamma_{p,q}$ and $\gamma_{r,s}$ intersect orthogonally,
or
\item $x = (2p, p + q)$ where $p \neq q$, or
\item $x = (p+q,2q)$ where $p \neq q$, or
\item $x = (2p,2p) \in X$.
\end{itemize}
\end{lem}

\begin{proof}
Suppose that $x = (\xi,\eta)$ corresponds to a flag $(\ell,H)$ where
the divisors $\xi, \eta \in \Sym^2(\CP^1)$ have a point in common, say
$p$.  By the construction of the embedding given above, this means
\begin{itemize}
\item The projective line $H \subset \CP^2$ passes through $Y$ at $\phi(p)$, and
\item The tangent line $T_{\phi(p)}Y$ contains $\ell$.
\end{itemize}
Since $\ell \in H$, both $\phi(p)$ and $\ell$ lie in
$T_{\phi(p)}Y \cap H$.  Since distinct projective lines intersect in a
single point, we have either $\ell = \phi(p)$, in which case
$\xi = 2p$, or $\T_\phi(p)Y = H$, in which case $\eta = 2p$, or both.

This shows that $x$ has one of the given forms, with the exception of
the orthogonality condition in the first case.  Hence we must show
that for distinct $p,q,r,s$ the geodesics $\gamma_{p,q}$ and
$\gamma_{r,s}$ intersect orthogonally in $\H^3$ if and only if the
corresponding pair of a point and projective line in $\CP^2$ form a
flag, i.e.~the projective line spanned by $\phi(r)$ and $\phi(s)$ is
concurrent with the tangents $T_{\phi(p)}Y$ and $T_{\phi(q)}Y$.  This
can be done with an elementary explicit calculation, but we prefer to
give a coordinate-free proof.

Given two points $p,q \in \CP^1$, the \emph{half turn}
$\tau_{p,q} : \CP^1 \to \CP^1$ is the unique non-trivial holomorphic
involution fixing $p$ and $q$.  Geometrically, $\tau_{p,q}$ is the
extension to the ideal boundary of the isometry $\H^3 \to \H^3$ which
rotates about $\gamma_{p,q}$ by angle $\pi$.  Thus geodesics $\gamma_{p,q}$
and $\gamma_{r,s}$ intersection orthogonally if and only if $\{r,s\}$
is an orbit of $\tau_{p,q}$.

\begin{figure}
\begin{center}
\begingroup%
  \makeatletter%
  \providecommand\color[2][]{%
    \errmessage{(Inkscape) Color is used for the text in Inkscape, but the package 'color.sty' is not loaded}%
    \renewcommand\color[2][]{}%
  }%
  \providecommand\transparent[1]{%
    \errmessage{(Inkscape) Transparency is used (non-zero) for the text in Inkscape, but the package 'transparent.sty' is not loaded}%
    \renewcommand\transparent[1]{}%
  }%
  \providecommand\rotatebox[2]{#2}%
  \ifx\svgwidth\undefined%
    \setlength{\unitlength}{320bp}%
    \ifx\svgscale\undefined%
      \relax%
    \else%
      \setlength{\unitlength}{\unitlength * \real{\svgscale}}%
    \fi%
  \else%
    \setlength{\unitlength}{\svgwidth}%
  \fi%
  \global\let\svgwidth\undefined%
  \global\let\svgscale\undefined%
  \makeatother%
  \begin{picture}(1,0.46505771)%
    \put(0,0){\includegraphics[width=\unitlength,page=1]{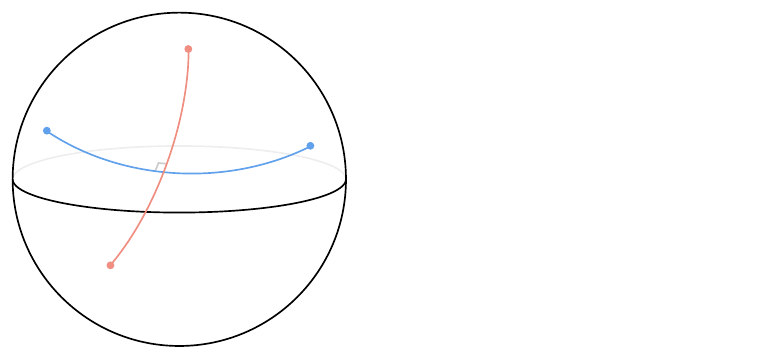}}%
    \put(0.0346,0.283){\color[rgb]{0,0,0}\makebox(0,0)[lb]{$p$}}%
    \put(0.415,0.265){\color[rgb]{0,0,0}\makebox(0,0)[lb]{$q$}}%
    \put(0.235,0.416){\color[rgb]{0,0,0}\makebox(0,0)[lb]{$x$}}%
    \put(0.083,0.071){\color[rgb]{0,0,0}\makebox(0,0)[lb]{$\tau_{p,q}(x)$}}%
    \put(0,0){\includegraphics[width=\unitlength,page=2]{halfturns.pdf}}%
    \put(0.576,0.250){\color[rgb]{0,0,0}\makebox(0,0)[lb]{$u$}}%
    \put(0.922,0.352){\color[rgb]{0,0,0}\makebox(0,0)[lb]{$v$}}%
    \put(0.734,0.322){\color[rgb]{0,0,0}\makebox(0,0)[lb]{$w$}}%
    \put(0.540,0.018){\color[rgb]{0,0,0}\makebox(0,0)[lb]{$\hat{\tau}_{u,v}(w)$}}%
  \end{picture}%
\endgroup%
\caption{Hyperbolic and projective models of a half turn on $\CP^1$.}
\label{fig:halfturns}
\end{center}
\end{figure}

Given a pair of points $\{u,v\} \subset Y$, we can define a map
$\hat{\tau}_{u,v} : Y \to Y$ as follows: Let $H^* = T_uY \cap T_vY$,
which is a point not on $Y$.  The projective line joining $H^*$ to
$w \in Y$ intersects $Y$ in a second point, which is
$\hat{\tau}_{u,v}(w)$.  (See Figure \ref{fig:halfturns}.)  Since this
defines an involutive, non-trivial holomorphic automorphism of $Y$
fixing $u$ and $v$, it is $\phi$-conjugate to a half turn of $\CP^1$,
i.e.
\[ \hat{\tau}_{u,v}(\phi(t)) = \phi(\tau_{p,q}(t)).\]
On the other hand, by definition of $\hat{\tau}_{u,v}$ the points
$\phi(r),\phi(s)$ form an orbit if and only if the projective line they
span is concurrent with $T_uY$ and $T_vY$.  Hence the $\phi$-conjugacy
of $\hat{\tau}$ and $\tau$ gives the desired equivalence between
orthogonality and incidence.
\end{proof}

We now analyze the Kapovich-Leeb-Porti construction in terms of the
\emph{divisor model} of $\F$ given by the Lemma.  First we note that
the model thickening in this case is the union of the complex
$1$-dimensional Schubert varieties,
$\Phi = X_{(2,1,3)} \cup X_{(1,3,2)}$, and it is easily checked that
$X_{(2,1,3)} = \{ (E_1,H) \suchthat H \in (\CP^2)^* \}$ while
$X_{(1,3,2)} = \{ (\ell,E_2) \suchthat \ell \in \CP^2 \}$.  The
corresponding description of $\Lambda$ is that it consists of flags
$\{ (\ell,H)\}$ in which either $\ell \in \phi(\RP^1)$ or $H$ is
tangent to $Y$ along $\phi(\RP^1)$.  In terms of divisors, then,
$\Lambda$ consists of pairs $(\xi,\eta)$ where either $\xi = 2p$ or
$\eta = 2p$, for $p \in \RP^1$.

Let $\H_+, \H_-$ denote the connected components of
$\CP^1 \setminus \RP^1$, and $X_\pm$ the compact Riemann surfaces that
are the quotients of $\tilde{\phi}(\H_\pm) \subset Y$ by the $\rho$-action of
$\pi_{1}S$.  Considering each of the cases from Lemma
\ref{lem:divisor-flag-model}, we find that
$\Omega = \F \setminus \Lambda$ can be described in the divisor model
as $\Omega_0 \cup \tilde{E}_+ \cup\tilde{E}_+^* \cup \tilde{E}_- \cup \tilde{E}_-^*$ where
\begin{itemize}
\item $\Omega_0 = \{ (p+q,r+s) \suchthat p \neq q, \: r \neq s\} \cap \mathcal{F}$,
\item $\tilde{E}_\pm = \{ (2p,p+q) \suchthat p \in \H_\pm \}$, and
\item $\tilde{E}_\pm^* = \{ (p+q,2p) \suchthat p \in \H_\pm \}$.
\end{itemize}
Note that these sets are pairwise disjoint except for
\[ \tilde{E}_\pm \cap \tilde{E}_\pm^* = \{ (2p,2p) \suchthat p \in \H_{\pm} \} = \tilde{\phi}(\H_\pm).\]

Now we arrive at the desired hyperbolic-geometric description of
$\W$.  Let
$\rho_0 : \pi_{1}S \to \PSL_2\R < \PSL_2\C \simeq \Isom^+(\H^3)$
be the Fuchsian representation through which $\rho$ factors, or equivalently,
so that $\tilde{\phi} : \CP^1 \to X$ intertwines $\rho_0$ acting on
$\CP^1$ with $\rho$ acting on $\F$.  Let $N_0$ denote the oriented
orthonormal frame bundle of the quotient $\rho_0(\pi_{1}S) \backslash \H^3$ and
define $N = N_0 / (\Z/2 \times \Z/2)$ where
$(i,j) \in \Z/2 \times \Z/2$ acts on an orthonormal frame
$(v_1,v_2,v_3) \in T_x \H^3$ by
\[ (v_1, v_2, v_3) \mapsto ((-1)^i v_1, (-1)^j v_2, (-1)^{i+j}v_3).\]
Since $\rho_0(\pi_{1}S) \backslash \H^3   \simeq S \times \R$, we have $N_0 \simeq
S \times \R \times \SO(3)$ and $N \simeq S \times \R \times B$ where
$B = \SO(3) / (\Z/2 \times \Z/2)$.

\begin{thm}
\label{thm:sl3-compactification}
The quotient $\rho(\pi_{1}S) \backslash \Omega_0$ is diffeomorphic to $N$, and hence $\W
=  \rho(\pi_{1}S) \backslash \Omega$ is a compactification of $N$.  The boundary
of this compactification is the union of the four complex surfaces
\[ E_\pm := \rho(\pi_{1}S) \backslash \tilde{E_\pm} \text{ and } E_\pm^* :=
\rho(\pi_{1}S) \backslash \tilde{E_\pm^*} ,\]
each of which is biholomorphic to a $\CP^1$ bundle over $X_+$ or
$X_-$, and which intersect only in the complex curves $E_+ \cap E_+^* = X_+$ and $E_-
\cap E_-^* = X_-$.
\end{thm}

\begin{proof}
Using the divisor model, map $(p+q,r+s) \in \Omega_0$ to the
positively oriented orthonormal frame $(v_1,v_2,v_3)$ at
$\gamma_{p,q} \cap \gamma_{r,s} \in \H^3$ such that $v_1$ is a unit
vector along $\gamma_{p,q}$ and $v_2$ is a unit vector along
$\gamma_{r,s}$.  While there are two choices for each of $v_1$ and
$v_2$, the result is a well-defined point in the quotient of the frame
bundle of $\H^3$ by $\Z/2 \times \Z/2$.  This map is easily seen to be
a $\PSL_2\C$-equivariant, and both spaces have transitive, smooth
$\PSL_2\C$ actions with the same isotropy, so it is a
diffeomorphism. By equivariance it descends to the desired map
$ \rho(\pi_{1}S)  \backslash \Omega_0 \to N$.

Lemma~\ref{lem:divisor-flag-model} describes $\Omega_0$ as an open,
dense, and $\rho$-invariant subset of the cocompact domain of
discontinuity $\Omega$, hence $\W$ is a compactification of
$ \rho(\pi_{1}S) \backslash \Omega_0$.  It remains to verify the given
descriptions of the quotients of $\tilde{E}_\pm$.  We have already
seen that $\tilde{E}_\pm \cap \tilde{E}_\pm^* = \tilde{\phi}(H_\pm)$
which has quotient $X_\pm$.  To see that $E_{+}$ is a $\CP^1$ bundle
over $X_+$, note first that $\tilde{E}_+ \simeq \H_+ \times
\CP^1$ by the map $(2p,p+q) \mapsto (p,q)$.  Thus $\tilde{E}_+$ is a
trivial $\CP^1$ bundle over $\H_+$, and the projection $(2p,p+q)
\mapsto p$ intertwines the $\rho$-action on $\tilde{E}_+$ with the
$\rho_0$-action on $\H_+$, and $\rho$ acts on $\tilde{E}_+$ by a
discontinuous group of bundle automorphisms.  The quotient $E_+$ is
therefore a locally trivial $\CP^1$ bundle over $\rho_0(\pi_{1}S)
\backslash \H_+ \simeq \rho(\pi_{1}S) \backslash \tilde{\phi}(\H_+) = X_+$.  The cases $E_-$ and
$E_{\pm}^*$ are handled similarly.
\end{proof}

The decomposition of $\W$ described above is pictured schematically in
Figure \ref{fig:stratification}.

\begin{figure}
\begin{center}
\begingroup%
  \makeatletter%
  \providecommand\color[2][]{%
    \errmessage{(Inkscape) Color is used for the text in Inkscape, but the package 'color.sty' is not loaded}%
    \renewcommand\color[2][]{}%
  }%
  \providecommand\transparent[1]{%
    \errmessage{(Inkscape) Transparency is used (non-zero) for the text in Inkscape, but the package 'transparent.sty' is not loaded}%
    \renewcommand\transparent[1]{}%
  }%
  \providecommand\rotatebox[2]{#2}%
  \ifx\svgwidth\undefined%
    \setlength{\unitlength}{320.14453317bp}%
    \ifx\svgscale\undefined%
      \relax%
    \else%
      \setlength{\unitlength}{\unitlength * \real{\svgscale}}%
    \fi%
  \else%
    \setlength{\unitlength}{\svgwidth}%
  \fi%
  \global\let\svgwidth\undefined%
  \global\let\svgscale\undefined%
  \makeatother%
  \begin{picture}(1,0.46505771)%
    \put(0,0){\includegraphics[width=\unitlength,page=1]{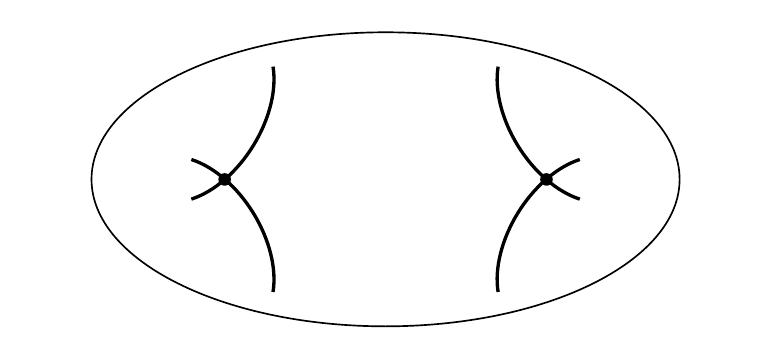}}%
    \put(0.290,0.330){\color[rgb]{0,0,0}\makebox(0,0)[lb]{$E_+$}}%
    \put(0.290,0.105){\color[rgb]{0,0,0}\makebox(0,0)[lb]{$E_+^*$}}%
    \put(0.135,0.215){\color[rgb]{0,0,0}\makebox(0,0)[lb]{$X_+$}}%
    \put(0,0){\includegraphics[width=\unitlength,page=2]{stratification.pdf}}%
    \put(0.820,0.215){\color[rgb]{0,0,0}\makebox(0,0)[lb]{$X_-$}}%
    \put(0.660,0.330){\color[rgb]{0,0,0}\makebox(0,0)[lb]{$E_-$}}%
    \put(0.660,0.104){\color[rgb]{0,0,0}\makebox(0,0)[lb]{$E_-^*$}}%
    \put(0.49,0.220){\color[rgb]{0,0,0}\makebox(0,0)[lb]{$N$}}%
  \end{picture}%
\endgroup%
\caption{Stratification of the $\PSL_3\C$ quotient manifold
  $\W$ consisting of the open stratum $N$, the $\CP^1$-bundles $E_\pm$,
  $E_\pm^*$, and the Riemann surfaces $X_\pm$.}
\label{fig:stratification}
\end{center}
\end{figure}

Since the oriented orthonormal frame bundle of $3$-dimensional
hyperbolic space is $\PSL_2\C$-equivariantly isomorphic to $\PSL_2\C$,
Theorem \ref{thm:sl3-compactification} equivalently describes $\W$ as a
compactification of the quotient
$\rho_0(\pi_{1}S) \backslash \PSL_2\C /(\Z/2 \times \Z/2)$.

Finally, we will show that the divisor model and hyperbolic picture of
$\W$ lead to a verification of Conjecture \ref{conjecture fiber} (on
the existence of a fiber bundle structure) in this case.  Such a fiber
bundle structure is easy to construct for the open, dense set
$N \subset \W$: There is a map from the frame bundle of $\H^3$ to
$\H^2$ which composes the projection of the frame bundle to its base
with the orthogonal projection from $\H^3$ to the totally geodesic
$\H^2$ preserved by $\PSL_2\R$.  This map is
$(\Z/2 \times \Z/2)$-invariant and $\PSL_2\R$-equivariant; taking the
quotient by $\Z/2 \times \Z/2$ and using the identification of Theorem
\ref{thm:sl3-compactification} we obtain an induced
$\PSL_2\R$-equivariant map
\[\tilde{\pi} : \Omega_0 \to \H^2.\]
Taking a further quotient by $\rho_0(\pi_{1}S)$, a map
$\pi : N \to S \simeq (\rho_0(\pi_{1}S) \backslash \H^2)$ is obtained.  The
identification of $N$ with a product, $N \simeq S \times \R \times B$,
can be made in such a way that the map $\pi$ is simply projection onto the first
factor.

To show that $\W$ is also a fiber bundle, we extend $\tilde{\pi}$ and
$\pi$ to $\Omega$ and $\W$, respectively:
\begin{thm}
\label{thm:sl3-fibering}
The map $\tilde{\pi} : \Omega_0 \to \H^2$ extends to a proper
$\PSL_2\R$-equivariant continuous map $\hat{\pi} : \Omega \to \H^2$.  Therefore,
\begin{rmenumerate}
\item $\Omega$ has the structure of $\PSL_2\R$-equivariant continuous
fiber bundle over $\H^2$ with fiber a compact topological space $F$,
\item $\Omega$ is homeomorphic to $\H^2 \times F$, and
\item The quotient manifold $\W = \Gamma \backslash \Omega$ is a continuous fiber
bundle over $S$ with fiber $F$.
\end{rmenumerate}
\end{thm}

\begin{proof}
Statements (i)-(iii) are simple consequences of the existence of
such a map $\hat{\pi}$: Because $\H^2$ is a homogeneous space of $\PSL_2\R$, a
continuous equivariant map from a $\PSL_2\R$-space to $\H^2$ is
necessarily an equivariant locally trivial fibration.  The fiber is
compact by properness of $\hat{\pi}$, so (i) follows.  Since $\H^2$ is
contractible this bundle is trivial, giving (ii).  Finally, using the
equivariant structure of bundle $\hat{\pi} : \Omega \to \H^2$ we can
take the quotient by $\rho_0(\pi_{1}S)$ to obtain (iii).

Now we construct $\hat{\pi}$.  Let
$\Omega' = \Omega \setminus \Omega_0$, which is a closed set.  Since
we seek an extension of the map $\tilde{\pi}$, it suffices to define
$\hat{\pi}$ on the set $\Omega'$, which in the divisor model consists
of pairs of the form $(2p,p+q)$ or $(p+q,2p)$ with $p \not \in \RP^1$.
Let $\Pi : \CP^1 \setminus \RP^1 \to \H^2$ be the extension to the
ideal boundary of orthogonal projection $\H^3 \to \H^2$; equivalently
$\Pi$ is the union of the natural $\PSL_2\R$-equivariant
diffeomorphisms $\H_+ \to \H^2$ and $\H_- \to \H^2$. Define:
\begin{equation*}
\begin{split}
\hat{\pi}(2p,p+q) &= \Pi(p)\\
\hat{\pi}(p+q,2p) &= \Pi(p)
\end{split}
\end{equation*}
This is evidently a continuous and $\PSL_2\R$-equivariant map
$\Omega' \to \H^2$, since the map $\Pi$ itself has these properties
and the two definitions above agree on their common domain $\{(2p,2p)
\suchthat p \in \CP^1 \setminus \RP^1\}$.

It remains to show that $\hat{\pi}$ is continuous on the entire domain
$\Omega$, and that it is proper.  Both will follow by elementary geometric
arguments.

For continuity, since $\Omega'$ is closed it suffices to consider a
sequence $\omega_n \in \Omega_0$ converging to
$\omega_\infty \in \Omega'$ and to show
$\tilde{\pi}(\omega_n) \to \hat{\pi}(\omega_\infty)$.  We suppose the
limit point has the form $\omega_\infty = (2p, p+q)$ with
$p \in \H_+$, the argument in the other cases being completely
analogous.  Since $\omega_{n}\in \Omega_{0},$ we can 
write $\omega_n = (p_n + p_n', p_n'' + q_n)$
with each of the sequences $\{p_n\}, \{p_n'\}, \{p_n''\}$ converging
to $p$, and with $q_n \to q$.  Recalling the construction of
$\tilde{\pi}$ and the map from the frame bundle to $\Omega_{0}$ from the
proof of Theorem \ref{thm:sl3-compactification}, we see that
$\tilde{\pi}(\omega_n)$ is the orthogonal projection to $\H^2$ of the
point $\gamma_{p_n,p_n'} \cap \gamma_{p_n'',q_n} \in \H^3$

Consider the disk $D \subset \H_+$ of radius $\epsilon$ centered at
$p$ with respect to the Poincar\'e metric of $\H_+$.  The orthogonal
projection to $\H^2$ of any geodesic in $\H^3$ with ideal endpoints in
$D$ is contained in the $\epsilon$-disk centered at
$\Pi(p) = \hat{\pi}(\omega_\infty)$.  For large enough $n$ we have
$p_n,p_n',p_n'' \in D$, and $\tilde{\pi}(\omega_n)$ is the projection
to $\H^2$ of a point on $\gamma_{p_n,p_n'}$, hence
$d_{\H^{2}}(\tilde{\pi}(\omega_n), \hat{\pi}(\omega_\infty)) < \epsilon$.  Thus
$\tilde{\pi}(\omega_n) \to \hat{\pi}(\omega_\infty)$ as $n \to \infty$, and
$\hat{\pi}$ is continuous.

To see that $\hat{\pi}$ is proper, we consider a compact exhaustion of
$\Omega$ constructed by taking complements of small open neighborhoods
of $\Lambda$.  Recall $\Lambda$ consists of divisor pairs of the form
$(2p,p+q)$ or $(p+q,2p)$ where $p$ lies on $\RP^1$. Fix an auxiliary
metric on $\CP^1$ and define $N_\eps(\Lambda)$ to consist of divisor
pairs $(p+q,r+s)$ in which there is a disk of radius $\epsilon$ in
$\CP^1$ with center in $\RP^1$ which contains at least three of the
points $p,q,r,s$.

Fix a basepoint $x_0$ in $\H^2$ (which we could take to be the origin
in the unit ball model of $\H^3$). Then for each $R > 0$ there exists
$\eps = \eps(R) > 0$ such that if $y \in \H^3$ lies in the hyperbolic
convex hull of a disk on $\CP^1$ of radius $\eps$, then
$d_{\H^3}(x_0,y) > R$.  That is, a half space in $\H^3$ bounded by a
sufficiently small circle is far from $x_0$.

We claim that if $\omega \in N_\eps(\Lambda) \cap \Omega$, then
$\hat{\pi}(\omega)$ lies in such a half-space, and thus is far from
$x_0$ for $\eps$ small enough.  To see this, first consider
$\omega \in N_\eps(\Lambda)\cap \Omega_{0}$ which we can 
write as $\omega = (p+q,r+s)$ with
$p,q,r,s$ distinct, and so that $p,q,r$ lie in an $\eps$-disk $D$
which is centered on $\RP^1$.  Let $B$ be the half-space in $\H^3$
with ideal boundary $D$; note $B$ is invariant by reflection in $\H^2$
and $D$ is invariant by inversion in $\RP^1$.  Then
$\hat{\pi}(\omega) = \tilde{\pi}(\omega)$ is the orthogonal projection
to $\H^2$ of a point $x \in \gamma_{p,q} \subset \H^3$.  Since both
$x$ and its reflection $\bar{x}$ in $\H^2$ lie in $B$, so does the
segment joining them.  The intersection of this segment with $\H^2$ is
the orthogonal projection of $x$ to $\H^{2}$, which is $\hat{\pi}(\omega)$, so
$\hat{\pi}(\omega) \in B$.

The remaining case is that $\omega \in \Omega'$, in which case
we can write $\omega = (2p,p+q)$ or $\omega = (p+q,2p)$, with $p$ in
an $\eps$-disk $D$ of the type considered above.  Then
$\hat{\pi}(\omega) = \Pi(p)$, and $\Pi(p) \in B$ because it lies on
the geodesic $\gamma_{p,\bar{p}}$, where $\bar{p}$ is the inversion of
$p$ in $\RP^1$, and $p,\bar{p} \in D$.

Now if $\omega_n \in \Omega$ satisfies $\omega_n \to \infty$, then for
each $R>0$ we have for all sufficiently large $n$ that $\omega_n \in N_{\eps(R)}(\Lambda)$.
The argument above shows $d_{\H^2}(x_0,\hat{\pi}(\omega_n)) > R$ for
such $n$.  Thus $\hat{\pi}(\omega_n) \to \infty$ in $\H^2$, and
$\hat{\pi}$ is proper.
\end{proof}

\nocite{}

\vspace{1.5em}

\noindent Department of Mathematics, Statistics, and Computer Science\\
University of Illinois at Chicago\\
\texttt{david@dumas.io}\\

\medskip

\noindent Mathematisches Institut\\
Ruprecht-Karls-Universit\"{a}t Heidelberg\\
\texttt{asanders@mathi.uni-heidelberg.de}\\

\end{document}